\documentclass{amsart}
\usepackage[british,UKenglish,USenglish,english,american]{babel}
\usepackage{newlfont}
\usepackage{amssymb}
\usepackage{amsmath}
\usepackage{amsfonts}
\usepackage{latexsym}
\usepackage{amsthm}
\usepackage{mathrsfs}
\usepackage[all,cmtip]{xy}
\usepackage{caption}
\usepackage{enumerate}
\usepackage{hyperref}
\usepackage{stmaryrd}
\usepackage{tabu}
\usepackage{longtable,booktabs}
\calclayout

\usepackage[usenames,dvipsnames]{xcolor}
\hypersetup{
	colorlinks,
	citecolor=Green,
	linkcolor=blue,
	urlcolor=Blue}

\theoremstyle{plain}                    
\newtheorem{teo}{Theorem}[section]     
\newtheorem{theoremalpha}{Theorem}

\newtheorem{coralpha}[theoremalpha]{Corollary}
\newtheorem{prop}[teo]{Proposition}
\newtheorem{cor}[teo]{Corollary}       
\newtheorem{lem}[teo]{Lemma}            
\theoremstyle{definition}               
\newtheorem{notations}{Notations}[]
\newtheorem{defin}[teo]{Definition}
\theoremstyle{remark}
\newtheorem{rmk}[teo]{Remark}

\numberwithin{equation}{section}

\newcommand{\oo}{\mathcal{O}}
\newcommand{\mt}{\mathcal}

\newcommand{\BG}{\mathbf{Bun}}

\newcommand{\thig}{\mathcal{T}^*\mathcal{M}}
\newcommand{\thigd}{\mathcal{T}^*\mathcal{M}^{\delta}}
\newcommand{\hit}{\mathbf{Higgs}}
\newcommand{\jator}[1]{\mathbf{Bun}_{J_{#1}}}
\newcommand{\japtor}[1]{\operatorname{Gr}_{J_{#1}}}
\newcommand{\spr}[1]{\operatorname{Spr}_{#1}(a)}
\newcommand{\spreg}[1]{\operatorname{Spr}^{\reg}_{#1}(a)}

\newcommand{\hitgd}{\mathcal{H}}
\newcommand{\nc}{\mathcal{N}}
\newcommand{\M}{\mathcal{M}}
\newcommand{\Md}{\mathcal{M}^{\delta}}
\newcommand{\ad}[1]{\operatorname*{ad}(#1)\otimes\omega}

\newcommand{\g}{\mathfrak g}
\newcommand{\tg}{\mathfrak t}
\newcommand{\cg}{\mathfrak c}
\newcommand{\Dg}{\mathfrak D}
\newcommand{\AG}{\mathbb A}
\newcommand{\DG}{\mathbb D}

\newcommand{\ab}{\operatorname*{ab}}
\newcommand{\sm}{\operatorname*{sm}}
\newcommand{\sing}{\operatorname*{sing}}
\newcommand{\red}{\operatorname*{red}}
\newcommand{\reg}{\operatorname*{reg}}
\newcommand{\Spec}{\operatorname{Spec}}
\newcommand{\ads}{\operatorname*{ad}}
\newcommand{\Id}{\operatorname{Id}}
\newcommand{\Out}{\operatorname{Out}}
\newcommand{\Aut}{\operatorname{Aut}}

\newcommand{\smt}{\overline{\mathcal M}}
\newcommand{\op}[1]{\operatorname{#1}}

\newcommand{\scr}{\mathscr}
\newcommand{\SO}{\operatorname{SO}}
\newcommand{\Spin}{\operatorname{Spin}}
\newcommand{\PSO}{\operatorname{PSO}}
\newcommand{\bbZ}{\mathbb{Z}}

\newcommand{\disk}[1]{k\llbracket #1\rrbracket}

\newcommand{\Pic}{\operatorname{Pic}}

\begin{document}
\title{Automorphisms of moduli spaces of principal bundles over a smooth curve}
	\author{Roberto Fringuelli}
	\address{Dipartimento di Matematica, Universit\`a di Roma ``La Sapienza'',\\ Piazzale Aldo Moro 5, I-00185, Roma, Italy,\\E-mail: r.fringuelli@uniroma1.it}
	\maketitle
	\begin{abstract}For any almost-simple group $G$ over an algebraically closed field $k$ of characteristic zero, we describe the automorphism group of the moduli space of semistable $G$-bundles over a connected smooth projective curve $C$ of genus at least $4$. The result is achieved by studying the singular fibers of the Hitchin fibration. As a byproduct, we provide a description of the irreducible components of two natural closed subsets in the Hitchin basis: the divisor of singular cameral curves and the divisor of singular Hitchin fibers.
	\end{abstract}

\section*{Introduction}
Let $k$ be an algebraically closed field of characteristic zero. The moduli space $\overline{\mt M}_G(C)$ of semistable $G$-bundles over a connected smooth projective $k$-curve $C$ of genus $g\geq 2$, where $G$ is a connected reductive $k$-group, is a proper $k$-scheme. These moduli spaces have been studied extensively, due to their connections with non-abelian Hodge theory, conformal field theory, geometric Langlands program and, more recently, mirror symmetry.

The automorphisms of these moduli spaces are interesting objects, especially those ones arising from the geometry of the curve $C$ and the group $G$. For instance, some of these automorphisms are constructed starting from automorphisms of the group $G$ itself and their loci of fixed points may be described as subvarieties in $\overline{\mt M}_G(C)$ of $G$-bundles admitting reductions to suitable subgroups of $G$, see \cite{SancSpin},  \cite{SancOut}, \cite{OrRa}, \cite{OrBa}, \cite{ScSc}. Furthermore, the group of the automorphisms of these moduli spaces has been already described for some groups, see \cite{Sanc}, \cite{BGMsl}, \cite{BGMsym}, \cite{HR}, \cite{KP95},
The aim of this paper is to provide a complete description of the automorphism group of the connected components of the moduli space $\overline{\mt M}_G(C)$, where $G$ is almost-simple and $C$ has genus at least $4$. 

Before presenting the main theorem, we fix some notations. There are three natural ways to produce automorphisms for the moduli space $\overline{\mt M}_G(C)$ of semistable $G$-bundles.
\begin{enumerate}[(a)]
	\item Pull-back along an automorphism $\sigma:C\to C$ of the curve, i.e.
	$
		E\mapsto \sigma^*E.
$
	\item Action of an outer automorphism $\rho\in \Out(G)$ of the group, i.e.
$
		E\mapsto \rho( E):=E\times^{\rho,G}G,
$
	where $E\times^{\rho,G}G$ denotes the product $E\times G$ modulo the equivalence relation $(eh,g)\sim (e,\rho(h)g)$, for any $e\in E$ and $g,h\in G$.
	\item Twisting by a $\scr Z(G)$-bundle $Z$ over $C$ (with $\mathscr Z(G)\subset G$ center), i.e.
$
E\mapsto	E\otimes Z,
$
	where $E\otimes Z$ denotes the fibered product $E\times_C Z$ modulo the equivalence relation $(et,z)\sim (e,zt)$, for any $e\in E$, $z\in Z$ and $t\in\scr Z(G)$.
\end{enumerate}
For more details, see Section \ref{Sec:taut-aut}. The connected components $\overline{\mt M}^{\delta}_G(C)$ of the moduli space $\overline{\mt M}_G(C)$ are in bijection with the elements of the fundamental group $\pi_1(G)$. In general, the automorphisms described in (a), (b) and (c) may move these components. We are only interested in those automorphisms, which fix a given component. When $G$ is semisimple, the automorphisms of type (a) and (c) fix the connected components. In particular, they descend to automorphisms of the irreducible variety $\overline{\mt M}_G^{\delta}(C)$, for any $\delta\in\pi_1(G)$. On the other hand, the automorphisms of $\overline{\mt M}_G(C)$ of type (b), preserving the connected component $\overline{\mt M}_G^{\delta}(C)$, are given by elements in the subgroup of elements in $\Out(G)$ fixing $\delta\in\pi_1(G)$, i.e. $$\Out(G,\delta):=\{\rho\in\Out(G)\vert\pi_1(\rho)(\delta)=\delta\}\subset\Out(G).$$
The main result of this paper is the following.
\begin{theoremalpha}[= Proposition \ref{P:aut-m} and Theorem \ref{T:T}]\label{T:A}Let $C$ be a connected smooth projective $k$-curve of genus $g\geq 4$, $G$ be an almost-simple $k$-group and $\delta\in\pi_1(G)$. The morphism
	$$
	\def\arraystretch{1.5}\begin{array}{ccl}
		H^1(C,\mathscr Z(G))\rtimes \Big(\operatorname{Out}(G,\delta)\times\operatorname{Aut}(C)\Big)&\xrightarrow{\sim}&\operatorname{Aut}\left(\overline{\mt M}_G^{\delta}(C)\right)\\
		(Z,\rho,\sigma)&\mapsto& \Big(E\mapsto \rho((\sigma^{-1})^*E)\otimes Z\Big)
	\end{array}$$
is an isomorphism of groups, where:
\begin{enumerate}[(i)]
	\item $\Aut(C)$ acts on $H^1(C,\mathscr Z(G))$ by pull-back, i.e. $\sigma.Z=(\sigma^{-1})^*Z$, for any $\sigma\in\Aut(C)$,
	\item $\Out(G,\delta)$  acts on $H^1(C,\mathscr Z(G))$ as follows: $\rho.Z=Z\times^{\rho,\mathscr Z(G)}\mathscr Z(G)$, for any $\rho\in\Out(G)$.
\end{enumerate}
The same holds for the open locus $\mt M^{\delta}_G(C)$ of stable and simple $G$-bundles.
\end{theoremalpha}
 In particular, the automorphism group of the moduli space $\overline{\mt M}_G^{\delta}(C)$ is generated by the transformations of type (a), (b) and (c). 
 
 Using the classification of almost-simple groups, we are able to give a more explicit description of these groups. For any $l\in\mathbb N$, we denote by $\Pic(C)[l]$ the group of $l$-torsion line bundles over $C$. Note that $\Pic(C)[l]\cong (\bbZ/l\bbZ)^{2g}$ as groups. We then have the following corollary.

\begin{coralpha}[= Theorems \ref{T:An}, \ref{T:BnCn}, \ref{T:Dn}, \ref{T:E6} and \ref{T:resto}]\label{T:conti} Let $C$ be a connected smooth projective $k$-curve of genus $g\geq 4$. Then, the isomorphism of Theorem \ref{T:A} gives the following table.
{\small
\setlength\tabcolsep{1.5pt}
	\begin{longtable}{|>{$}l <{$}|>{$}l <{$}| >{$} l<{$}|>{$} l<{$}|}
\hline \multicolumn{4}{r}{\textit{Continued on next page}}
\endfoot
\multicolumn{4}{r}{}
\endlastfoot
\hline
	 \g	&	G	& \delta\in\pi_1(G) & \Aut(\overline{\mt M}^\delta_G(C))=\Aut(\mt M^\delta_G(C))\\\midrule
	 \endhead\hline
	 A_1&\op{SL}_2& \delta\in\{0\}& \Pic(C)[2]\rtimes \Aut(C)\\*
	 	 &\op{PSL}_2& \delta\in\bbZ/2\bbZ& \Aut(C)\\*
	 	 \hline
		A_{n-1},\;\scriptstyle{n\geq 3} & \op{SL}_n/\mu_r,\; \scriptstyle{\text{s.t. } r|n}& 2\delta=0\in\mathbb Z/r\bbZ & \Pic(C)[n/r]\rtimes \left(\mathbb Z/2\times\Aut(C)\right)\\*
		&\op{SL}_n/\mu_r,\; \scriptstyle{\text{s.t. } r|n}	& 2\delta\neq 0\in\mathbb Z/r\bbZ &\Pic(C)[n/r]\rtimes \Aut(C)\\
		\hline
		B_{n},\;\scriptstyle{n\geq 2} & \Spin_{2n+1} &\delta\in\{0\} &\Pic(C)[2]\rtimes\Aut(C)\\*
		& \SO_{2n+1} &\delta\in\bbZ/2\bbZ &\Aut(C)\\
		\hline
		C_{n},\;\scriptstyle{n\geq 3} & \op{Sp}_{2n} &\delta\in\{0\} &\Pic(C)[2]\rtimes\Aut(C)\\*
		& \op{PSp}_{2n} &\delta\in\bbZ/2\bbZ &\Aut(C)\\
		\hline
		D_4	&	\op{Spin}_8 & \delta\in\{0\} &(\Pic(C)[2])^2\rtimes (S_3\times\Aut(C))\\*
		&	\op{SO}_8 & \delta\in\bbZ/2\bbZ &\Pic(C)[2]\rtimes (\bbZ/2\bbZ\times\Aut(C))\\*
		&	\op{PSO}_8 & \delta=(0,0)\in(\bbZ/2\bbZ)^2 &S_3\times\Aut(C)\\*
		&	\op{PSO}_8 & \delta\neq(0,0)\in(\bbZ/2\bbZ)^2 &\bbZ/2\bbZ\times\Aut(C)\\
		\hline
		D_{2n},\;\scriptstyle{n\geq 3} &\op{Spin}_{4n}&\delta\in\{0\}&(\Pic(C)[2])^2\rtimes (\bbZ/2\bbZ\times \Aut(C))\\*
		&\op{SemiSpin}_{4n}&\delta\in\mathbb Z/2\mathbb Z&\Pic(C)[2]\rtimes \Aut(C)\\*
		&\op{SO}_{4n}&\delta\in\mathbb Z/2\mathbb Z&\Pic(C)[2]\rtimes (\bbZ/2\bbZ\times\Aut(C))\\*
		&\op{PSO}_{4n}&\delta=(0,0),(1,1)\in(\mathbb Z/2\mathbb Z)^2&\mathbb Z/2\mathbb Z\times \Aut(C)\\*
		&\op{PSO}_{4n}&\delta= (1,0),(0,1)\in(\mathbb Z/2\mathbb Z)^2&\Aut(C)\\
		\hline
		D_{2n+1},\;\scriptstyle{n\geq 2} &\op{Spin}_{4n+2}&\delta\in\{0\}&\Pic(C)[4]\rtimes (\mathbb Z/2\mathbb Z\times \Aut(C))\\*
		&\op{SO}_{4n+2}&\delta\in\mathbb Z/2\mathbb Z&\Pic(C)[2]\rtimes (\bbZ/2\bbZ\times\Aut(C))\\*
		&\op{PSO}_{4n+2}&\delta=0,2\in\mathbb Z/4\mathbb Z&\mathbb Z/2\mathbb Z\times \Aut(C)\\*
		&\op{PSO}_{4n+2}&\delta=1,3\in\mathbb Z/4\mathbb Z&\Aut(C)\\
		\hline
		E_{6} &\mathbb{E}_6^{\operatorname{sc}}&\delta\in\{0\}&\Pic(C)[3]\rtimes\left(\mathbb Z/2\mathbb Z\times \Aut(C)\right)\\
		&\mathbb{E}_6^{\operatorname{ad}}&\delta=0\in\mathbb Z/3\mathbb Z&\mathbb Z/2\mathbb Z\times \Aut(C)\\
		&\mathbb{E}_6^{\operatorname{ad}}&\delta\neq 0\in\mathbb Z/3\mathbb Z&\Aut(C)\\
		\hline
		E_7& \mathbb{E}_7^{\operatorname{sc}}&\delta\in\{0\}&\Pic(C)[2]\rtimes\Aut(C)\\
		&\mathbb{E}_7^{\operatorname{ad}}&\delta\in\mathbb Z/2\mathbb Z&\Aut(C)\\
		\hline
		E_8&\mathbb{E}_8&\delta\in\{0\}&\Aut(C)\\
		\hline 
		F_4&\mathbb{F}_4&\delta\in\{0\}&\Aut(C)\\
		\hline 
		G_2&\mathbb{G}_2&\delta\in\{0\}&\Aut(C)\\
		\hline 
	\end{longtable}}
\end{coralpha}
We also prove a Torelly-type theorem for these moduli spaces.

\begin{theoremalpha}[= Theorem \ref{T:torelli} and Corollary \ref{C:torelli}]\label{T:B}Let $C_1$ and $C_2$ be two connected smooth projective $k$-curves of genus $g\geq 2$ and $G$ be an almost-simple $k$-group, with the only exception of the cases where $g=2$ and $G$ is either $\op{SL_2}$ or $\op{PGL}_2$, which are excluded.
	
	If $\overline{\mt M}_G^{\delta_1}(C_1)$ is isomorphic to $\overline{\mt M}_G^{\delta_2}(C_2)$, for some $\delta_1,\delta_2\in\pi_1(G)$, then $C_1$ is isomorphic to $C_2$. The same holds for the locus of stable and simple bundles.
\end{theoremalpha}

We remark that Theorem \ref{T:B} has been already proved in greater generality by Biswas and Hoffmann \cite{BH12}, with different methods. In particular, they showed that the same statement holds if $G$ is a non-abelian reductive group. Furthermore, Theorem \ref{T:A} and Corollary \ref{T:conti} have been already proved in some cases. When $G=\op{SL}_n$, the computation has been carried out by Kouvidakis and Pantev \cite{KP95}. An alternative proof has been provided by Hwang and Ramanan \cite{HR}. Later, Biswas, G\'omez and Mu\~noz, after simplifying the proof for $G=\op{SL}_n$ in \cite{BGMsl}, extended the result to the symplectic group $\op{Sp}_{2n}$ \cite{BGMsym}. The cases $G$ simply-connected of type $E_6$ and $F_4$ have been studied by Ant\'on-Sancho \cite{Sanc}, by adapting the Biswas-G\'omez-Mu\~noz strategy. All the proofs rely on the study of the singular fibers of the Hitchin fibration. 
\vspace{0.1cm}

Our contribution consists in providing a unified version of Biswas-G\'omez-Mu\~noz proof, which holds for any almost-simple group $G$. The main difference between our paper and the other ones consists in the study of the Hitchin fibration. Instead of using the theory of spectral covers, we will use the theory of cameral covers. 
\vspace{0.1cm}

The main advantage of the cameral covers is that the picture is more uniform with respect to the spectral covers, as the group varies. Just for having an idea: it is well-known that, when $G=\op{GL}_n,\op{SL}_n$, a spectral curve is smooth if and only if the corresponding Hitchin fiber is an abelian variety. For other groups, this statement is false. For example, when $G=\op{SO}_n$, all the spectral curves are singular but the generic Hitchin fiber is an abelian variety (see \cite[\S 5.15, 5.17]{Hi}). Hence, the study of the relation between spectral curves and Hitchin fibers requires a case-by-case analysis. On the other hand, the generic cameral cover is smooth and the corresponding Hitchin fiber is smooth, for any reductive group $G$. However, unless $G$ has semi-simple rank $1$ (i.e. the derived subgroup $\mathscr D(G)\subset G$ is equal to either $\op{SL_2}$ or $\op{PGL_2}$), there are singular cameral covers, whose corresponding Hitchin fiber is an abelian variety. Despite this complication, the theory of cameral covers allows us to provide a uniform description of the locus in the Hitchin basis corresponding to abelian Hitchin fibers, for any reductive group $G$, avoiding the case-by-case analysis. For more details, see Section \ref{Sec:hit-disc} and references therein.

\vspace{0.1cm}

We now briefly explain the proof of Theorem \ref{T:A}. We only treat the surjectivity, which is the most delicate part. We start with an isomorphism of the moduli space $$\varphi:\overline{\mt M}_G^{\delta}(C)\xrightarrow{\sim}\overline{\mt M}_G^{\delta}(C).$$ 
The codifferential $d\varphi^{*}$ gives a $\mathbb G_m$-equivariant birational morphism of the moduli space $\mt H^{\delta}_G(C)$ of semistable $G$-Higgs bundles, where $\mathbb G_m$ acts on $\mt H^{\delta}_G(C)$ by scaling the Higgs field. Furthermore, the morphism $d\varphi^*$ sits in a commutative diagram of $\mathbb G_m$-equivariant morphisms of the following shape
\begin{equation*}\label{E:iso-cart}
	\xymatrix{
		\mt H^{\delta}_G(C)\ar[d]^h\ar@{-->}[r]^{d\varphi^*}&	\mt H^{\delta}_G(C)\ar[d]^h\\
		\AG\ar[r]^{\varphi_{\AG}}&\AG
	}
\end{equation*}
where $h$ is the Hitchin fibration and $\mathbb G_m$ acts on the Hitchin basis $\mathbb A=\oplus_{i=1}^{\operatorname{rank}(G)}\AG_i$ with weight $d_i$ on the subspace $\AG_i:=H^0(C,\omega^{d_i})$ (in the paper $\omega$ is the cotangent bundle of the curve).
In particular, the isomorphism $\varphi_{\AG}$ preserves:
\begin{enumerate}[(i)]
	\item the eigenspace $\AG_{\max}$ in $\AG$ of highest weight (i.e. $\max$ is the integer such that $d_{\max}$ is the highest number among $d_1,\ldots,d_{\op{rank}(G)}$),
	\item the closed subset $\DG^{\ab}$ of $\AG$, where the fibers of the Hitchin fibration $h$ are not abelian varieties.
\end{enumerate}The projectivization $\mathbb P(\DG^{\ab}\cap \AG_{\max})$ of the intersection between these two loci is equal to the dual variety of $C$ with respect to the linear system $\AG_{\max}$. Since the curve is equal to its double dual, the isomorphism $\varphi_{\AG}$, restricted to the intersection $\DG^{\ab}\cap \AG_{\max}$, produces an isomorphism $\sigma$ of the curve $C$. Using this fact, after some work, we are able to give a more explicit description of $\varphi_{\AG}$. In particular, we get:
$$
\varphi_{\AG}\left(\bigoplus_{i=1}^rH^0(C,\omega^{d_i}_C(-\sigma(p))\right)=\bigoplus_{i=1}^rH^0(C,\omega^{d_i}_C(-p)),\text{ for any }p\in C.
$$
From this, we are able to show that for a generic $G$-bundle $E$ with image $\varphi(E)$, there exists a canonical isomorphism of adjoint bundles $$\ads(E)\xrightarrow[\Psi]{\sim}\ads(\varphi(E))\Longrightarrow \mt N_E\xrightarrow[\Psi]{\sim}\mt N_{\varphi(E)},$$ which preserves the subsets $\mt N_E$ and $\mt N_{\varphi(E)}$ of nilpotent elements of $\ads(E)$ and $\ads(\varphi(E))$, respectively. Using the Springer resolution of the nilpotent cone in the Lie algebra $\g$, we then show that $\Psi$ induces an isomorphism of $G/B$-fibrations over $C$
$$
E/B\xrightarrow[\widehat \Psi]{\sim} \varphi(E)/B,\quad\text{ where }B \text{ Borel subgroup of }G.
$$
Thanks to a result of Demazure, there exists a canonical isomorphism between the relative tangent bundle of the $G/B$-fibration $E/B$ (resp. $\varphi(E)/B$) and the adjoint bundle $\ads(E)$ (resp. $\ads(\varphi(E))$) and it preserves the Lie-brackets. Hence, the (relative) differential 
$$
\ads(E)=\mt T_{(E/B)/C}\xrightarrow[d\widehat{\Psi}]{\sim}\mt T_{(\varphi(E)/B)/C}=\ads(\varphi(E))
$$
is an isomorphism of Lie algebra bundles. If this happens the $G$-bundles $E$ and $\varphi(E)$ differ by a combination of the transformations of type (a), (b) and (c). Since $E$ was chosen generic and it sits in the same connected component of $\overline{\mt M}_G(C)$ containing $\varphi(E)$, the isomorphism $\varphi$ must be in the image of the homomorphism described in Theorem \ref{T:A}.\\

\noindent\textbf{Further generalizations.} There are (at least) four ways of generalizing this problem. (1) What happens for a general reductive group? In this generality, the picture is more complicated. As an example, when $G=G_1\times G_2$ is the product of two almost-simple groups, the automorphism group of $\overline{\mt M}_G(C)$ should contain two copies of $\Aut(C)$. Note that case $G=\op{GL}_n$ has been computed in \cite{KP95}. (2) Another version of the problem would be to study the automorphism group for the so-called twisted moduli spaces of $G$-bundles (e.g. moduli spaces of vector bundles of fixed rank with determinant isomorphic to a fixed, non-trivial, line bundle). For instance, the twisted case has been treated in \cite{KP95} and \cite{BGMsl} for $G=\op{SL}_n$ and in \cite{BGMsym} for $G=\op{Sp}_{2n}$. (3) Another generalization is the case of parabolic (or, more generally, parahoric) $G$-bundles. When $G=\op{GL}_n,\op{SL}_n$, the problem has been studied in \cite{BHK}, \cite{CFK}, \cite{AG}. (4) An alternative path could be the study of the automorphisms of the moduli stack $\BG_G$ of all $G$-bundles. In this case, the automorphisms form a category (more precisely a $2$-group). We hope to come back on these problems in future works.\\

\noindent\textbf{Outline of the paper.} In Section \ref{S:LieAlg}, we collect some results about reductive groups and their adjoint representations on their own Lie algebras. In Section \ref{Sec:hit-fib}, we introduce the moduli stacks of (Higgs) bundles on curves and we recall their main properties, following the works of Faltings \cite{Fa93} and Ng\^o \cite{NGO10}. In section \ref{Sec:hit-disc}, we give a description of the closed subset $\DG^{\diamondsuit}$, resp. $\DG^{\ab}$, of the Hitchin basis $\AG$ parameterizing singular cameral curves, resp. non-abelian Hitchin fibers. Then, we derive some consequences on the Hitchin discriminant. In particular, we show that generic Hitchin fiber is either an abelian variety or birational to a $\mathbb P^1$-fibration over an abelian variety. In section \ref{Sec:torelli}, we prove the Torelli Theorem (= Theorem \ref{T:B}) for $G$-bundles. In Section \ref{Sec:taut-aut}, we show that the homomorphism of Theorem \ref{T:A} is well-defined and injective. In Section \ref{Sec:nilp}, we recall the necessary facts about the nilpotent cone of a Lie algebra and the tangent bundle of the flag variety. The surjectivity of the homomorphism of Theorem \ref{T:A} is proved in Section \ref{Sec:last}. In Section \ref{Sec:comp}, we compute the explicit presentations of the automorphism group $\Aut(\overline{\mt M}_G(C))$ contained in Corollary \ref{T:conti}.

\begin{notations}In the paper, $k$ will be always an algebraically closed field of characteristic zero. We usually denote by $G$ a (connected) reductive group over $k$, by $T_G$ a maximal torus of $G$ and by $B_G$ a Borel subgroup containing $T_G$. We denote by $\mathscr Z(G)$ the center of $G$, by $\mathscr D(G)$ the derived subgroup of $G$, and by $G^{\operatorname{ad}}$ the adjoint group $G/\mathscr Z(G)$. As usual, we denote by $W_G:=N(T_G)/T_G$ its Weyl group. The rank of $G$ (i.e. $\dim T_G$) will be denoted by the letter $r$.
\end{notations}

\begin{notations}We denote by $\g$, $\mathfrak b_G$ and $\tg_G$ the Lie algebras of $G$, $B_G$ and $T_G$, respectively. The set of roots (resp. coroots) will be denoted by $\Phi_G$ (resp. $\Phi^{\vee}_G$).  For any root $\alpha$, we denote by $\g_{\alpha}$ the 1-dimensional sublinear space of $\g$, where the torus $T_G$ acts with weight $\alpha$. For any root $\alpha:T_G\to\mathbb G_m$, we denote by $d\alpha:\tg_G\to k$ its Lie-algebra counterpart. We denote by $\g^{\reg}$ the subset of regular elements of $\g$, i.e. those elements $x\in\g$ such that the kernel of the adjoint homomorphism $[x,-]:\g\to\g$ has (minimal) dimension equal to $r$.
\end{notations}

\begin{notations}As usual $\pi_1(G):=\op{Hom}(\mathbb G_m,T_G)/\langle\Phi^{\vee}\rangle$ denote the fundamental group of $G$, which is defined as the quotient of the cocharacter lattice of $T_G$ by the sublattice $\langle\Phi^{\vee}\rangle$ generated by integral linear combinations of coroots.
\end{notations}

When it is clear from the context, we will remove the reference to the group $G$ from the notation (e.g. $T$ instead of $T_G$, $\tg$ instead of $\tg_G$, $\Phi$ instead of $\Phi_G$, etc).

\begin{notations}All the schemes and stacks will be supposed over $k$. We usually denote by $C$ a (connected) smooth projective curve over $k$, by $g:=g(C)$ its geometric genus and by $\omega_C$ (or $\omega$, when it is clear from the context) the cotangent bundle of $C$. A $G$-bundle over $C$ is a $G$-torsor over $C$, where $G$ acts on the right. Given a $G$-bundle $E$ over $C$ and a group-homorphism $\phi:G\to H$, we denote by $$E\times^{\phi,G}H:=(E\times H)/\sim$$ the associated $H$-bundle over $C$, where $\sim$ is the equivalence relation $(eg,h)\sim (e,\phi(g)h)$. Similarly, given a $G$-bundle $E$ over $C$ and a group-homorphism $\phi:G\to \op{GL}(V)$, we denote by $$E\times^{\phi,G}V:=(E\times V)/\sim$$ the associated vector bundle over $C$, where $\sim$ is the equivalence relation $(eg,v)\sim (e,\phi(g)v)$. When it is clear from the context, we remove the reference to the homomorphism $\phi$ from the notation (e.g. $E\times^GH$ instead of $E\times^{\phi,G}H$).
\end{notations}

\section{Lie algebras.}\label{S:LieAlg}
In this section, we collect some facts about the adjoint action of a reductive group $G$ on its Lie algebra $\g$.

\begin{teo}\label{T:konstant}Let $G$ be a reductive group over $k$. The following facts are true.
	\begin{enumerate}[(i)]
		\item\label{T:konstant1} The inclusion $\tg\subset\g$ of Lie algebras induces an isomorphism of rings $k[\g]^G\cong k[\tg]^W$, where $G$ acts on $\g$ via the adjoint representation.
		\item\label{T:konstant2} The ring $k[\g]^G$, resp. $k[\tg]^W$, is freely generated by homogeneous polynomials $p_1,\ldots,p_r$ in $k[\g]$, resp. in $k[\tg]$, of degrees $d_1,\ldots,d_r$. Moreover, the degrees do not depend on the choice of the polynomials.
		\item\label{T:konstant3} The morphism of affine varieties
		$$
	\begin{array}{cccl}
			\chi:&\g&\longrightarrow &\mathfrak{c}:=\Spec(k[\tg]^W)\\
			&x&\longmapsto& (p_1(x),\ldots,p_r(x))
	\end{array}
$$
is $\mathbb G_m$-equivariant, where $\mathbb G_m$ acts on $\g$ with weight $1$ and it acts on $\cg$ as follows: $t.(u_1,\ldots,u_r):=(t^{d_1}u_1,\ldots,t^{d_r}u_r)$.
\item\label{T:konstant4} The morphism of affine varieties
\begin{equation*}
	\begin{array}{cccl}
		\pi:&\tg&\longrightarrow &\mathfrak{c}\\
		&t&\longmapsto& (p_1(t),\ldots,p_r(t))
	\end{array}
\end{equation*}
is a Galois cover with Galois group $W$. The discriminant (or branch divisor) $\Dg$ is given by the zeroes of the $W$-invariant function
$
\prod_{\alpha\in\Phi}d\alpha.
$
\item\label{T:konstant5} For any simple root $\alpha$, we fix a non-zero vector $x_{\alpha}\in\g_{\alpha}$, we set $x_+:=\sum_{\alpha\text{ simple }}x_{\alpha}$. Then, there exists a unique $\mathfrak sl_2$-triple $\{x_-,h,x_+\}$ in $\g$, with $h\in\tg$. 
\item\label{T:kostant5bis} With the notations of Point $(\ref{T:konstant5})$. Let $\g^{x_+}$ be the Lie algebra of the centraliser of $x_+$. Then, the affine subspace $x_-+\g^{x_+}\subset\g$ is contained in the open subset $\g^{\reg}$. Furthermore, the restriction of $\chi$ to this subspace
$$
x_-+\g^{x_+}\to \cg
$$
is an isomorphism. Its inverse $\epsilon:\cg\to \g^{\reg}$ is called \emph{Kostant section}.
\item\label{T:konstant6} Let $P\subset G$ be a parabolic subgroup with Levi subgroup $M$. Let $\mathfrak p\subset \mathfrak g$ be the corresponding parabolic subalgebra and $\mathfrak m$ be its Levi subalgebra. Then the inclusions $\tg\subset\mathfrak m\subset\mathfrak p$ induce the isomorphisms of rings $k[\mathfrak p]^P\cong  k[\mathfrak m]^M\cong k[\mathfrak t]^{W_M}$, where $W_M$ is the Weyl group of $M$.
\end{enumerate}
\end{teo}

\begin{proof}All the facts are proved in \cite{Ko63} (see also \cite[Theorem 2.1]{NGO6}), with the only exception of the isomorphism $k[\mathfrak p]^P\cong k[\mathfrak m]^M$, which has been proved in \cite[Lemme 2.7.1]{CL1}.
\end{proof}

The Weil group $W$ acts on the set of roots $\Phi$. We denote by $\Phi_1,\ldots,\Phi_m$ the sets of $W$-orbits in $\Phi$. This will give us the irreducible decomposition of the discriminant
$$
\Dg=\Dg_1\cup\cdots\cup\Dg_m,
$$
where $\Dg_i$ is the zero-locus of the $W$-invariant function $\prod_{\alpha\in\Phi_i}d\alpha$. Furthermore, the singular locus $\Dg^{\sing}$ is the zero loci of the polynomial system
$$
\begin{cases}
	\prod_{\alpha\in\Phi}d\alpha=0,\\
	\sum_{\alpha\in\Phi}\left(\prod_{\beta\in\Phi\setminus \{\alpha\}}d\beta\right)=0.
\end{cases}
$$
In other words, it may be described as the image of points in $\tg$ contained in at least two distinct hyperplane-roots. In particular, it is a union of irreducible closed subsets in $\cg$ of dimension $r-2$. We denote these components by $\Dg^{\sing}_1,\ldots,\Dg^{\sing}_n$. It follows from the discussion that the irreducible components of $\Dg^{\sing}$ are in bijection with orbits of the diagonal action of the Weil group $W$ on the product $\Phi\times\Phi$.

\begin{rmk}\label{R:h}Let $G$ be a reductive group with Lie algebra $\g$. Here, we collect some facts about the degrees $d_1,\ldots,d_r$ introduced in Theorem \ref{T:konstant}.
\begin{enumerate}[(i)]
	\item $G$ is semisimple if and only if all the degrees $d_1,\ldots,d_r$ are greater or equal than $2$.
	\item  If $G$ almost-simple, the highest degree among $d_1,\ldots,d_r$ is the so-called Coxeter number $h$, which is equal to the number of roots divided by the rank of the maximal torus, i.e. $h:=\frac{|\Phi|}{r}$.
	\item $G=T$ is abelian if and only if all the degrees are equal to $1$. Furthermore, we have $\pi_G=\Id:\tg\to\tg$.
\end{enumerate}	
\end{rmk}

\begin{rmk}\label{R:decom-pi} In general, a reductive group $G$ admits a canonical isogeny $\mathscr Z(G)^o\times\mathscr D(G)\to G$, where $\mathscr Z(G)^o\subset G$ is the neutral component of the center $\mathscr Z(G)$ and $\mathscr D(G)\subset G$ is the derived subgroup. Then, we have the a commutative diagram of affine schemes
	\begin{equation}
		\xymatrix{
			\tg_G\ar@{=}[d]\ar[rr]^{\pi_G}&\quad&\cg_G\ar@{=}[d]\\
			\tg_{\mathscr Z(G)}\oplus \tg_{\mathscr D(G)}\ar[rr]^{(\Id,\pi_{\mathscr D(G)})}&\quad&\tg_{\mathscr Z(G)}\oplus \cg_{\mathscr D(G)}\\
		}
	\end{equation}
	where the horizontal rows are the morphisms in Theorem \ref{T:konstant}$(\ref{T:konstant4})$.
\end{rmk}

We do not know an explicit description of the function $\prod_{\alpha\in\Phi}d\alpha$ as a polynomial in $k[\tg]^W$. However, as the next result shows, we have a partial description, which is enough for our purposes (e.g. see Proposition \ref{P:dual-curve})

\begin{lem}\label{R:tang-cone-disc} Assume $\g$ simple. If we order the degrees such that $d_1\leq\ldots\leq d_r$, we have the following equality
	$$
	\prod_{\alpha\in\Phi}d\alpha = c p_r^{r} +f_1(p_1,\ldots,p_r)+\ldots+f_m(p_1,\ldots,p_r),
	$$ 
where $c\in k\setminus\{0\}$ and $f_i$ are monomials in $k[\tg]^W=k[p_1,\ldots,p_r]$ of total degree stricly greater than $r$.
\end{lem}

\begin{proof}Since we assume $\g$ simple, there exists an isomorphism of rings
	$$
	k[\tg]\cong k[d\alpha_1,\ldots,d\alpha_r],
	$$
	where $\alpha_1,\ldots,\alpha_r$ are the simple roots. The polynomial $\prod_{\alpha\in\Phi}d\alpha$ is an homogeneous polynomial of degree $|\Phi|$ in $k[\tg]$. On the other hand, $p_r$ is an homogeneous polynomial in $k[\tg]$ of degree $|\Phi|/r$ (see Remark \ref{R:h}). Hence we must have,
	$$
	\prod_{\alpha\in\Phi}d\alpha = c p_r^{r} +f_1(p_1,\ldots,p_r)+\ldots+f_m(p_1,\ldots,p_r),
	$$ 
	where $c\in k$ and $f_i$ are monomials in $k[\tg]^W=k[p_1,\ldots,p_r]$. Without loss of generality, we may assume that $f_i\neq d p_r^r$ for some constant $d\in k\setminus\{0\}$, for any $i$. We need to show that $c\neq 0$ and $\deg(f_i)>r$.
	
	\noindent\fbox{$c\neq 0$} We prove it by contradiction. Assume that $c=0$. Then, any element $t\in\tg$ such that $\pi(t)=(0,\ldots,0,a)\in\cg$, for some $a\in k$, is contained in some hyperplane-root $d\alpha:\tg\to k$. In particular, for any $t$ of this shape, the kernel of the adjoint homomorphism $[t,-]:\g\to\g$ contains the sublinear space $\g_{\alpha}\oplus\tg\oplus \g_{-\alpha}$. In particular, the kernel of $[t,-]$ have dimension at least $r+2$, i.e. $t$ is not regular. However, if $a\neq 0$, the element $t$ must be cyclic (then regular), see \cite[Corollary 9.2]{Ko59}.	
	
	\noindent\fbox{$\deg(f_i)>r$} We prove it by contradiction. Assume that $\deg(f_i)\leq r$, i.e. there exist $(a_1,\ldots,a_r)\in \mathbb N^r$, with $a_i\neq 0$ for some $i\neq r$, such that $
		f_i=p_1^{a_1}\cdots p_r^{a_r}$, $a_1+\cdots+a_r\leq r$ and $a_1d_1+\cdots+a_rd_r=|\Phi|$.
	Since $\g$ simple, the degree $d_i$ is stricly smaller than $d_r=|\Phi|/r$ for any $i< r$. In particular, since there is at least one $a_i\neq 0$ for some $i< r$, we get the contradiction
	$$
	|\Phi|=a_1d_1+\ldots+a_rd_r<(a_1+\cdots+a_r)|\Phi|/r\leq |\Phi|.
	$$	
\end{proof}

\section{Hitchin fibration.}\label{Sec:hit-fib}
Here, we recall some fundamental facts about (Higgs) bundles on curves and their moduli stacks following the works of Faltings \cite{Fa93} and Ng\^o \cite{NGO10}. Unless otherwise stated, in this section, $G$ is a (connected) reductive group and $C$ is a smooth projective curve of genus $g\geq 2$.

\subsection{The moduli stack $\hit_G$ of $G$-Higgs bundles.}\label{SS:hihh}
For any $G$-bundle $E$ over $C$, we denote by $\ads(E):=E\times^{\operatorname{Ad},G}\g$ the adjoint bundle on $C$ of attached to $E$ via the adjoint representation $\text{Ad}:G\to \op{GL}(\g)$.

\begin{defin}\label{D:Higgs}A \emph{$G$-Higgs bundle} is a pair $(E,\theta)$, where $E$ is a $G$-bundle over $C$ and $\theta$ is a global section of the vector bundle $\ad{E}$. A \emph{regular $G$-Higgs bundle} is a $G$-Higgs bundle $(E,\theta)$ such that, for any point $p\in C$, the kernel of the linear morphism of vector spaces
$$
[\theta_p,-]:\operatorname*{ad}(E_p)\to\ad{E_p}_p
$$
has dimension $r$.
\end{defin}

We denote by $\hit_G$ the moduli stack of $G$-Higgs bundles over $C$ and by $\BG_G$ the moduli stack of $G$-bundles over $C$. They are algebraic stacks locally of finite type over $k$. They are related by the forgetful morphism $$\hit_G\to\BG_G:(E,\theta)\mapsto E.$$ Consider the affine space $\AG$ of sections of the vector bundle $\cg_\omega:=\cg\times^{\mathbb{G}_m}\omega^*$, where $\omega^*:=\omega\setminus\{\operatorname{zero-section}\}$ is the $\mathbb G_m$-bundle attached to $\omega$. After fixing a set of generators of $k[\g]^G=k[\tg]^W$, we get the $\mathbb{G}_m$-equivariant isomorphisms
\begin{equation*}\label{E:dec-eq}
	\cg_\omega\cong\omega^{d_1}\oplus\ldots\oplus\omega^{d_r}\quad\text{ and }\quad\AG\cong H^0(C,\omega^{d_1})\oplus\ldots\oplus H^0(C,\omega^{d_r}).
\end{equation*}
For the rest of the paper, we will fix such a set of free generators, we also assume that $d_1\leq \ldots\leq d_r$ and we set $\AG_i:=H^0(C,\omega^{d_i})$. From this decomposition, we deduce that:
\begin{align*}
\dim\AG&=\sum_{i=1}^rd_i(2g-2)+r(1-g)+\#\{i|d_i=1\}=\\
&=\sum_{i=1}^rd_i(2g-2)+r(1-g)+\dim\mathscr Z(G)=\\&=\left(r+\frac{|\Phi|}{2}\right)(2g-2)+r(1-g)+\dim\mathscr Z(G)=\\&=\dim G(g-1)+\dim \mathscr Z(G).
\end{align*}
Indeed, the first two equalities follows by using Riemann-Roch and Remarks \ref{R:h} and \ref{R:decom-pi}. We recall that the integers $d_1,\ldots,d_r$ are the degrees of a set of free generators for the ring of $W$-invariants $k[\tg]^W$. Since the Weyl group $W$ is generated by $|\Phi|/2$ reflections, we have $\sum_{i=1}^{r}(d_i-1)=|\Phi|/2$ (see \cite{ST}) and, so, the third equality follows immediately. The last equality follows because $\dim G=\dim\g=r+|\Phi|$.

 We call \emph{Hitchin morphism} the following map:
\begin{equation*}\label{E:hit-map}
\begin{array}{cccl}
h:&\hit_G&\longrightarrow&\AG\\
& (E,\theta)&\longmapsto& \left(p_1(\theta),\ldots,p_r(\theta)\right).
\end{array}
\end{equation*}
Given a section $a\in\AG$, we call \emph{Hitchin fiber over $a$} the inverse image $\hit_G(a):=h^{-1}(a)$. The choice of a theta-characteristic $\kappa$ (i.e. a line bundle $\kappa$ such that $\kappa^2=\omega$) and the Kostant section (see Theorem \ref{T:konstant}(v)) gives a section of the Hitchin morphism
\begin{equation}\label{E:Hit-kos}
\epsilon_\kappa:\AG\to \hit_G,
\end{equation}
called the \emph{Hitchin-Kostant section}, for more details see \cite[Proposition 2.5]{NGO6} and \cite[Lemma 2.2.5]{NGO10}. The image of $\epsilon_\kappa$ is contained in the locus $\hit_G^{\reg}$ of regular $G$-Higgs bundles, see \cite[\S 4.3.4]{NGO10}. In particular, we get a family of regular $G$-Higgs bundles over $\AG$. For any $a\in \AG$, we denote by $(E_a,\theta_a):=\epsilon_\kappa(a)$, the corresponding regular $G$-Higgs bundle given by the image of the Hitchin-Kostant section.

\vspace{0.1cm}

We conclude the subsection with some basic facts about these moduli stacks. These facts are probably well-known to the experts. However, we did not find a reference in this generality. For this reason, we added a sketch of the proof.

\begin{prop}\label{P:hit-prop}The following facts hold.
\begin{enumerate}[(i)] 
	\item\label{P:hit-prop1} The connected components of the moduli stacks $\hit_G$ and $\BG_G$ are in bijection with $\pi_1(G)$.
	\item\label{P:hit-prop2} For any $\delta\in\pi_1(G)$, the corresponding connected component $\BG_G^{\delta}$ of the moduli stack $\BG_G$ is a smooth algebraic stack locally of finite type over $k$. Furthermore, it has dimension $\dim G(g-1)$.
	\item\label{P:hit-prop3} For any $\delta\in\pi_1(G)$, the corresponding connected component $\hit_G^{\delta}$ of the moduli stack $\hit_G$ is an irreducible algebraic stack locally of finite type over $k$. Furthermore, it is a locally complete intersection of pure dimension $2\dim G(g-1)+\dim\mathscr Z(G)$.
	
	\item\label{P:hit-prop5} The Hitchin morphism $h^{\delta}:\hit_G^{\delta}\to\AG$ is faithfully flat with fibers of pure dimension $\dim G(g-1)$.
\end{enumerate}
\end{prop}

\begin{proof}\fbox{$(\ref{P:hit-prop2})$} The moduli stack $\BG_G$ is locally of finite type over $k$ by \cite[Proposition 4.4.5]{BK} and smooth by \cite[Proposition 4.5.1]{BK}. For the dimension see \cite[Corollary 8.1.9]{BK}.\\
\fbox{$(\ref{P:hit-prop1})$+$(\ref{P:hit-prop3})$} The statement of $(i)$ for $\BG_G$ can be found in \cite{Ho10}. We now focus on the moduli stack $\hit_G$. When $G$ is semi-simple, a proof can be found in \cite{BD}. More precisely, by \cite[Proposition 1.1.1]{BD}, the moduli stack $\mathbf{Bun}_G$ is \emph{very good}, i.e. for any integer $n>0$, we have that the substack
	\begin{equation}\label{E:very-good}
	\{E\in\mathbf{Bun}_G\vert \dim \operatorname{Aut}(E)=n\}=\{E\in\mathbf{Bun}_G\vert, \dim H^1(C,\ad{E})=n\}
	\end{equation}
has codimension strictly greater than $n$ in $\BG_{G}$. We remark that the two sets above coincide by Serre duality and by the equality $\dim H^0(C,\ads(E))=\dim \operatorname{Aut}(E)$. The property of being very good implies that the forgetful morphism $\hit_G\to\mathbf{Bun}_G$ induces a bijection of irreducible/connected components and also that $\hit_G$ is an algebraic stack locally of finite type and a locally complete intersection of pure dimension $2\dim G(g-1)$, see \cite[Section 1]{BD}. Now, we assume $G$ reductive. Observe that the moduli stack $\hit_G$ sits in the following cartesian diagram
	\begin{equation}\label{E:cart-higg}
		\xymatrix{
			\hit_{G}\ar[d]\ar[r]&\hit_{G^{\operatorname{ad}}}\ar[d]\\
			H^0(C,\tg_{\mathscr Z(G)}\otimes\omega)\times\mathbf{Bun}_G\ar[r]& \mathbf{Bun}_{G^{\operatorname{ad}}},	
		}
	\end{equation}
where the horizontal map on the bottom sends the pair $(s,E)$ to the $G^{\ads}$-bundle $E\times^GG^{\ads}=E/\scr Z(G)$. By \cite[Corollary 4.2]{Ho10}, the horizontal arrows are smooth of relative dimension 
\begin{align*}
\dim H^0(C,\tg_{\mathscr Z(G)}\otimes\omega)+\dim \BG_{\scr Z(G)}&=\dim\scr Z(G)g+ \dim \mathscr Z(G)(g-1)=\\&= \dim \mathscr Z(G)(2g-1).
\end{align*}
 Since $G^{\ads}$ is semi-simple, we have the equalities
\begin{align*}
\dim \hit_G&=\dim \hit_{G^{\ads}}+\dim \mathscr Z(G)(2g-1)=\\&=2\dim G^{\ads}(g-1)+\dim \mathscr Z(G)(2g-1)=\\&=2\dim G(g-1)+\dim \mathscr Z(G).
\end{align*}
Furthermore, the smoothness of the horizontal maps implies that $\hit_G$ is a locally complete intersection of pure dimension $2\dim G(g-1)+\dim \mathscr Z(G)(g-1)$. Now we compute the irreducible/connected components of $\hit_G$. From the semi-simple case (see \eqref{E:very-good}), we deduce that, for any $n>0$, the substack
$$
\mathbf{Bun}_{G,= n}:=\{E\in\mathbf{Bun}_G\vert \,H^1(\ad{E})= n+\dim \mathscr Z(G)\}\subset \mathbf{Bun}_G
$$
has codimension strictly greater than $n$ in $\mathbf{Bun}_G$. From this fact, we deduce that the locus $\hit_G^o$ of $G$-Higgs bundles such that $\dim H^1(C,\ad{E})=\dim\mathscr{Z}(G)$ (i.e. of minimal dimension) is dense in $\hit_G$. We proceed by contradiction. Assume that $\hit_G^o$ is not dense in $\hit_G$. Then, there exists an irreducible component of $\hit_G$ with an open subset $\hit_G^{\#}$ mapping dominantly onto the substack $\mathbf{Bun}_{G,= n}$, for some $n>0$. In particular, there exists a $G$-bundle $E$ in $\mathbf{Bun}_{= n}$ such that
\begin{align*}
\dim \hit_G^*&=\dim \mathbf{Bun}_{G,= n}+\dim H^0(C,\ad{E})<\\
&<(-n+\dim\mathbf{Bun}_G)+\chi(\ad{E})+\dim H^1(C,\ad{E})=\\
&=(-n+\dim\mathbf{Bun}_G)+\dim G(g-1)+(n+\dim \mathscr Z(G))=\\
&=2\dim G(g-1)+\dim \mathscr Z(G)=\dim\hit_G.
\end{align*}
Since all the irreducible components of $\hit_G$ have the same dimension, we get a contradiction. Hence, the connected components of $\hit_G$ are in bijection with those of $\hit_G^o$. The point $(i)$ follows by observing that the forgetful morphism $\hit_G^o\to\mathbf{Bun}_G$ is a vector bundle. This also show that all the connected components are irreducible, concluding the proof of the point $(iii)$.\\
\fbox{$(\ref{P:hit-prop5})$} When $G$ semi-simple, a proof of the assertion about the flatness and the dimension of the fibers can be found in \cite[Theorem 2.2.4(iii)]{BD}. When $G$ is reductive, the assertion can be deduced from the observation that the statement holds for the adjoint (and semi-simple) group $G^{\ads}$, together with the cartesian diagram \eqref{E:cart-higg}.
\end{proof}

\subsection{Cameral covers.}\label{SS:cameralcovers} To any element $a=(a_1,\ldots,a_r)\in\AG$, the corresponding \emph{cameral curve} $C_a$ is the curve in the total space of the vector bundle $\tg_\omega:=\tg\times\omega$ defined by the equations
\begin{equation*}
\begin{cases}
	p_1(t)=a_1(x),\\
	\vdots\\
	p_r(t)=a_r(x),
\end{cases}
\end{equation*}
where $x$ is a (local) coordinate for $C$ and $t$ are the tautological coordinates $(t_1\otimes dx,\ldots,t_r\otimes dx)$ along the fibers of the projection $\tg_\omega\to C$. 
In other words, the cameral curve sits in the following cartesian diagram
\begin{equation}\label{E:cam-cov}
\xymatrix{
	C_a\ar[d]^{\pi_a}\ar[r]&\tg_{\omega}\ar[d]^{\pi}\\
	C\ar[r]^<<<<<<a&\cg_\omega.
}
\end{equation}
In the paper, we use the same symbol $\pi$ for the morphism $\tg\to \cg$ in Theorem \ref{T:konstant}$(\ref{T:konstant4})$ and its $\omega$-twisted version $\pi:\tg_{\omega}\to \cg_{\omega}$ in \eqref{E:cam-cov}. We call \emph{cameral cover} the morphism $\pi_a:C_a\to C$ in \eqref{E:cam-cov}.
\vspace{0.1cm}

The \emph{discriminant divisor }$\Dg_{\omega}:=\Dg\times^{\mathbb G_m}\omega^*$ is an hypersurface in $\cg_{\omega}$. The image $a(C)$ of a section $a\in \AG$ always meets the divisor $\Dg_{\omega}$, because $\omega$ is ample. So, we may define some subsets of sections $a\in\AG$ depending on the position of the curve $a(C)$ with respect to the discriminant divisor $\Dg_{\omega}$. For instance, let $\AG^{\heartsuit}$ be the subset of $\AG$ whose sections are not contained in $\Dg_{\omega}$, i.e.
\begin{equation*}\label{E:ag-hearth}
	\AG^{\heartsuit}:=\{a\in\AG\,| a(C)\nsubseteq \Dg_{\omega}\}.
\end{equation*}
The next proposition resumes the main facts about this set.
\begin{prop}\label{P:heart}The open subset $\AG^{\heartsuit}$ is not-empty. Moreover, if $a\in\AG^{\heartsuit}$, then the cameral cover $C_a$ is reduced and connected.
\end{prop}

The proof of this fact is due to Ng\^o \cite{NGO10} under different hypotheses on the line bundle $\omega$. We briefly explain how to adapt the strategy on \emph{loc.cit.} in our situation.

\begin{proof}By Remark \ref{R:decom-pi}, the morphism $\pi_G:\tg_G\to\cg_G$ can be re-written
$$
		\left(\operatorname{Id},\pi_{\mathscr{D}(G)}\right):	\tg_{\mathscr Z(G)}\oplus\tg_{\mathscr D(G)}\to \tg_{\mathscr Z(G)}\oplus \cg_{\mathscr D(G)}.
$$
Hence, we may assume $G$ semi-simple. The non-emptiness will follow from Proposition \ref{P:gen-discr}. The cameral curve is reduced by \cite[Lemme 4.5.1]{NGO10} and connected by \cite[Prop. 4.6.1]{NGO10}. We remark that the canonical bundle $\omega$ does not satisfy the hypothesis $\deg(\mt L)\geq 2g$ in \emph{loc.cit.} However, the same result, with the exact same proof, still holds if we assume $\mt L^{n}$ very ample for $n=d_1,\ldots,d_r$. Since we are assuming $G$ semi-simple, we have $d_1,\ldots,d_r\geq 2$ (see Remark \ref{R:h}). Hence, we get the assertion, because the cotangent bundle $\omega^n$ of a smooth curve of genus at least two is very ample for any $n\geq 2$.
\end{proof}	

\subsection{Regular centralizer.}
Here, we define a moduli stack of torsors under an abelian group scheme $J_a$ on $C$, which depends on the choice of the section $a\in \AG$. It turns out that these moduli stacks completely describe the regular locus of the Hitchin fibers.

\begin{defin}\label{D:centralizer-gr}Let $\pi:\tg_{\omega}\to \cg_{\omega}$ be the $\omega$-twisted quotient in \eqref{E:cam-cov}. The \emph{group of $W$-equivariant maps over $\cg_{\omega}$}, denoted by $J^1$, is the smooth and commutative group scheme over $\cg_{\omega}$
	$$
	\begin{array}{cccl}
		J^1:&\operatorname*{Sch}/\cg_{\omega}&\longrightarrow&(\operatorname*{Grps})\\
		&(S\to \cg_{\omega})&\longmapsto&\pi_{*}\left(\tg_{\omega}\times T\right)^W(S):=\text{Hom}_W(S\times_{\cg_{\omega}}\tg_{\omega},T).
	\end{array}
	$$
	The smooth and commutative group scheme $J$ over $\cg_{\omega}$ is the subfunctor of $J^1$, which associates to any $\cg_{\omega}$-scheme $S$, the subset $J(S)\subset J^1(S)$ of the $W$-equivariant morphisms 
	$$
	f:S\times_{\cg_{\omega} }\tg_{\omega}\to T
	$$
	such that for any geometric point $x\in S\times_{\cg_{\omega}} \tg_{\omega}$ fixed by the hyperplane-root reflection $s_{\alpha}$ for some $\alpha$, we have $\alpha(f(x))\neq -1$.\\	
	For any section $a\in\AG$, we denote by $J_a=a^*J$ the group scheme over $C$ obtained as a pull-back of $J$ along the section $a:C\to \cg_{\omega}$.
\end{defin}

\begin{rmk}The quotient $J^1/J$ is a $\mathbb Z_2$-torsion sheaf. Furthermore, the condition of being in $J$ always holds for those roots whose corresponding coroots are primitive (i.e. they are injective as cocharacters). In particular, if the derived subgroup of $G$ does not contain the odd special orthogonal group $\op{SO}_{2n+1}$ (note that $\op{PGL}_2=\op{SO}_3$) as a direct factor, the groups $J$ and $J^1$ coincide, see \cite[Remark at page 123]{BD}.
\end{rmk}
For any $a\in\AG$, we denote by \emph{$\jator{a}$ the moduli stack of $J_a$-torsors over $C$}. It is a group-stack and it acts on the Hitchin fiber $\hit_G(a)$, by twisting the $G$-Higgs bundle. 

\begin{rmk}To be honest, the action is defined using a moduli stack of torsors with respect to an another group scheme: the universal centralizer group scheme, see \cite[\S 2]{NGO10}. However, these two groups (and so the moduli stacks of torsors under these groups) are isomorphic, see \cite[Proposition 2.4.7]{NGO10}. We have chosen this interpretation, because it shows more explicitly the role of the cameral curve in the geometry of the moduli stack, see for instance \cite[\S 4.5.2]{NGO10}. 
\end{rmk}	
	The next proposition shows the relationship between these group-stacks and the Hitchin fibers.

\begin{prop}Fix $a\in\AG$. The substack $\hit_G^{\reg}(a)\subset \hit_G(a)$ of regular $G$-Higgs bundles is a smooth open substack of the Hitchin fiber over $a$. Furthermore, $\hit_G^{\reg}(a)$ is a trivial $\jator{a}$-torsor.
\end{prop}

\begin{proof}See \cite[Proposition 4.3.3]{NGO10}.
\end{proof}

As consequence of the above proposition, the moduli stack $\jator{a}$ gives us some information about the geometry of the Hitchin fiber $\hit_G(a)$. The next two propositions collect the main facts about this group-stack.

\begin{prop}\label{P:aff-ab} Fix $a\in\mathbb{A}^{\heartsuit}$. Then the following facts hold.
	\begin{enumerate}[(i)]
		\item There exists a surjective morphism of stacks $\jator{a}\to \jator{a}^{\ab}$, where $\jator{a}^{\ab}$ is a group-stack, whose neutral component is an abelian stack (i.e. quotient stack of an abelian variety by the trivial action of a diagonalizable group).
		\item The fiber over the identity of the morphism $\jator{a}\to \jator{a}^{\ab}$ is an affine group scheme, which will be denoted by $\jator{a}^{\operatorname{aff}}$.
	\end{enumerate}
\end{prop}

\begin{proof}See \cite[Proposition 4.8.2]{NGO10}.
\end{proof}
The next formula gives the dimension of the affine part $\jator{a}^{\operatorname{aff}}$ of the moduli stack $\jator{a}$. The result has been obtained by Ng\^o (see \cite[Proposition 4.9.5]{NGO10}) based upon a result of Bezrukavnikov \cite{Bez}. Here, we present the formula as stated in \cite[Lemme 9.12.1]{CL2}, where they also provide an alternative proof. 
\begin{prop}[Ng\^o-Bezrukavnikov formula]\label{P:delta-for}Fix $a\in \AG^{\heartsuit}$. Let
	\begin{enumerate}[(i)]
		\item	$\pi_{a}:C_a\to C$ be the corresponding cameral cover,
		\item $B_a\subset C$ be the branch locus of the cameral cover,
		\item $\rho_a:\widetilde C_a\to C_a$ be a normalization of $C_a$.
	\end{enumerate} We then have
	\begin{equation}\label{E:delta}
	\delta(a):=\dim\jator{a}^{\operatorname{aff}}=\dim H^0\Big(C,\tg\otimes_{\oo_C}\big(\rho_{a*}(\oo_{\widetilde C_a})/\pi_{a*}(\oo_{C_a}\big)\Big)^W=\sum_{p\in B_a}\delta_{p}(a),
	\end{equation}
	where
	$$
	\delta_{p}(a)=\frac{1}{2}\left\{\deg_{p}(a^*\Dg)-\dim\tg+\dim\tg^{W_x}\right\}.
	$$
	Here, $W_x$ is the stabilizer of a point $x$ in the normalized curve $\widetilde C_a$ over the point $p\in C$.
\end{prop}

\begin{rmk}\label{R:delta-upsc}The previous proposition allows us to define a function
	\begin{equation}\label{E:delta-funct}
		\delta:\AG^{\heartsuit}\to\mathbb Z_{\geq 0}:a\mapsto\delta(a)=\dim\jator{a}^{\operatorname{aff}}.
	\end{equation}
By the results in \cite[\S 5.6]{NGO10}, the function $\delta$ is upper-semicontinuous.
\end{rmk}

\subsection{Product Formula.}

In this section, we will present the \emph{product formula}, due to Ng\^o. Roughly speaking, it describes the Hitchin fiber in terms of the group-stack of the previous subsection and certain affine Springer fibers. 

Before presenting the formula, we need some notations. Fix an element $a\in\AG^{\heartsuit}$. Let $\pi_a:C_a\to C$ be the corresponding cameral cover. We denote by $B_a\subset C$ the branch locus of the cameral cover. For any point $p\in B_a$, we define the following objects:
\begin{itemize}
	\item the formal disc $\Delta_p:=\operatorname*{Spec}(\widehat\oo_{C,p})$ centered at $p$ and the punctured formal disc $\Delta_p^*:=\Delta_p\setminus\{p\}$,
	\item the group scheme $J_{a,p}\to \Delta_p$ obtained from the group scheme $J_a\to C$ (see Definition \ref{D:centralizer-gr}) via pull-back along the natural morphism $\Delta_p\to C$.
	\item the group-functor $\japtor{a,p}$, which assigns, to any noetherian affine $k$-scheme $S$, the isomorphism classes of pairs $(P,\varphi)$, where $P$ is a $J_{a,p}$-torsor over $\Delta_p\times S$ and $\varphi$ is an isomorphisms of $J_{a,p}$-torsors between $P$ and the trivial one over $\Delta_p^*\times S$.
\end{itemize}

\begin{defin}\label{D:springer}Let $a\in\AG^{\heartsuit}$ and let $(E_a,\theta_a)$ be the regular $G$-Higgs bundle over $C$ given by the Hitchin-Kostant section, see \eqref{E:Hit-kos}. The \emph{affine Springer fiber over $a$} is the functor $\spr{p}$, which assigns, to any noetherian affine $k$-scheme $S$, the isomorphism classes of triples $(E,\theta,\varphi)$, where
	\begin{enumerate}[(i)]
		\item $E$ is a $G$-bundle over $\Delta_p\times S$,
		\item $\theta$ is a section of the adjoint bundle $\ads(E)$ over $\Delta_p\times S$,
		\item $\varphi$ is an isomorphism of $G$-bundles between $E$ and $E_a$ over $\Delta^*_p\times S$, which commutes with the sections $\theta$ and $\theta_a$.
	\end{enumerate}
We denote by $\spreg{a}$ the subfunctor of $\spr{a}$ of regular elements, i.e. those elements where the kernel of the morphism $[\theta_p,-]:\operatorname*{ad}(E)\to \operatorname*{ad}(E)$ at the fiber over $p$ has dimension $r$.
\end{defin}

The group-functor $\japtor{a,p}$ acts on the affine Springer fiber $\spr{p}$. Both functors are representable by (possibly non-reduced) schemes. We denote by $\japtor{a,p}^{\operatorname{red}}$ and $\operatorname{Spr}^{\operatorname{red}}_p(a)$ their reduced counterparts, respectively. The next theorem clarifies the relation between the Hitchin fibers and the affine Springer fibers.

\begin{teo}[Product Formula]\label{T:prod} Let $a\in\AG^{\heartsuit}$ and $B_a$ be the branch locus of the corresponding cameral cover. We have a homeomorphism of stacks
	\begin{equation*}\label{E:prod}
		\prod_{p\in B_a}\operatorname{Spr}^{\operatorname{red}}_p(a)\stackrel{\prod_{p\in B_a}\japtor{a,p}^{\operatorname{red}}}{\times}\jator{a}\longrightarrow \hit_G(a).
	\end{equation*}
Furthermore, the restriction to the open substacks of regular elements
$$
\prod_{p\in B_a}\operatorname{Spr}^{\operatorname{red,reg}}_p(a)\stackrel{\prod_{p\in B_a}\japtor{a,p}^{\operatorname{red,reg}}}{\times}\jator{a}\subset \prod_{p\in B_a}\operatorname{Spr}^{\operatorname{red}}_p(a)\stackrel{\prod_{p\in B_a}\japtor{a,p}^{\operatorname{red}}}{\times}\jator{a}$$
is an isomorphism onto the open locus of regular elements $\hit_G^{\reg}(a)$.
\end{teo}

\begin{proof} We sketch the proof as explained in \cite[\S 4.15]{NGO10}. By gluing along the Hitchin-Kostant section, we get a morphism 
	\begin{equation}\label{E:prodproof1}
	\prod_{p\in B_a}\operatorname{Spr}^{\operatorname{red}}_p(a)\to \hit_G(a),
\end{equation}
which sends regular elements in regular $G$-Higgs bundles. With a similar argument, we get a homomorphism of groups $\prod_{p\in B_a}\japtor{a,p}^{\operatorname{red}}\to\jator{a}$, compatible with \eqref{E:prodproof1}. Note that the morphism induces an action of the source on the target. Furthermore, twisting the Hitchin-Kostant section by a $J_a$-torsor gives a morphism
	\begin{equation}\label{E:prodproof2}
\jator{a}\to \hit_G^{\reg}(a)\hookrightarrow\hit_G(a).
\end{equation}
Combining the morphisms \eqref{E:prodproof1} and \eqref{E:prodproof2} together and then doing the quotient by the diagonal action of the group $\prod_{p\in B_a}\japtor{a,p}^{\operatorname{red}}$, we get the desired morphism \eqref{E:prod}. By \cite[Theorem 4.6]{NGO10}, it induces an equivalence over the $k$-points. Hence, it is a homeomorphism. The second assertion follows by construction.
\end{proof}

\subsection{Twisted $G$-Higgs bundles}\label{SS:twisted}

In this subsection, we present a slight generalization of the moduli stacks of $G$-Higgs bundles introduced in Subsection \ref{SS:hihh}. The results of this part are used only once in the paper and under a particular assumption (see Remark \ref{R:Warn} below). However, since we believe that the statements contained here could be of independent interest, we decided to write them in full generality.

Let $\mathcal L$ be a line bundle over the curve $C$. Let $\hit_{G,\mathcal L}$ be the moduli stack of $\mathcal L$-twisted $G$-Higgs bundles $(E,\theta)$, where $E$ is a $G$-bundle and $\theta$ is a section in $H^0(C,\mathrm{ad}(E)\otimes\mathcal L)$. It is an algebraic stack locally of finite type over $k$ and it comes equipped with its own Hitchin morphism
$$
\hit_{G,\mathcal L}\xrightarrow{h_\mathcal L} \AG_{\mathcal L}=H^0(C,\mathfrak c_\mathcal L)=\oplus_{i=1}^rH^0(C,\mathcal L^{d_i}):(E,\theta)\mapsto (p_1(\theta),\ldots,p_r(\theta)) ,
$$
where $\mathfrak c_\mathcal L:=\mathfrak c\times^{\mathbb G_m}\mathcal L^*$ and $\mathcal L^*:=\mathcal L\setminus\{\text{zero-section}\}$. Observe that when $\mathcal L=\omega$, we recover the definition of $G$-Higgs bundles given in Subsection \ref{SS:hihh} (i.e. $\hit_{G,\omega}=\hit_G$). For more details about $\mathcal L$-twisted $G$-Higgs bundles, see \cite[\S 4]{NGO10}.

\begin{rmk}\label{R:Warn} We warn the reader that, in this paper, the stack of $\mathcal L$-twisted $G$-Higgs bundles $\hit_{G,\mathcal L}$ and its properties (see Lemma \ref{L:twisted} and Corollary \ref{C:twisted} below) are used only in the proof of Lemma \ref{L:bohi}, for the case $\mathcal L=\omega(-p)$ with $p$ point in $C$. 
\end{rmk}
Following Definition \ref{D:Higgs}, we denote by $\hit_{G,\mathcal L}^{\mathrm{reg}}$ the open locus of regular $\mathcal L$-twisted $G$-Higgs bundles, i.e. $(E,\theta)$ such that, for any point $q\in C$, the kernel of the linear morphism of vector spaces $[\theta_q,-]:\mathrm{ad}(E_q)\to\mathrm{ad}(E_q)\otimes\mathcal L_q$ has dimension $r$.

The main result of this subsection is the existence of sections for the $\mathcal L$-twisted Hitchin morphism, when $G$ is adjoint.

\begin{lem}\label{L:twisted}Let $G$ be an adjoint reductive group and $\mathcal L$ be a line bundle on $C$. Then, there exists a section $\sigma:\AG_{\mathcal L}\to \hit_{G,\mathcal L}^{\mathrm{reg}}$ of the Hitchin morphism $h_\mathcal L$, whose image is contained in the regular locus.
\end{lem}

\begin{proof}We follow the same argument contained in the proof of \cite[Proposition 2.5]{NGO6} (for more details about the differences between our proof and the one in \emph{loc.cit.}, see Remark \ref{R:kostant} below). Recall that the center $\mathscr Z(G)$ of a reductive group $G$ is the intersection of the kernels of all the roots. In particular, we have $\mathrm{Hom}(\mathscr Z(G),\mathbb G_m)=\mathrm{Hom}(T,\mathbb G_m)/\langle\Phi\rangle$, where $\langle\Phi\rangle$ denotes the sublattice generated by the roots. In particular, since we assumed $G$ adjoint (i.e. $\mathscr Z(G)=\{1\}$), the character lattice $\mathrm{Hom}(T,\mathbb{G}_m)=\langle\Phi\rangle$ is generated by the roots. Let $\Delta=\{\alpha_1,\ldots,\alpha_r\}$ be a set of simple roots for $G$ and $\omega_1,\ldots,\omega_r$ be the dual basis in $\mathrm{Hom}(\mathbb G_m, T)$ with respect to $\Delta$, i.e. $\alpha_j\circ\omega_i=t^{\delta_{ij}}$. Then the cocharacter $\rho=\sum_{i=1}^r \omega_i$ acts by conjugation on the Lie algebra $\g$ as it follows
	\begin{equation}\label{E:rho(t)}
		\rho(t).x=\begin{cases}
			x&\text{if }x\in \mathfrak t;\\
			t^{ht(\alpha)}x&\text{if }x\in \mathfrak g_\alpha;
		\end{cases}
	\end{equation}
	where $ht(\alpha)=\sum m_i$, if $\alpha=\sum_{i=1}^rm_i\alpha_i$. By \cite[Proposition 2.2]{NGO6},
	the Kostant section $\epsilon:\cg\to \g^{\mathrm{reg}}$ (defined in Therorem \ref{T:konstant}$(\ref{T:konstant6})$) is $\mathbb{G}_m$-equivariant, where $\mathbb G_m$ acts on $\cg$ with weights $d_1,\ldots,d_r$ and on $\g$ in the following way: $\rho_+(t).x:=\rho(t).(tx)$, for any $x\in\g$ and $t\in\mathbb G_m$. The homomorphism $\mathbb G_m\to \mathbb G_m\times G:t\mapsto (t,\rho(t))$ and the Kostant section $\epsilon$ induce a commutative diagram of quotient stacks
	$$
	\xymatrix@R=2pt{
		[\mathfrak c/\mathbb G_m]\ar[dr]\ar[rr]&& [\g^\mathrm{reg}/\mathbb G_m\times G]\ar[dl]\\
		&\mathcal B\mathbb G_m&
	}
	$$
	Pulling-back the above diagram, via the morphism $C\to \mathcal B\mathbb G_m$ classifying the line bundle $\mathcal L$, we get a section $\epsilon_{\mathcal L}:\mathfrak c_{\mathcal L}\to [\g^\mathrm{reg}_{\mathcal L}/G]$ of the natural morphism $[\g^\mathrm{reg}_{\mathcal L}/G]\to \mathfrak c_{\mathcal L}$. In particular, we get the section of the Hitchin morphism
	\begin{equation}
		\begin{array}{crcl}
			\sigma :&\AG_{\mathcal L}=H^0(C,\cg_\mathcal L)&\longrightarrow& Hom_C(C,[\g^\mathrm{reg}_\mathcal L/G])=\hit^{\mathrm{reg}}_{G,\mathcal L}\\
			&C\xrightarrow{a}\cg_{\mathcal L}&\mapsto &	C\xrightarrow{a}\cg_{\mathcal L}\xrightarrow{\kappa_\mathcal L}[\g^{\mathrm{reg}}_{\mathcal L}/G] 
		\end{array}
	\end{equation}
\end{proof}

\begin{rmk}\label{R:kostant} The proof of the existence of the section $\sigma$ in Lemma \ref{L:twisted} is a modification of the proof of \cite[Proposition 2.5]{NGO6}. In \emph{loc.cit.} the author constructs a section for arbitrary reductive groups under the assumption that $\mathcal L=\mathcal M^{\otimes 2}$ (condition not satisfied by $\omega(-p)$, because it has odd degree).

	The reason of the assumption is that, for an arbitrary reductive group, the action of $\mathbb G_m$ on $\g$, with weights as in \eqref{E:rho(t)}, is not coming from the adjoint action of a cocharacter of the group. But it is true if we take its square: the corresponding cocharacter is the sum of the positive coroots. Taking the square of the action implies that, if we want to define a section of the Hitchin morphism properly, we need the choice of a root for $\mathcal L$.

	However, if we assume $G$ adjoint, the $\mathbb G_m$-action described in \eqref{E:rho(t)} is coming from a cocharacter (the sum of fundamental coweights) and, so, we do not need to assume that $\mathcal L$ is a square.
\end{rmk}

Following the notation of Subsection \ref{SS:cameralcovers}, we denote by $\AG_{\mathcal L}^\heartsuit$ the locus of sections $a$ such that $a(C)\nsubseteq \mathfrak D_{\mathcal L}:=\mathfrak D\times^{\mathbb G_m}\mathcal L^*$. We set $\hit_{G,\mathcal L}^{ \heartsuit}:=(h_\mathcal L)^{-1}(\AG_\mathcal L^\heartsuit)$ and $\hit_{G,\mathcal L}^{ \mathrm{reg}, \heartsuit}:=\hit_{G,\mathcal L}^{ \mathrm{reg}}\cap \hit_{G,\mathcal L}^{ \heartsuit}$.

\begin{rmk}Since the choice of the line bundle $\mathcal L$ is arbitrary, the affine space $\AG_\mathcal L$ may be a point (i.e. there is only the zero-section). Furthermore, even if $\AG_\mathcal L$ has positive dimension, the open subset $\AG_\mathcal L^\heartsuit$ could be empty.
\end{rmk}

The existence of a section for the Hitchin morphism allows us to deduce some important properties about the regular locus over $\AG_\mathcal L^\heartsuit$.

\begin{cor}\label{C:twisted}Let $G$ be an adjoint reductive group and $\mathcal L$ be a line bundle on $C$. Assume that $\AG^\heartsuit_\mathcal L\neq \emptyset$. Then, the following facts hold true.
	\begin{enumerate}[(i)]
		\item\label{i:bohi1} The morphism $h_\mathcal L:\hit_{G,\mathcal L}^{\mathrm{reg},\heartsuit}\to \AG^{\heartsuit}_{\mathcal L}$ is smooth and surjective.
		\item\label{i:bohi2} The locus $\hit_{G,\mathcal L}^{\mathrm{reg},\heartsuit}$ is open and dense in $\hit_{G,\mathcal L}^{\heartsuit}$.
	\end{enumerate}
\end{cor}
	
\begin{proof}	\fbox{\eqref{i:bohi1}} By Lemma \ref{L:twisted}, the locus of regular $\mathcal L$-twisted $G$-Higgs bundles surjects onto $\AG_{\mathcal L}$. Arguing as in the proof of \cite[Proposition 4.3.3]{NGO10}, we have that $\hit^{\mathrm{reg}}_{\mathcal L}$ is a trivial torsor under a group-stack $\BG_{J_{\mathcal L}}$ (denoted by $\mathcal P$ in \cite[\S 4.3.1]{NGO10}) over $\AG_\mathcal L$. Since the group-stack is smooth over $\AG_\mathcal L$ (see \cite[Proposition 4.3.5]{NGO10}), the same holds for the moduli stack of regular $\mathcal L$-twisted $G$-Higgs bundles.\\
	\fbox{\eqref{i:bohi2}} There exists a Product Formula (= Theorem \ref{T:prod}) for any fiber of the Hitchin morphism over $\AG^{\heartsuit}_{\mathcal L}$. Indeed, the statements (and their proofs) of \emph{loc.cit.} still works if we replace the group stack $\BG_{J_a}$ with the fiber of the group stack $\BG_{J_{\mathcal L}}$ and the Kostant-Hitchin section with the section $\sigma$ constructed in Lemma \ref{L:twisted}. One should also re-define the functor $\spr{p}$ by replacing the local isomorphism with the Hitchin-Kostant section (see Definition \ref{D:springer}$(iii)$) with a local isomorphism with the section $\sigma$. Similarly as for the functors $\spr{p}$, these new functors are represented by the schemes $\mathcal M_p(a)$ defined in \cite[\S 3.2, p. 40]{NGO10}.
	
	By the second statement of Product Formula and the fact that the regular locus of each scheme $\spr{p}$ ($=\mathcal M_p(a)$) is open and dense (see \cite[Proposition 3.10.1]{NGO10}), we see that the locus of regular $\mathcal L$-twisted $G$-Higgs bundles $\hit_{G,\mathcal L}^{\mathrm{reg}}$ is dense in each fiber over $\AG_{\mathcal L}^\heartsuit$. In particular, the locus of regular $\mathcal L$-twisted $G$-Higgs bundles over $\AG_{\mathcal L}^{\heartsuit}$ is dense in  $\hit_{G,\mathcal L}^{\heartsuit}$.
\end{proof}

\section{The Hitchin discriminant.}\label{Sec:hit-disc}

In this section, we study the geometry of the Hitchin fibers with respect to their image in the Hitchin basis $\AG$ (see Propositions \ref{P:uni-stack} and \ref{T:geom-fiber}). As a consequence, we show that the locus of non-abelian Hitchin fibers is covered by rational curves (see Theorem \ref{T:hit-disc}).

To achieve this, we focus our attention on two subsets of $\AG$. The first one, denoted by $\AG^{\diamondsuit}$, has been already introduced and studied in \cite{NGO10}. It parameterizes sections, whose corresponding cameral curves are smooth (see Proposition \ref{P:gen-discr}). By the results in \emph{loc.cit.}, the Hitchin fibers over $\AG^{\diamondsuit}$ are always abelian stacks. However, unless $G$ has semi-simple rank $1$, there are sections, whose Hitchin fibers are abelian stacks and their cameral curves are singular. For these reasons, we will introduce another open subset, denoted by $\AG^{\operatorname{ab}}$, which will parameterize sections, whose corresponding Hitchin fiber is an abelian stack (see Proposition \ref{P:gen-ab}).

Unless otherwise stated, in this section, $G$ is a reductive group and $C$ is a smooth projective curve of genus $g\geq2$.

\subsection{A characterization of the regular Hitchin fibers.} In this subsection, we present a criterion to check when a Hitchin fiber (over $\AG^{\heartsuit}$) contains non regular $G$-Higgs bundles.
	
\begin{prop}\label{P:descr-reg}Let $a\in \AG^{\heartsuit}$. The following facts are equivalent
	\begin{enumerate}[(i)]
		\item\label{P:descr-reg1} $\hit_G^{\reg}(a)=\hit_G(a)$.
		\item\label{P:descr-reg2} the connected components of the Hitchin fiber $\hit_G(a)$ are abelian stacks.
		\item\label{P:descr-reg3} the dimension of the affine part of the moduli stack $\jator{a}$ of $J_a$-torsors is zero, i.e. $$\delta(a)=\dim \jator{a}^{\operatorname{aff}}=0.$$
		\item\label{P:descr-reg4} for any point $p\in C$ such that $a(p)\in\Dg_{\omega}$, the image of the differential of the section $a:C\to\cg_{\omega}$ at the point $p$, i.e.
		$$
		da_p:\mt T_pC\to\mt  T_{a(p)}\cg_{\omega},
		$$ 
		is not contained in the tangent cone $\mt C_{a(p)}\Dg_{\omega}$ of $\Dg_{\omega}$ at the point $a(p)$.
	\end{enumerate}
\end{prop}

\begin{proof}
	
\fbox{$(\ref{P:descr-reg1})\Leftrightarrow(\ref{P:descr-reg3})$}
By Product Formula (= Theorem \ref{T:prod}), the point $(\ref{P:descr-reg1})$ is equivalent to the condition
$$
	\spr{p}=\spreg{p},\quad\forall p\in B_a.
$$
Let $\delta_p(a)$ be the local delta-invariant defined in the Ng\^o-Bezrukavnikov formula (= Proposition \ref{P:delta-for}).  Then, the assertion is equivalent to proving the following
\begin{equation*}
	\spr{a}=\spreg{a}\iff\delta_p(a)=0.
\end{equation*} It is known that $\delta_p(a)=\dim\spr{p}=\dim\japtor{a,p}$ (see \cite[Proposition 3.7.5]{NGO10}).
The implication $\Leftarrow$ is exactly the content of \cite[Corollary 3.7.2]{NGO10}. Conversely, assume that $\spr{p}=\spreg{p}$. By \cite[Proposition 3.4.1]{NGO10}, there exists a finite abelian subgroup $\Gamma\subset\japtor{a,p}$ such that the quotient $\operatorname{Spr}^{\red}(a)/\Gamma$ is a projective variety. On the other hand, the scheme  $\operatorname{Spr}^{\red, \reg}(a)/\Gamma$ of regular elements is a trivial $\japtor{a,p}^{\operatorname{red}, \operatorname{reg}}/\Gamma$-torsor (see \cite[Lemma 3.3.1]{NGO10}). Since the group scheme $\japtor{a,p}^{\operatorname{red}, \operatorname{reg}}/\Gamma$ is affine (see \cite[Lemma 3.8.1]{NGO10}), we have that $\operatorname{Spr}^{\operatorname{red}}(a)/\Gamma=\operatorname{Spr}^{\operatorname{red,reg}}(a)/\Gamma$ is an affine and projective variety. In particular, it is zero-dimensional and, so, $\delta_p(a)=\dim\spr{p}=0$.\\
\fbox{($\ref{P:descr-reg2})\Leftrightarrow(\ref{P:descr-reg3})$} It follows from Propositions \ref{P:aff-ab} and \ref{P:delta-for}.\\
\fbox{$(\ref{P:descr-reg4})\Rightarrow(\ref{P:descr-reg1})$} See \cite[Lemma C.0.1(3)]{Li20}.\\
\fbox{$(\ref{P:descr-reg3})\Rightarrow(\ref{P:descr-reg4})$} First, assume there exists a point $p\in C$ such that the curve $a(C)$ meets the smooth locus of $\Dg_{\omega}$ at $a(p)$ and the image of the differential of the section $a:C\to\cg_{\omega}$ is contained in the tangent space of $\Dg_{\omega}$ at $a(p)$. In other words, the curve $a(C)$ meets the smooth locus of the discriminant with multiplicity at least two. Arguing as in the proof of \cite[Lemme 4.7.3]{NGO10}, we may show that the cameral cover $\pi_a: C_a\to C$, locally at a point $v$ over $p$, is given by the following morphism
$$
\operatorname*{Spec}(k\llbracket t,x\rrbracket/(t^2-x^m))\to \operatorname*{Spec}(k\llbracket x\rrbracket), \quad \text{for some }m>1.
$$
The stabilizer $W_x$ of a point $x$ into the normalization of $C_a$ over $v$ is trivial, if $m$ is even and it is cyclic of order two, and generated by a hyperplane-root reflection $s_\alpha\subset W$ for some root $\alpha$, if $m$ odd. By Proposition \ref{P:delta-for}, we get
$$
\delta(a)\geq \delta_{p}(a)=\frac{1}{2}\left(\deg(mp)-\dim\tg+\dim\tg^{W_x}\right)=\begin{cases}
	{m/2,}&\text{if }m\text{ even},\\
	(m-1)/2,&\text{if }m\text{ odd}. 
\end{cases}
$$
Since $m>1$, we get $\delta(a)>0$ and so the assertion follows. Now, assume that the curve $a(C)$ never meets the smooth locus of the discriminant. By hypothesis, there exists $p\in C$ such that the curve $a(C)$ meets the singular locus of $\Dg_{\omega}$ at $a(p)$ and the image of $da_p(\nu)\in\mt  T_{a(p)}\cg_{\omega}$ is contained in the tangent cone $\mt C_{a(p)}\Dg_{\omega}$. By Lemma \ref{L:claim} below, there exists a family of sections $a_R:C_R:=C\times \operatorname{Spec}(R)\to \cg_{\omega}$ over a complete DVR $R$, such that
\begin{itemize}
	\item the image of the generic fiber $a_\eta:C_\eta\to \cg_{\omega}$ meets the smooth locus of the discriminant $\Dg_{\omega}$ at $a_{\eta}(p)$ with multiplicity at least two,
	\item the special fiber $a_s:C_s\to\cg_\omega$ coincides with the section $a$ (after the pull-back along $s\to \operatorname{Spec}k$).
\end{itemize} 
By upper-semicontinuity of the function $\delta:\AG^{\heartsuit}\to\mathbb Z_{\geq 0}$, we have $\delta(a)=\delta(a_s)\geq \delta(a_\eta)$ (see Remark \ref{R:delta-upsc}). Arguing as above, we get $\delta(a_\eta)>0$ and, so, $\delta(a)>0$.
\end{proof}

\begin{lem}\label{L:claim} Let $a\in\AG$ be a section such that $a(C)$ meets the discriminant divisor $\Dg_{\omega}$ at $a(p)$, for some point $p$, and the image of the differential $da_p:\mt T_pC\to \mt T_{a(p)}\cg_{\omega}$ is contained in the tangent cone $\mt C_{a(p)}\Dg_{\omega}$ of $\Dg_{\omega}$ at $a(p)$. Then, there exists a complete DVR $R$ and a family of sections 
	\begin{equation*}
		a_R:C_R:=C\times \operatorname{Spec}(R)\to \cg_{\omega},
	\end{equation*}
	with the following properties:
	\begin{enumerate}[(i)]
		\item if $\eta$ is the generic point of $\Spec(R)$, the image of the section $a_\eta:C_\eta\to \cg_{\omega}$ meets the smooth locus of the discriminant $\Dg_{\omega}$ at $a_{\eta}(p)$ with multiplicity at least two;
		\item if $s$ is the special point of $\Spec(R)$, the section $a_s:C_s\to\cg_\omega$ coincides with the section $a$ (after doing pull-back along $s\to \operatorname{Spec}k$).
	\end{enumerate}	
\end{lem}

\begin{proof}Let $\Delta:=\Spec(\disk{t})$ be the formal disc with special point $s$ and generic point $\eta$. Set $q:=a(p)\in \Dg$. By \cite[Proposition 20.2]{GH}, for any vector $\nu$ in the tangent cone $\mt C_q\Dg$ of $\Dg$ in $q$, there exists a morphism $\varphi:\Delta\to\mt  T\Dg\subset\mt  T\cg$ such that $\varphi(s)=(q,v)$ and $\varphi(\eta)$ is a tangent vector over the smooth locus $\Dg^{\sm}$ of $\Dg$.

Fix such a morphism $\varphi:\Delta\to\mt  T\Dg\subset\mt  T\cg$. Consider the morphism of vector bundles $C\times\AG\to \cg_\omega$. By Remark \ref{R:decom-pi}, the morphism $\pi_G:\tg_{G}\to\cg_{G}$ decomposes as follows
	$$
	(\Id,\pi_{\mathscr D(G)}):\tg_{\mathscr Z(G)}\oplus\tg_{\mathscr D(G)}\to\cg_{\mathscr Z(G)}\oplus\cg_{\mathscr D(G)}.
	$$
	Furthermore, the discriminant divisor $\Dg_G$ is the zero set of a polynomial defined by the hyperplane-roots. This means that $\Dg_{G}$ is the pull-back (along the obvious projection) of the discriminant divisor $\Dg_{\mathscr D(G)}\subset \cg_{\mathscr D(G)}$. In particular, the tangent cone decomposes as $\mt C_q\Dg_G=\mt T_{q_1}\cg_{\mathscr Z(G)}\oplus\mt  C_{q_2}\Dg_{\mathscr D(G)}$. Consider the composition
	$$
	\AG_G\xrightarrow{D} \AG_{\mathscr D(G)}\xrightarrow{\operatorname{ev_{2p}}}\cg_{\mathscr D(G)}\otimes\oo_{C,p}/\mathfrak m_{C,p}^2=\mt T\cg_{\mathscr D(G)}:b\mapsto (D(b)(p),D(b')(p)).
	$$
	It is a surjective linear morphism of vector spaces. In particular, it is smooth. Thus, using the morphism $\varphi:\Delta\to\mt T\cg$, we may lift $a:C\to \cg_\omega$ to a family of sections $C\times \Delta\to\cg_\omega$ with the properties requested by the lemma.
\end{proof}

\subsection{The subset $\AG^{\diamondsuit}\subset\AG$.}
We denote by $\AG^{\diamondsuit}$ the subset of $\AG^{\heartsuit}$, whose sections cut the divisor $\Dg_{\omega}$ transversally, i.e. 
\begin{equation*}
	\AG^{\diamondsuit}:=\{a\in\AG\,| a(C) \text{ intersects }\Dg_{\omega}\text{outside of the singular locus, with multiplicity }1 \}.
\end{equation*}
Consider its complementary subset $\DG^{\diamondsuit}:=\AG\setminus\AG^{\diamondsuit}$. It turns out that $\DG^{\diamondsuit}$ is a divisor in $\AG$. Furthermore, its irreducible components are naturally bijection with the irreducible components of the discriminant divisor and its singular locus. More precisely, set
$$
\def\arraystretch{1.2}\begin{array}{llll}
	\Dg_{\omega}&=&\Dg_{1,\omega}\cup\cdots\cup\Dg_{m,\omega},& \text{ such that  }\Dg_{i,\omega}:=\Dg_{i}\times^{\mathbb G_m}\omega^*,\\
	\Dg^{\sing}_{\omega}&=&\Dg^{\sing}_{1,\omega}\cup\cdots\cup\Dg^{\sing}_{n,\omega},& \text{ such that  }\Dg^{\sing}_{i,\omega}:=\Dg^{\sing}_{i}\times^{\mathbb G_m}\omega^*,
\end{array}
$$
where $\Dg_1,\ldots\Dg_m$, resp. $\Dg^{\sing}_1,\ldots,\Dg^{\sing}_n$, are the irreducible components of the discriminant divisor $\Dg$, resp. of the singular locus $\Dg^{\sing}$ of the discriminant divisor (see Section \ref{S:LieAlg}). The main result of this subsection is the following.

\begin{prop}\label{P:gen-discr} A section $a\in \AG$ is contained in $\DG^{\diamondsuit}$ if and only if the corresponding cameral curve $C_a$ is singular. 	
Furthermore, the set $\DG^{\diamondsuit}$ is the union of the irreducible hypersurfaces in $\AG$ 
	$$
	\DG^{\ab}_{1},\ldots, \DG^{\ab}_{m},\,\DG^{\operatorname*{extra}}_{1},\ldots,\DG^{\operatorname*{extra}}_{n},
	$$
	where:
	\begin{enumerate}[(i)]
		\item\label{P:gen-discr1} a section $a\in \mathbb A$ is contained in $\DG_i^{\ab}$ if and only if there exists a point $p\in C$ such that $a(p)\in\Dg_{i,\omega}$ and the image of the differential $da_p:\mt T_pC\to \mt T_{a(p)}\cg_{\omega}$ is contained in the tangent cone $\mt C_{a(p)}\Dg_{\omega}$ of $\Dg_{\omega}$ at $a(p)$. 
		\item\label{P:gen-discr2} a section $a\in \AG$ is contained in $\DG_i^{\operatorname*{extra}}$ if and only if there exists a point $p\in C$ such that $a(p)\in \Dg^{\sing}_{i,\omega}$.
	\end{enumerate}
\end{prop}

Before proving the proposition, we present some consequences. 

\begin{cor}
	\label{C:gen-discr} The generic section of $\DG^{\diamondsuit}$ meets $\Dg_{\omega}$ transversally, with the exception of a unique point $p\in C$, where either it meets the smooth locus of $\Dg_{\omega}$ with multiplicity $2$ at $a(p)$ or it meets singular locus $\Dg_{\omega}^{\sing}$ transversally at $a(p)$. More precisely:
	\begin{enumerate}[(i)]
		\item\label{C:gen-discr1} the generic section in $\DG^{\ab}_i$ meets $\Dg_{\omega}$ transversally, with the exception of a unique point in $p\in C$, where it meets the smooth locus $\Dg^{\sm}_{\omega}$ along the irreducible component $\Dg_{i,\omega}$ with multiplicity 2 at the point $a(p)$;
		\item\label{C:gen-discr2} the generic section in $\DG^{\operatorname*{extra}}_i$ meets $\Dg_{\omega}$ transversally, with the exception of a unique point $p\in C$, where it meets the irreducible component $\Dg^{\sing}_{i,\omega}$ at $a(p)$ and it meets the singular locus $\Dg_{\omega}^{\sing}$ transversally at the point $a(p)$.
	\end{enumerate}
\end{cor}
By the previous corollary, we also get the following result.
\begin{cor}\label{C:irr-cam}The following facts hold
	\begin{enumerate}[(i)]
		\item\label{C:irr-cam1} For any section $a\in \AG^{\diamondsuit}$, the corresponding cameral curve is integral.
		\item\label{C:irr-cam2} For a generic section $a\in \DG^{\diamondsuit}$, the corresponding cameral curve is integral.
	\end{enumerate}
\end{cor}

Note that the integrality of the smooth cameral curves has been already proved in \cite{NGO10}. Indeed, as in \emph{loc.cit.}, it is a consequence of Propositions \ref{P:heart} and \ref{P:gen-discr}. Here we present an alternative proof, which holds for both the hypotheses of Corollary \ref{C:irr-cam}. A similar argument can be found in the proof of \cite[Theorem III.2(i)]{Fa93}.

\begin{proof}[Proof of Corollary \ref{C:irr-cam}] Consider the universal cameral cover
	\begin{equation}\label{E:univ-cam}
		\xymatrix{
			C_{\AG}\ar[d]^{\pi_{\AG}}\ar[r]&\tg_{\omega}\ar[d]^{\pi}\\
			C\times \AG\ar[r]&\cg_{\omega}
		}
	\end{equation}
	The horizontal maps in \eqref{E:univ-cam} are smooth and surjective and their fibers are affine spaces. In particular, the scheme $C_\AG$ is a connected smooth variety. Let $\AG^*$ be the open subset in $\AG$, whose sections are either in $\AG^{\diamondsuit}$ or generic in $\DG^{\diamondsuit}$ in the sense of Corollary \ref{C:gen-discr}. Let $U\subset C\times \AG^{*}$ be the maximal open subset such that the restriction $\pi_{\AG}^{-1}(U)\to U$ is \'etale. Since $C_{\AG}$ is connected, the restriction of $\pi_{\AG}$ over $U$ is a Galois cover with Galois group $W$. Equivalently, the induced homomorphism $\pi_1(U,u)\xrightarrow{f_u} W$, for some $u\in U$, from the (\'etale) fundamental group to the Weyl group is surjective. Note that the projection on the first factor $U\hookrightarrow C\times \AG^{*}\xrightarrow{\operatorname{pr}_1}\AG^{*}$ is surjective and the complement of $U$ in $C\times\AG^*$ is a relative normal crossing divisor over $\AG^*$. By \cite[XIII, 4.1+4.4]{SGA1}, we get an exact sequence of fundamental groups
	\begin{equation}\label{E:fund-grSGA}
		\pi_1(C\setminus B_a,u)\xrightarrow{\iota} \pi_1(U,u)\to \pi_1(\AG^*,a)\to 1,
	\end{equation}
	where $\operatorname{pr}_1(u)=a\in\AG^*$. To conclude the proof, it is enough to show that the composition
	\begin{equation}\label{E:fund-gr}
		\pi_1(C\setminus B_a,u)\xrightarrow{\iota} \pi_1(U,u)\xrightarrow{f_u} W
	\end{equation}
	is surjective. Indeed, the surjectivity implies that the \'etale cover $V_a:=C_a\setminus\pi_a^{-1}(B_a)\to C\setminus B_a$ is Galois. In particular, $V_a$ (and so $C_a$) must be irreducible. First of all, we prove the following
	
	\noindent\underline{Claim}: for each irreducible component $\mathfrak D_{i,\omega}$, there exists a point $q\in B_a$ such that $a(q)\in \mathfrak D_{i,\omega}\cap \mathfrak D_{\omega}^{\sm}$.
	
	Indeed, by Lemma \ref{L:a-meet-d} below, the curve $a(C)$ meets each irreducible component $\Dg_{i,\omega}$ in at least $4g-4$ points, counted with multiplicity. On the other hand, by definition of $\AG^\star$, the curve $a(C)$ always meets $\Dg_\omega$ on the smooth locus with multiplicity one, with exception of at most one point $a(p)$ where it meets $\Dg_\omega$ with multiplicity two. Hence, $a(C)$ meets each irreducible component of $\Dg_{i,\omega}\cap\Dg_\omega^{\sm}$ in at least $4g-6$ distinct points. Since $g\geq 2$, the claim follows.

	Fix $q\in B_a$ such that $a(q)\in \mathfrak D_{i,\omega}\cap \mathfrak D_{\omega}^{\sm}$. By definition, the image $a(q)\in \cg_\omega$ is a smooth point of an irreducible component $\Dg_{\omega,i}$ of the discriminant divisor $\Dg_{\omega}$, which corresponds to a $W$-orbit of roots (see Theorem \ref{T:konstant}$(\ref{T:konstant4})$). Note that the homomorphism \eqref{E:fund-gr} maps the loop around the point $q\in B_a$ to a reflection $s_{\alpha}\in W$ with respect to a root $\alpha$ in the $W$-orbit met by $a(q)$. 
	
	By the claim above, $a(C)$ meets all the connected components of the smooth locus of $\Dg_{\omega}$, namely all the $W$-orbits of roots. In particular, the image $W_a:=(f_u\circ \iota)(\pi_1(C\setminus B_a,u))$ contains at least one representative for each conjugacy class of hyperplane-root reflections in $W$. Since the image of $\pi_1(C\setminus B_a,u)$ in $\pi_1(U,u)$ is normal (by sequence \eqref{E:fund-grSGA}) and $f_u$ is surjective, the image $W_a$ is a normal subgroup of $W$. Hence, $W_a$ is a normal subgroup of $W$ and it contains all the hyperplane-roots reflections. So, $W_a=W$.
\end{proof}

\begin{lem}\label{L:a-meet-d}For any $a\in \AG^{\heartsuit}$, the curve $a(C)$ meets each irreducible component of the discriminant divisor $\mathfrak D_{\omega}$ in at least $4g-4$ points, counted with multiplicity.
\end{lem}

\begin{proof}Let $\mathfrak{D}_i$ be an irreducible component of $\mathfrak D$. We have that $$\mathfrak{D}_i=\Spec(k[\mathfrak t]^W/F_i)\subset \cg=\Spec(k[\mathfrak t]^W),$$ where $F_i=\prod_{\alpha\in\Phi_i}d\alpha$ and $\Phi_i$ is a $W$-orbit in the set of roots $\Phi$ (see Section \ref{S:LieAlg}). In particular, the divisor $\mathfrak{D}_{i,\omega}$ may be described as the inverse image of the zero-section of $\omega^{N_i}$ along the morphism
$$
F_i:\cg_\omega\to \omega^{N_i},
$$
where $N_i=\deg(F_i)=|\Phi_i|$ is the total degree of $F_i$ as a homogeneous polynomial in $k[\mathfrak t].$ Thus, any section $a\in \AG^{\heartsuit}$ produces a non-zero section $F_i(a)\in H^0(C,\omega^{N_i})$ of the line bundle $\omega^{N_i}$. By construction, $a(C)$ meets $\mathfrak{D}_{i,\omega}$ in $n$ points, counted with multiplicity, if and only if the section $F_i(a)$ vanishes at $n$ points of $C$, counted with multiplicity. Since $F_i(a)$ is a non-zero section of the line bundle $\omega^{N_i}$, the curve $a(C)$ meets $\mathfrak{D}_{i,\omega}$ at $\deg(\omega^{N_i})=N_i(2g-2)$ points, counted with multiplicity. Then, the assertion follows by observing that $N_i$ is even, because if the $W$-orbit $\Phi_i$ contains a root $\alpha$, it contains also its negative $-\alpha$.
\end{proof}

It remains to prove Proposition \ref{P:gen-discr} and Corollary \ref{C:gen-discr}. Before we do this, we need some auxiliary results.
\begin{lem}\label{L:abincuore}Let $a$ be a section such that $a(C)\subset\Dg_{\omega}$, i.e. $a\in\DG^{\heartsuit}:=\AG\setminus\AG^{\heartsuit}$. Then, the image of the differential $da_p:\mt T_pC\to \mt T_{a(p)}\Dg_{\omega}$ is contained in the tangent cone $\mt C_{a(p)}\Dg_{\omega}$ for any point $p\in C$. Furthermore, the closed subset $\DG^{\heartsuit}$ has codimension at least three in $\AG$.
\end{lem}

\begin{proof}Locally at a point $p\in C$, the section $a$ is given by a ring-homomorphism
	$
	\varphi:k[p_1,\ldots,p_r]\to k\llbracket t\rrbracket,
	$
	such that if $\Dg$ is the zero-set of a polynomial $F$, we must have $\varphi(F(p_1,\ldots,p_r))=0$. Without loss of generality, we may assume that $p$ corresponds to the ideal $(t)$ and $a(p)$ to the maximal ideal $(p_1,\ldots,p_r)$, i.e. $(p_1,\ldots,p_r)=\varphi^{-1}(t)$. There exists a number $s$ such that
	$$
	\varphi(p_i)=a_it^{s} \operatorname{mod} t^{s+1},\quad\text{ for }i=1,\ldots,r,
	$$
	where $(a_1,\ldots,a_r)\in k^r\setminus\{\underline{0}\}$. If $s>1$, the claim holds because $da_p(T_pC)$ is the zero-vector. Assume $s=1$ and let $F=F_n+F_{n+1}+\ldots$ be the decomposition in homogeneous polynomials of $F$. We then have:
	\begin{align*}
		0&=\varphi(F(p_1,\ldots,p_r))=F_n\left(\scriptstyle{\frac{\varphi(p_1)}{t},\ldots,\frac{\varphi(p_r)}{t}}\right)t^n+F_{n+1}\left(\scriptstyle{\frac{\varphi(p_1)}{t},\ldots,\frac{\varphi(p_r)}{t}}\right)t^{n+1}+\ldots\\
		&= F_n(a_1,\ldots,a_r)t^n+ \text{polynomials in } t \text{ of degree bigger than } n.
	\end{align*}
	The equality implies the vanishing of $F_n\left(a_1\ldots,a_r\right)$. Note that the vector $(a_1,\ldots,a_r)$ is a generator of the image of $da_p$ in $T_{a(p)}\cg_{\omega}$. By definition, the tangent cone at $(p_1,\ldots,p_r)$ is the zero set of the homogeneous polynomial $F_n$ and, so, we have the first statement.

	We now show the second statement. Consider the set $U$ of $4$-uples $(q_1,q_2,q_3,a)\in C^3\times\AG$ such that $a(C)$ meets $\Dg_{\omega}$ and the image of the differential at $a(q_i)$ is contained in the tangent cone $\mt C_{a(q_i)}\Dg_{\omega}$ for $i=1,2,3$. By the previous statement, the set $\DG^{\heartsuit}$ is contained in the image $\op{pr}(U)$ of $U$ along the projection $\op{pr}:C^3\times\AG\to\AG$. Hence, it is enough to show that $\op{pr}(U)$ has codimension at least three in $\AG$.

	We denote by $U_{q_1,q_2,q_3}$ the fiber over a triple $(q_1,q_2,q_3)\in C^3$ along the composition $U\hookrightarrow C^3\times \AG\to\ C^3$. The fiber is canonically equal to a subset of the inverse image of the product $(\mt T\Dg)^3\subset(\mt T \cg)^3$ along the morphism
	\begin{equation}\label{E:ev}
	\AG\xrightarrow{\op{ev}_{2q_1,2q_2,2q_3}}\bigoplus_{i=1}^3(\cg\otimes\oo_{C,q_i}/\mathfrak m_{C,q_i}^2)=(\mt T\cg)^3.
	\end{equation}
	We will treat the cases $g\geq 3$ and $g=2$, separately.\\	
	\fbox{$g\geq 3$} We first assume $G$ semi-simple. In particular, all the degrees are greater or equal than two. The cokernel of the evaluation map \eqref{E:ev} is equal to 
	\begin{equation}\label{E:d}
	\bigoplus_{i=1}^rH^1(C,\omega^{d_i}(-2q_1-2q_2-2q_3))\cong\bigoplus_{i=1}^rH^0(C,\omega^{1-d_i}(2q_1+2q_2+2q_3))^*.
	\end{equation}
	By assumptions on the genus, we have the following inequalities
	 $$
	 \deg(\omega^{1-d_i}(2q_1+2q_2+2q_3))=(1-d_i)(2g-2)+6\leq -4d_i+10.
	 $$
	In particular, the factors in the decomposition \eqref{E:d} vanish, if $d_i\geq 3$. If $d_i=2$, we have $\deg(\omega^{1-d_i}(2q_1+2q_2+2q_3))=8-2g$. In particular, the remaining factors vanish (i.e. the evaluation map is surjective) for either $g\geq 5$ and any triple in $C^3$, or $g=3,4$ and a generic triple in $C^3$. Hence, if $g\geq 3$ and $(q_1,q_2,q_3)\in C^3$ generic, the fiber $U_{q_1,q_2,q_3}$ has codimension at least six in $\AG$. Then, the same holds for $U$ and, so, the projection $\op{pr}(U)$ has codimension at least three in $\AG$. Now, assume $G$ reductive. The morphism $\pi_G:\tg_{G}\to\cg_{G}$ decomposes as follows
	$$
	(\Id,\pi_{\mathscr D(G)}):\tg_{\mathscr Z(G)}\oplus\tg_{\mathscr D(G)}\to\tg_{\mathscr Z(G)}\oplus\cg_{\mathscr D(G)},
	$$
	see Remark \ref{R:decom-pi}. Furthermore, the discriminant divisor $\Dg_G$ is the zero set of a polynomial defined by the hyperplane-roots. It means that $\Dg_{G}$ is the pull-back (along the obvious projection) of the discriminant divisor $\Dg_{\mathscr D(G)}\subset \cg_{\mathscr D(G)}$. In particular, the subset $\DG^{\heartsuit}$ is contained in the inverse image of the product $(\mt T\Dg_{\scr D(G)})^3\subset \cg_{\scr D(G)}^3$ along the composition
	$
	\AG=\AG_G\to\AG_{\scr D(G)}\xrightarrow{\op{ev}_{2q_1,2q_2,2q_3}}(\mt T\cg_{\scr D(G)})^3.
	$
	Arguing as in the semi-simple case, we conclude the proof.\\
	\fbox{$g=2$} Assume $G$ semi-simple. Arguing as above, it can be shown that for generic a generic triple $(q_1,q_2,q_3)\in C^3$, the cokernel of \eqref{E:ev} is isomorphic to the dual of
	$$
	H^0(C,\omega^{-1}(2q_1+2q_2+2q_3))^{\oplus \#\{d_i=2\}}.
	$$
	Moreover, this vector space is isomorphic to cokernel of the evaluation
	$$
	\bigoplus_{i\text{ s.t.} d_i=2}\AG_i=H^0(C,\omega^2)^{\#\{d_i=2\}}\xrightarrow{\op{ev}_{2q_1,2q_2,2q_3}}(\mt T\cg_1)^3,
	$$
	where $\cg_1$ is the linear subspace of $\cg$, where $\mathbb G_m$ acts with weight $2$. We divide the proof in three steps.
	\begin{enumerate}[(a)]
		\item $G$ almost-simple of rank $1$, i.e. $G=\op{SL}_2,\op{PGL}_2$. In this case, $\cg=\cg_1=\Spec k[t]$ and $\Dg=\{t=0\}$. In this case the evaluation \eqref{E:ev} is injective. Hence, the only section in $\DG^{\heartsuit}$ is the zero-section. Since $\dim\AG=3$, we have the assertion.
		\item $G$ almost-simple of rank at least two. It is enough to show that the projection $\mt T\Dg\subset\mt T\cg\to \mt T \cg_1=\Spec(k[p_1,dp_1])$ is surjective. First observe that the projection $\Dg\to \cg_1$ is $\mathbb G_m$-equivariant. Then, it must be surjective. Otherwise, we should have $\Dg\subset\{p_1=0\}$, which is false by Lemma \ref{R:tang-cone-disc} and the assumptions on the rank of $G$. A similar argument shows the surjectivity of the projection $\mt T\Dg\to \mt T\cg_1$.
		\item $G$ semi-simple. Up to isogeny, the group $G$ decomposes as a product $G_1\times \cdots\times G_s$, with $G_i$ almost-simple group for any $i$. Since the curve is irreducible, any section in $\DG^{\heartsuit}$ has image in some irreducible component $\Dg_{i,\omega}$ of the discriminant divisor. Hence, $\DG^{\heartsuit}$ is the union of the loci of sections with image in $\Dg_{i,\omega}$ for any $i\in\{1,\ldots,n\}$. By definition of the components $\Dg_i$, there exists a $j\in \{1,\ldots,s\}$ such that $\Dg_i$ is an irreducible component of the pull-back (along the obvious projection) of the discriminant divisor $\Dg_{G_j}\subset \cg_{G_{j}}$. In particular, the subset $\DG^{\heartsuit}$ is contained in the union of the inverse images of the products $(\mt T\Dg_{G_j})^3\subset (\mt T\cg_{G_j})^3$ along the compositions
		$
		\AG=\AG_G\to\AG_{G_j}\xrightarrow{\op{ev}_{2q_1,2q_2,2q_3}}(\mt T\cg_{G_j})^3,
		$
		for any $j=1,\ldots,s$. Hence, point (c) follows from the previous points.
	\end{enumerate}
The case $G$ reductive follows by arguing exactly as in the case $g\geq 3$.
\end{proof}

\begin{lem}\label{L:int-sect}Let $\mathfrak S$ be a $\mathbb G_m$-invariant irreducible subvariety of $\Dg$. The following facts hold.
\begin{enumerate}[(i)]
	\item\label{L:int-sect1} The locus of pairs $(p,a)\in C\times\AG$, such that $a(p)\in \mathfrak S_{\omega}:=\mathfrak S\times^{\mathbb G_m}\omega^*$, is an irreducible variety of codimension $r-\dim \mathfrak S$ in $C\times\AG$. If $\dim \mathfrak S=\dim \Dg-1$, the projection in $\AG$ is an irreducible variety of codimension $1$.
	\item\label{L:int-sect2} The locus of triples $(p,q,a)\in C\times C\times\AG$, such that $a(p),a(q)\in\mathfrak S_{\omega}$, is an irreducible variety of codimension $2(r-\dim\mathfrak S)$ in $C\times C\times\AG$.
\end{enumerate}
\end{lem}

\begin{proof}\fbox{$(\ref{L:int-sect1})$} We denote the locus in the assertion by $V$. Consider the morphism of vector bundles
	$
	\operatorname{ev}:C\times\AG \to\cg_{\omega}
	$.
	Since the line bundle $\omega^j$ is base-point-free for any positive $j$, the evaluation at the fibers $\operatorname{ev_p}:\AG\to \cg\otimes \oo_{C,p}/\mathfrak m_{C,p}=\cg$ is a smooth and surjective morphism with integral fibers of dimension $\dim \AG-r$. In particular, the inverse image $\operatorname{ev}_{p}^{-1}(\mathfrak S)$ is an irreducible variety of dimension $\dim \mathfrak S+\dim\AG-r$. Note that $\operatorname{ev}_{p}^{-1}(\mathfrak S)$ is nothing but the fiber over the point $p$ of the composition $V\hookrightarrow C\times \AG\xrightarrow{\pi_1} C$. This implies that $V$ is an irreducible variety in $C\times\AG$ of dimension $\dim \mathfrak S+\dim \AG+1-r$. If $\dim \mathfrak S=\dim\Dg-1$, the composition $V\hookrightarrow C\times\AG\to \AG$ is generically quasi-finite (by Lemma \ref{L:abincuore}). Hence, the last assertion follows.\\
	\fbox{$(\ref{L:int-sect2})$} Since the composition
	$
	\AG=\AG_G\to\AG_{\scr D(G)}\xrightarrow{\op{ev}_{p,q}}\cg_{\scr D(G)}^2
	$
	is surjective for any pair of points in $C$, the assertion is obtained by arguing as in the previous point.
\end{proof}

\begin{lem}\label{L:mult-sect}Consider the open and connected subset $\mathfrak U_{i,\omega}:=\Dg_{\omega}^{\sm}\cap \Dg_{i,\omega}$ of $\Dg_{i,\omega}$. Then, the locus $U$ of pairs $(p,a)\in C\times\AG$, such that $a(C)$ meets $\mathfrak U_{i,\omega}$ at $a(p)$ with multiplicity $\geq 2$, is an irreducible locally-closed subset of codimension $2$ in $C\times\AG$. The projection in $\AG$ is an irreducible locally-closed subset of codimension $1$.
\end{lem}

\begin{proof}We first assume $G$ semi-simple. Let $F$ be the (irreducible) polynomial in $k[\tg]^W$ such that $\Spec(k[\tg]^W/F)=\Dg_{i}$. Consider the morphism
	\begin{equation*}
	C\times\AG\xrightarrow{\operatorname{ev}} \cg_{\omega}\xrightarrow{F}\omega^{\deg F}.
	\end{equation*}
	Since $G$ is semi-simple, all the degrees $d_1,\ldots,d_r$ are strictly bigger than $1$. In particular, the line bundle $\omega^{d_i}$ is very ample for any $i=1,\ldots,r$. Hence, the evaluation $\operatorname{ev}_{2p}:\AG\to \cg\otimes \oo_{C,p}/\mathfrak m_{C,p}^2$ is a smooth and surjective homomorphism with integral fibers of dimension $\dim \AG-2r$. Let $\cg^o$ be the open subset of $\cg$, where the morphism $F:\cg\to\AG^1$ is smooth. A section $a$ meets the smooth locus of $\Dg_{i,\omega}$ at $a(p)$ with multiplicity $\geq 2$ if and only if $\operatorname{ev}_{2p}$ factors through $\cg^{o}$ and $\operatorname{ev}_{2p}(a)$ is contained in the fiber over $[m_{C,p}^2]$ of the smooth morphism
	$$
	\cg^o\otimes \oo_{C,p}/\mathfrak m_{C,p}^2\xrightarrow{F_p} \oo_{C,p}/\mathfrak m_{C,p}^2.
	$$
	Observe that the locus $F^{-1}_p([m_{C,p}^2])$ is isomorphic to the tangent bundle $\mt T\Dg_{i,\omega}^{\sm}$ over the smooth locus of $\Dg_{i,\omega}$. In particular, it is isomorphic to an irreducible open subset of the (possibly reducible) affine variety of dimension $2(r-1)$
	$$
	\operatorname*{Spec}k[p_1,\ldots,p_r,dp_1,\ldots,dp_r]/(F,dF)\subset \cg\times \cg\cong \cg\otimes \oo_{C,p}/\mathfrak m_{C,p}^2.
	$$
	By definition, we have the inclusion $\mathfrak U_{i,\omega}\subset \mathfrak D_{i,\omega}^{\sm}$. Then, the set $\operatorname{ev}_{2p}^{-1}(\mt T\mathfrak U_{i,\omega})$ is an irreducible locally-closed subset of $\AG$ of dimension $\dim T\Dg_{i,\omega}^{\sm}+\dim \AG-2r=\dim \AG-2$.\\	
By definition, the set $\operatorname{ev}_{2p}^{-1}(\mt T\mathfrak U_{i,\omega})$ is the fiber over the point $p$ of the composition $U\hookrightarrow C\times\AG\xrightarrow{\pi_1} C$. Hence, $U$ is an irreducible locally-closed subset of $C\times\AG$ of dimension $\dim\AG-1$. Since the composition $U\hookrightarrow C\times\AG\xrightarrow{\pi_2}\AG$ is generically quasi-finite (by Lemma \ref{L:abincuore}), we have the last assertion. Now, assume $G$ reductive. The morphism $\pi_G:\tg_{G}\to\cg_{G}$ decomposes as follows
	$$
	(\Id,\pi_{\mathscr D(G)}):\tg_{\mathscr Z(G)}\oplus\tg_{\mathscr D(G)}\to\tg_{\mathscr Z(G)}\oplus\cg_{\mathscr D(G)},
	$$
	see Remark \ref{R:decom-pi}. Furthermore, the discriminant divisor $\Dg_G$ is the zero set of a polynomial defined by the hyperplane-roots. It means that $\Dg_{G}$ is the pull-back (along the obvious projection) of the discriminant divisor $\Dg_{\mathscr D(G)}\subset \cg_{\mathscr D(G)}$. In particular, $\mt T\Dg_G=\mt T\tg_{\mathscr Z(G)}\oplus \mt T\Dg_{\mathscr D(G)}$. Consider the composition
	$$
	\AG_G\xrightarrow{D}  \AG_{\mathscr D(G)}\xrightarrow{\operatorname{ev_{2p}}}\mt\cg_{\mathscr D(G)}.
	$$
	By previous case, we know that the image $D(U)$ has codimension $1$ in $C\times \AG_{\mathscr D(G)}$. Since $D$ is linear and surjective and $U=D^{-1}(D(U))$, we have that $U$ has codimension $1$ in $C\times\AG_G$. The last assertion follows because the composition $U\hookrightarrow C\times\AG_G\xrightarrow{\pi_2}\AG_G$ is generically quasi-finite (by Lemma \ref{L:abincuore}).\\
\end{proof}

We are finally ready for proving the proposition and its corollary.

\begin{proof}[Proof of Proposition \ref{P:gen-discr}] The first part is exactly the content of \cite[Lemma 4.7.3]{NGO10}. We now focus on the description of the irreducible components of $\DG^{\diamondsuit}$. Consider the loci:
	\begin{enumerate}[(a)]
		\item $U\subset C\times \AG$ of pairs $(p,a)$ such that $a(C)$ meets $\Dg_{\omega}^{\sm}$ at $a(p)$ with multiplicity at least $2$,
		\item $V\subset C\times\AG$ of pairs $(p,a)$ such that $a(C)$ meets $\Dg_{\omega}^{\sing}$.
	\end{enumerate} 
	By definition $\DG^{\diamondsuit}=\pi_2(U)\cup\pi_2(V)$. 
	By Lemma \ref{L:int-sect}, the locus $\widetilde V:=\pi_2(V)$ is the union of (distinct) irreducible closed subsets $\widetilde V_1,\ldots,\widetilde V_n$ of codimension $1$ in $\AG$. Where $\widetilde V_i$ is the locus of sections meeting the irreducible component $\Dg^{\sing}_{\omega,i}$ of the singular locus $\Dg_{\omega}^{\sing}$. In other words, $\widetilde{V}_i=\DG^{\operatorname*{extra}}_i$ for $i=1,\ldots n$.\\	
	Similarly, by Lemma \ref{L:mult-sect}, the closure of the projection of $U$ into $\AG$, i.e. $\widetilde{U}:=\overline{\pi_2(U)}=\pi_2(\overline{U})$, is the union of (distinct) irreducible closed subsets $\widetilde U_1,\ldots,\widetilde U_m$ of codimension $1$ in $\AG$. Where $\widetilde U_i$ is the closure of the locus of sections meeting the smooth locus $\Dg^{\sm}_{\omega}$ through the irreducible component $\Dg_{\omega,i}$ with multiplicity $\geq 2$ at least in one point. It remains to show that $\widetilde{U}_i=\DG^{\operatorname*{ab}}_i$ for $i=1,\ldots,n$. By Lemma \ref{L:claim}, any section in $\DG^{\ab}_i$ is the limit of a section in $\pi_2(U_i)$, i.e. $\DG^{\operatorname*{ab}}_i\subset \widetilde{U}_i$. Conversely, let $a$ be a section in $\widetilde{U}_i$. We distinguish two cases. First case: $a\in\AG^{\heartsuit}\cap \widetilde{U}_i$. Then the condition $a\in\DG_i^{ab}$ follows by Proposition \ref{P:descr-reg} and the upper-semicontinuity of the function $\delta:\AG^{\heartsuit}\to\mathbb Z_{\geq 0}$ (see Remark \ref{R:delta-upsc}). Second case: $a\in \widetilde{U}_i$ and $a(C)\subset \Dg_{\omega}$. Then, the inclusion follows by Lemma \ref{L:abincuore}. 
\end{proof}
Corollary \ref{C:gen-discr} has been already proved in \cite[Lemma C.0.2]{Li20}, when $G$ is semi-simple and $\omega=\mt L$ is a line bundle of degree greater than $2g$. Since our assumptions are slightly different, we include a self-contained proof.

\begin{proof}[Proof of Corollary \ref{C:gen-discr}] \fbox{$(\ref{C:gen-discr1})$} Fix $i=1,\ldots,m$. By Proposition \ref{P:gen-discr}, a generic point in $\DG_i^{\ab}$ corresponds to those sections meeting the irreducible divisor $\Dg_{i,\omega}$ outside of the singular locus of the entire discriminant divisor $\Dg_{\omega}$. Then, it is enough to show that
		\begin{enumerate}[(a)]
			\item the set of pairs $(p,a)\in C\times\AG$, such that $a(C)$ meets the smooth locus $\Dg^{\sm}_{\omega}$ of the discriminant divisor with multiplicity greater or equal than $3$ at $a(p)$, has codimension at least three in $C\times\AG$;
			\item the set of sections $a\in\AG$, such that $a(C)$ meets the smooth locus $\Dg^{\sm}_{\omega}$ of the discriminant divisor with multiplicity greater than $2$ in at least two points, has codimension at least 2 in $\AG$.
		\end{enumerate}
		Point (a): we denote the set in the assertion by $U$. The fiber $U_p$ over a point $p\in C$ along the projection $U\hookrightarrow C\times \AG\to C$ is canonically equal to an open subset of the inverse image of $[\mathfrak m_{C,p}^3]$ along the composition $$\AG=\AG_G\to\AG_{\mathscr D(G)}\xrightarrow{\operatorname{ev}_{3p}} \cg_{\mathscr D(G)}\otimes \oo_{C,p}/\mathfrak m_{C,p}^3\xrightarrow{\prod_{\alpha\in\Phi}d\alpha}\oo_{C,p}/\mathfrak m_{C,p}^3,$$ Using Riemann-Roch, we see that the evaluation $\operatorname{ev}_{3p}$ is surjective (and smooth) for any $p\in C$, if $g\geq 3$ and for $p\in C$ generic, if $g=2$. In particular, for a generic point $p$, the fiber $U_p$ has codimension at least $3$ in $\AG$. Hence, the set $U$ has codimension $3$ in $C\times\AG$.	Point (b): it follows by arguing as in the proof of the second statement of Lemma \ref{L:abincuore}.\\
\fbox{$(\ref{C:gen-discr2})$} Fix $i=1,\ldots,n$. By Proposition \ref{P:gen-discr}, a section $a$ is in $\widetilde{\DG}_i^{\operatorname{extra}}:=\DG_i^{\operatorname{extra}}\setminus(\bigcup_{j\neq i}\DG_j^{\operatorname{extra}}\cup\bigcup_j\DG_j^{\mathrm{ab}})$ if and only if $a(C)$ meets the singular locus of the discriminant divisor $\Dg_{\omega}^{\sing}$ only along the irreducible subvariety $\Dg^{\sing}_{i,\omega}$ and it meets the smooth locus $\Dg_{\omega}^{\sm}$ transversally. Hence, for concluding the proof, we need to show that a generic section in $\widetilde{\DG}_i^{\operatorname{extra}}$ meets the singular locus $\Dg^{\sing}_{\omega}$ in a unique point and transversally. Thus, it is enough to show that
		\begin{enumerate}[(a)]
			\item the set of pairs $(p,a)\in C\times\AG$, such that $a(C)$ meets the singular locus $(\Dg^{\sing}_{\omega})^{\sing}$ of $\Dg^{\sing}_{\omega}$ at $a(p)$, has codimension at least three in $C\times\AG$;
			\item the set of triples $(p,q,a)\in C\times C\times\AG$, such that $a(p),a(q)\in\Dg^{\sing}_{\omega}$, has codimension at least four in $C\times C\times \AG$;
			\item the set of pairs $(p,a)\in C\times\AG$, such that $a(C)$ meets the smooth locus $(\Dg^{\sing}_{\omega})^{\sm}$ of $\Dg^{\sing}_{\omega}$ at $a(p)$ with multiplicity at least two, has codimension at least four in $C\times \AG$.
		\end{enumerate}
Point (a) follows from Lemma \ref{L:int-sect}$(\ref{L:int-sect1})$. Point (b) follows from Lemma \ref{L:int-sect}$(\ref{L:int-sect2})$. Point (c):  we denote the set in the assertion by $U$. The fiber $U_p$ over a point $p\in C$ along the projection $U\hookrightarrow C\times \AG\to C$ is canonically equal to the inverse image of the tangent bundle $\mathcal T ((\Dg^{\sing}_{\omega})^{\sm})$ of $(\Dg^{\sing}_{\omega})^{\sm}$ in $\mathcal T\cg_{\mathscr D(G)}$ along the composition $$\AG=\AG_G\to\AG_{\mathscr D(G)}\xrightarrow{\operatorname{ev}_{2p}} \cg_{\mathscr D(G)}\otimes \oo_{C,p}/\mathfrak m_{C,p}^2=\mathcal T\cg_{\mathscr D(G)}.$$Using Riemann-Roch, we see that the composition above is surjective, for any $p$. Since $\mathcal T ((\Dg^{\sing}_{\omega})^{\sm})$ is a locally closed subscheme of codimension at least four in $\mathcal T\cg_{\mathscr D(G)}$, the fiber $U_p$ has codimension four in $\AG$. Hence, the set $U$ has codimension at least four in $C\times\AG$.
\end{proof}

\subsection{The subset $\AG^{\ab}\subset\AG$.}
In the previous section, we have defined the open subset $\AG^{\diamondsuit}$ of $\AG$ and we have seen that it parameterizes smooth cameral curves. In this subsection, we will define another open subset in $\AG$, containing $\AG^{\diamondsuit}$, and we will show that it parameterizes Hitchin fibers which are isomorphic to abelian stacks.

We denote by $\AG^{\ab}$ the subset in $\AG$, whose sections satisfy condition $(\ref{P:descr-reg4})$ of Proposition \ref{P:descr-reg}, i.e.
$$
\AG^{\ab}:=\{a\in\AG\,|\,\operatorname{Im}(da_p)\not\subset\mt C_{a(p)}\Dg_{\omega} \text{ for any }p\in a^{-1}(\Dg_{\omega})\subset C\}.
$$
We denote $\DG^{\ab}:=\AG\setminus\AG^{\ab}$ its complementary subset. The goal of the section is proving the following proposition.

\begin{prop}\label{P:uni-stack}Fix $\delta\in\pi_1(G)$. The following facts hold.
	\begin{enumerate}[(i)]
		\item\label{P:uni-stack1} For any section $a\in\AG^{\ab}$, the Hitchin fiber $\hit_G^\delta(a)$ is an abelian stack.
		\item\label{P:uni-stack2} For a generic section $a\in\DG^{\ab}$, there exists a finite birational morphism of stacks
		\begin{equation*}\label{E:uni-stack}
		\mathbb P^1\stackrel{\mathbb G_m}{\times}\jator{a}^{\delta}\to \hit_G^{\delta}(a),
\end{equation*}
where the left-hand side is a locally trivial $\mathbb P^1$-fibration over the abelian stack $\jator{a}^{\delta,\ab}$.
	\end{enumerate}
\end{prop}

Before proving this, we need the following fact.

\begin{prop}\label{P:gen-ab} The closed subset $\DG^{\ab}$ is the union of the subsets
	$$
	\DG^{\ab}_{1},\ldots,\DG^{\ab}_{m},
	$$
	where $\DG^{\ab}_{i}$ is the irreducible hypersurface in $\AG$ described in Proposition \ref{P:gen-discr}. Furthermore, we have the chain of inclusions $\AG^{\diamondsuit}\subset\AG^{\ab}\subset\AG^{\heartsuit}$.
\end{prop}

\begin{proof} The first statement follows from the description of the irreducible components of $\DG^{\diamondsuit}$ (see Proposition \ref{P:gen-discr}). In particular, we get $\DG^{ab}\subset \DG^{\diamondsuit}$ and, so, the first inclusion in the second statement holds. The second inclusion follows by Lemma \ref{L:abincuore}.
\end{proof}

\begin{rmk}By\label{abincuore} Proposition \ref{P:gen-ab}, we know that $\AG^{\ab}\subset\AG^{\heartsuit}$. In particular, $a\in\AG^{\ab}$ if and only $a\in\AG^{\heartsuit}$ and it satisfies one of the equivalent conditions of Proposition \ref{P:descr-reg}.
\end{rmk}

\begin{rmk}Proposition \ref{P:gen-ab} gives a criterion for counting the number of irreducible components of the divisor $\DG^{\ab}$. In particular, we may deduce the following facts.
	\begin{enumerate}[(a)]
		\item The divisor $\DG^{\ab}$ is empty if and only if $G$ is abelian.
		\item The divisor $\DG^{\ab}$ is irreducible and non-empty if and only if the derived subgroup $\mathscr D(G)$ is an almost-simple group of type ADE.
		\item Assume $\mathscr D (G)$ being a non-trivial almost-simple group. The divisor $\DG^{\ab}$ is irreducible, if $\mathscr D (G)$ is of type ADE and reducible with two components, otherwise.
	\end{enumerate}
\end{rmk}
\begin{rmk}Propositions \ref{P:gen-discr} and \ref{P:gen-ab} together explain how far is the inclusion $\AG^{\diamondsuit}\subset \AG^{\ab}$ from being an equality. In particular, the following facts are equivalent.
	\begin{enumerate}[(a)]
		\item The inclusion $\AG^{\diamondsuit}\subsetneq \AG^{\ab}$ is strict.
		\item The discriminant divisor $\Dg\subset\cg$ is singular.
		\item the group $G$ has semi-simple rank greater or equal than $2$, i.e. $\mathscr D(G)\neq \op{SL}_2,\op{PGL}_2$.
	\end{enumerate}
\end{rmk}

Thanks to Proposition \ref{P:gen-ab}, we can now show the main result of the subsection.

\begin{proof}[Proof of Proposition \ref{P:uni-stack}] Point $(\ref{P:uni-stack1})$ follows by Remark \ref{abincuore}. We now show the point $(\ref{P:uni-stack2})$. Let $a$ be a generic section in $\DG^{\ab}$. By Proposition \ref{P:gen-ab} and Corollary \ref{C:gen-discr}$(\ref{C:gen-discr1})$, we may assume that $a(C)$ meets $\Dg_{\omega}$ transversally, with the exception of a unique point in $p\in C$, where it meets the smooth locus $\Dg^{\sm}_{\omega}$ along the irreducible component $\Dg_{\omega,i}$ with multiplicity 2 at the point $\alpha(p)$. In particular, $a\in\AG^{\heartsuit}$. By Product formula (= Theorem \ref{T:prod}), we have a homeomorphism of stacks
	\begin{equation}\label{E:prod-generic}
		\prod_{q\in B_a}\operatorname{Spr}^{\operatorname{red}}_q(a)\stackrel{\prod_{p\in B_a}\japtor{a,q}^{\operatorname{red}}}{\times}\jator{a}\to \hit_G(a),
	\end{equation}
where $B_a$ is the branch locus of the cameral cover $C_a$ corresponding to $a$. Arguing as in the proof of Proposition \ref{P:descr-reg}, we get that all the elements in the affine Springer fiber over the point $q\neq p$ are regular, i.e. $\operatorname{Spr}^{\operatorname{red}}_q(a)=\operatorname{Spr}^{\operatorname{red},\reg}_q(a)$ for any $q\in B_a\setminus\{p\}$. The latter is a torsor with respect to the group $\japtor{a,p}^{\operatorname{red}}$. In particular, the homeomorphism \eqref{E:prod-generic} becomes
\begin{equation}\label{E:prod-generic2}
	\operatorname{Spr}^{\operatorname{red}}_p(a)\stackrel{\japtor{a,p}^{\operatorname{red}}}{\times}\jator{a}\to \hit_G(a).
\end{equation}
We now describe the affine Springer fiber $\operatorname{Spr}^{\operatorname{red}}_p(a)$ following \cite[Section 8.3]{NGO10}. There exists a co-root $\alpha^{\vee}$ such that
$$
\operatorname{Spr}^{\operatorname{red}}_p(a)=\bigsqcup_{j\in\Lambda/\langle\alpha^{\vee}\rangle_{\mathbb Z}}X_j,
$$
where $\langle\alpha^{\vee}\rangle_{\mathbb Z}$ is the sublattice of the cocharacter lattice $\Lambda:=\operatorname{Hom}(\mathbb G_m,T)$ of the maximal torus $T\subset G$ and $X_j=\bigcup_{i\in<\alpha^{\vee}>_{\mathbb Z}}\mathbb P^1$ is an infinite chain of rational lines. Furthermore, there exists an exact sequence of groups
$$
1\to\mathbb G_m\to\japtor{a,p}^{\operatorname{red}}\to\Lambda\to 1,
$$
where $\mathbb G_m$ acts on the affine Springer fibers by scalar multiplication. Fix a projective line $\mathbb P^1$ in $X_0$, we get a finite birational morphism of stacks
\begin{equation}\label{E:mor-comp}
	\mathbb P^1\times^{\mathbb G_m}\jator{a}\to \operatorname{Spr}^{\operatorname{red}}_p(a)\stackrel{\japtor{a,p}^{\operatorname{red}}}{\times}\jator{a}.
\end{equation}
Then, the composition of the morphisms \eqref{E:prod-generic2} and \eqref{E:mor-comp} gives the required morphism.
\end{proof}

\subsection{The moduli space $\hitgd_G$ of semistable $G$-Higgs bundles.}\label{SS:ss}
Here, we present the analogous results of the previous sections for the (good) moduli space of semistable bundles.

\begin{defin}A $G$-bundle $E$ is \emph{simple} if $\Aut(E)=\mathscr Z(G)$. A $G$-bundle $E$ is \emph{stable} (resp. \emph{semistable}) if for any reduction $F$ of $E$ to any parabolic subgroup $P\subset G$, we have 
	$$
	\deg(F\times^{\op{Ad},P}\mathfrak p)<0
	, \quad(\text{resp. }  \deg(F\times^{\op{Ad},P}\mathfrak p)\leq 0),$$ where $\op{Ad}:P\to \op{GL}(\mathfrak p)$ is the adjoint action of $P$ on its Lie algebra $\mathfrak p$.\\
A $G$-Higgs bundle $(E,\theta)$ is \emph{stable} (resp. \emph{semistable}) if for any reduction $F$ of $E$ to any parabolic subgroup $P\subset G$ such that $\theta\in H^0(C,(F\times^{\op{Ad},P}\mathfrak p)\otimes\omega )\subset H^0(C,\ad{E})$, we have
$$
\deg(F\times^{\op{Ad},P}\mathfrak p)<0
, \quad(\text{resp. }  \deg(F\times^{\op{Ad},P}\mathfrak p)\leq 0).$$
\end{defin}
The (semi)stable locus $\BG_{G}^{(s)s}$ is an open substack of the moduli stack $\BG_G$. It is well-known that the moduli stack $\BG_G^{ss}$ admits a good moduli space $\overline{\mt M}_G$ (or $\overline{\mt M}_G(C)$, when the reference to the curve is needed). The next theorem resumes the main properties of this moduli space.

\begin{teo}\label{T:bg-fac}The following facts hold.
\begin{enumerate}[(i)]
	\item The moduli stack $\BG_G^{\operatorname{(s)s}}$ of (semi)stable $G$-bundles is an open and dense substack of finite type over $k$ in $\BG_G$.
	\item The good moduli morphism $\BG_G^{\operatorname{ss}}\to \overline{\mt M}_G$ gives a bijection of connected components $\pi_0(\BG_G)=\pi_0(\overline{\mt M}_G)=\pi_1(G)$.
	\item For any $\delta\in \pi_1(G)$, the connected component $\overline{\mt M}^{\delta}_G$ of semistable $G$-bundle of degree $\delta$ is an irreducible normal projective variety over $k$.
	\item For any $\delta\in \pi_1(G)$, the subset $\mt M^{\delta}_G$ of stable and simple $G$-bundles of degree $\delta$ is contained in the smooth locus of the connected component $\overline{\mt M}^{\delta}_G$.
\end{enumerate}
\end{teo}

On the other hand, the (semi)stable locus $\hit_G^{\operatorname{(s)s}}$ is an open substack of the entire moduli stack $\hit_G$. It is well-known that the moduli stack $\hit_G^{\operatorname{ss}}$ admits a good moduli space $\hitgd_G$. The Hitchin morphism factors through the good moduli space
\begin{equation}\label{E:hit-fac}
h:\hit_G\to \hitgd_G \to \AG.
\end{equation}
With abuse of notation, we denote with the same symbol $h$ the second morphism in the composition \eqref{E:hit-fac} and we call it \emph{Hitchin morphism}. Analogously, we call \emph{Hitchin fiber over $a$} the inverse image $\hitgd_G(a):=h^{-1}(a)\subset\hitgd_G$ with respect to the section $a\in \AG$. 
\begin{teo}\label{T:ss-higgs} The following facts hold.
	\begin{enumerate}[(i)]
		\item\label{T:ss-higgs1} The moduli stack $\hit_G^{\operatorname{(s)s}}$ of (semi)stable $G$-Higgs bundles is an open and dense substack of finite type over $k$ in $\hit_G$.
		\item\label{T:ss-higgs2} The good moduli morphism $\hit_G^{\operatorname{(s)s}}\to \hitgd_G$ gives a bijection of connected components $\pi_0(\hit_G)=\pi_0(\hitgd_G)=\pi_1(G)$.
		\item\label{T:ss-higgs3} Fix $\delta\in\pi_1(G)$. The Hitchin morphism $h^{\delta}:\hitgd_G^{\delta}\to \AG$ is projective and faithfully flat with geometrically connected fibers of dimension $\dim\AG=\dim G(g-1)+\dim\mathscr Z(G)$.
	\end{enumerate}
\end{teo}

\begin{proof}The first two points are consequences of the third one. With the exception of the flatness, a proof of the third point (in this generality) can be found in \cite{Fa93}. The flatness follows from the flatness of the stacky Hitchin morphism (see Proposition \ref{P:hit-prop}\label{P:hit-prop4}) and the fact that the good moduli morphism $\hit_{G}^{\operatorname{ss}}\to\hitgd_G$ preserves the flatness (see \cite[Theorem 4.16(ix)]{AL}).
\end{proof}
\begin{prop}\label{T:geom-fiber}Fix $\delta\in\pi_1(G)$. The following facts hold.
	\begin{enumerate}[(i)]
		\item\label{T:i1:geom-fiber} For any $a\in \AG^{\ab}$, the Hitchin fiber $\hitgd_G^{\delta}(a)$ is an abelian variety.
		\item\label{T:i2:geom-fiber} For a generic $a\in\DG^{\ab}$, the Hitchin fiber $\hitgd_G^{\delta}(a)$ is a uniruled irreducible variety.
	\end{enumerate}
\end{prop}

\begin{proof}\fbox{$(\ref{T:i1:geom-fiber})$}
	Fix a section $a\in\AG^{\ab}$. By Proposition \ref{P:uni-stack}$(\ref{P:uni-stack1})$, we know that the stacky Hitchin fiber $\hit_G^{\delta}(a)$ is an abelian stack, i.e. $\hit_G^{\delta}(a)=[A/H]$, where $A$ is an abelian variety and $H$ is a multiplicative group acting trivially on $A$. In particular, the canonical morphism $[A/H]\to A$ makes $A$ a good moduli space of the stacky Hitchin fiber. More precisely, it is a tame moduli space (in the sense of Alper, see \cite[Definition 7.1]{AL}). Consider the open substack $\hit_G^{\delta, ss}(a)\subset\hit_G^{\delta}(a)$ of semistable $G$-Higgs bundles. By the above discussion, we must have $\hit_G^{\delta, \operatorname{ss}}(a)=[A^{\operatorname{ss}}/H]$, where $A^{\operatorname{ss}}$ is an open subset of $A$. By the universal property of the good moduli space, we have
	$
	\hitgd_G^{\delta}(a)=A^{\operatorname{ss}}.
	$
	Since the Hitchin morphism is projective (see Theorem \ref{T:ss-higgs}), we must have $A^{\operatorname{ss}}=A$. Equivalently, any $G$-Higgs bundle in $\hit_G^{\delta}(a)$ is semistable. This proves the point $(\ref{T:i1:geom-fiber})$.\\
\fbox{$(\ref{T:i2:geom-fiber})$} Let $a$ be a generic section in $\DG^{ab}$. By Corollary \ref{C:irr-cam}$(\ref{C:irr-cam2})$, we may assume that the corresponding cameral curve is integral. By Lemma \ref{L:irrthenst} below, all the $G$-Higgs bundles in the stacky Hitchin fiber $\hit_G^{\delta}(a)$ are stable. We then have the following commutative diagram
\begin{equation}\label{E:comm-uniruled}
	\xymatrix{
		\mathbb P^1\times^{\mathbb G_m}\jator{a}^{\delta}\ar[d]\ar[r]^{\nu}& \hit_G^{\delta}(a)\ar[d]\\
		\widetilde{\hitgd_G}^{\delta}(a)\ar[r]^{\overline{\nu}}& \hitgd_G^{\delta}(a)
	}
\end{equation}
where:
\begin{enumerate}[(a)]
	\item $\nu$ is the finite birational morphism of stacks (see Proposition \ref{P:uni-stack}),
	\item  $\widetilde{\hitgd_G}^{\delta}(a)$ is the tame moduli space of $\mathbb P^1\times^{\mathbb G_m}\jator{a}^{\delta}$, which exists because the latter is a locally trivial $\mathbb P^1$-fibration over the abelian stack $\jator{a}^{\delta,\ab}$ (see \emph{loc.cit.}),
	\item the vertical maps are the canonical maps onto the tame moduli spaces,
	\item the existence of the morphism $\overline{\nu}$ follows by the universal property of good moduli spaces. 
\end{enumerate}
The morphism $\overline{\nu}$ is birational, because its counterpart $\nu$ is birational. Furthermore, the moduli space $\widetilde{\hitgd_G}^{\delta}(a)$ is a locally trivial $\mathbb P^1$-fibration over the good moduli space of the abelian stack $\jator{a}^{\delta,\ab}$, which is an abelian variety. In particular, the good moduli space $\hitgd_G^{\delta}(a)$ is a uniruled variety.
\end{proof}

\begin{lem}\label{L:irrthenst}Let $a$ be a section in $\AG$ such that the corresponding cameral curve is integral, then any $G$-Higgs bundle in the Hitchin fiber $\hit_G(a)$ is stable.
\end{lem}

\begin{proof}We will show a stronger result: if the cameral curve $C_a$ is integral, then any $G$-Higgs bundle $(E,\theta)$ in the Hitchin fiber $\hit_G(a)$ does not have a reduction to a parabolic subgroup $P\subset G$ compatible with $\theta$. 
	
	We proceed by contradiction. Let $(E,\theta)$ be a $G$-Higgs bundle in $\hit_G(a)$ with a reduction $(F,\theta)$ to a parabolic subgroup $P\subset G$ compatible with the section $\theta$. Let $M$ be a Levi-subgroup of $P$ and $W_M\subset W$ be the Weyl group of $M$. Without loss of generality, we may assume that $\tg\subset\mathfrak m\subset \mathfrak g$. Note that having a $G$-Higgs bundle is the same as having a morphism of stacks $C\to [\g_{\omega}/G]$. Then, having the reduction to a parabolic subgroup is equivalent to have a commutative diagram of stacks 
	\begin{equation}\label{E:8}
			\xymatrix{
				&[\mathfrak p_{\omega}/P]\ar[r]\ar[d]& [\g_{\omega}/G]\ar[d]\\
				C\ar[ru]\ar[r]^b\ar@/_2pc/[rr]^a& \cg_{M,\omega}\ar[r]&\cg_{G,\omega}
				}
		\end{equation}
	The first vertical arrow follows from Theorem \ref{T:konstant}$(\ref{T:konstant6})$. In particular, the section $a$ factors through $\cg_{M,\omega}$. By definition of cameral cover, we have a commutative diagram of schemes
	\begin{equation}\label{E:9}
			\xymatrix{
				C_b\ar@{->>}[d]^{\pi_b}\ar@{^{(}->}[r]^\iota&	C_a\ar@{^{(}->}[r]\ar@{->>}[d]&\tg_{\omega}\ar@{->>}[d]\\
				C\ar@{^{(}->}[r]^-\sigma\ar@{=}[rd]&	\cg_{M,\omega}\times_C\cg_{G,\omega}\ar@{->>}[d]^q\ar@{^{(}->}[r]&\cg_{M,\omega}=\tg_\omega/W_M\ar@{->>}[d]\\
				&	C\ar@{^{(}->}[r]^-a&\cg_{G,\omega}=\tg_\omega/W\\
				}
	\end{equation}
	where the squares are cartesian and the composition of the vertical arrows in the middle is the cameral cover $\pi_a:C_a\to C$ corresponding to $a$. Now observe that $\pi_b$ is a cameral cover  (with respect to the reductive group $M$) given by the section $b$ in Diagram \eqref{E:8}. In particular, $C_b$ is a projective curve over $k$. On the other hand, the morphism $\sigma$ is a closed embedding (because it is a section of the morphism $q$). Hence, $\iota$ is a closed embedding. It follows from the construction that $\iota(C_b)\subsetneq C_a$: indeed, for each $p\in C$, the fiber $\pi_b^{-1}(p)$ is a zero-dimensional scheme of lenght $|W_M|$ and the fiber $\pi_a^{-1}(p)$ over $p$ of the composition of the vertical arrows in the middle of Diagram \eqref{E:9} is a zero-dimensional scheme of lenght $|W|$. Since $C_a$ is reduced, we may choose $p$ such that $\pi_a^{-1}(p)$ is reduced. Hence, the inclusion $\iota(\pi_b^{-1}(p))\subset \pi_a^{-1}(p)$ is strict because $W_M\subsetneq W$. 
	
	In particular, $C_a$ contains the proper subcurve $\iota(C_b)$. Then, $C_a$ is not irreducible and so we reach a contradiction.
\end{proof}

\begin{rmk}From the proof of Theorem \ref{T:geom-fiber}, one can also deduce that the $G$-Higgs bundles in the stacky Hitchin fibers over $\AG^{\ab}$ are always stable. Furthermore, one can also show that the morphism $\overline{\nu}:\widetilde{\hitgd_G}(a)\to\hitgd_G(a)$ in \eqref{E:comm-uniruled} is the normalization of $\hitgd_G(a)$.
\end{rmk}

\begin{cor}\label{C:geom-fiber}Fix $\delta\in\pi_1(G)$. The following facts hold.
	\begin{enumerate}[(i)]
		\item\label{C:i1:geom-fiber} The cotangent bundle $\thig_G^\delta$ of the moduli space $\mt M_G^\delta$ of simple and stable $G$-bundles of degree $\delta$ is an open and dense subset of the moduli space $\hitgd_G^\delta$.
		\item\label{C:i2:geom-fiber} For any $a\in \AG^{\ab}$, the intersection $\thig_G^\delta\cap \hitgd_G(a)$ is a non-empty open subset of the abelian variety $\hitgd_G^\delta(a)$.
		\item\label{C:i3:geom-fiber} For a generic $a\in\DG^{\ab}$, the intersection $\thig_G^\delta\cap \hitgd_G(a)$ is a non-empty open subset of the uniruled variety $\hitgd_G^\delta(a)$.
	\end{enumerate}
\end{cor}

\begin{proof} We remove the reference to the group $G$ from the notation. \fbox{$(\ref{C:i1:geom-fiber})$} By standard deformation theory, we know that the fiber of tangent bundle $T\mt M\to \mt M$ over a simple and stable $G$-bundle $E$ is canonically identified with $H^1(C,\ads(E))$. By Serre duality, the fiber of the cotangent bundle $\thig\to\mt M$ over the same point is the vector space $H^0(C,\ad{E})$. It is easy to check that if a $G$-bundle $E$ is stable, then $(E,\theta)$ is a stable $G$-Higgs bundle for any section $\theta\in H^0(C,\ad{E})$. In particular, $\thig$ is an open subset in $\hitgd$. The density follows because the inclusion $\thig\subset \hitgd$ gives a bijection between the connected components, i.e. $\pi_0(\mt M)=\pi_0(\thig)=\pi_0(\hitgd)=\pi_1(G)$. \fbox{$(\ref{C:i2:geom-fiber})+(\ref{C:i3:geom-fiber})$} By previous point and Theorem \ref{T:geom-fiber}, we have the statement, except for the non-emptiness of the intersection. This last property follows by \cite[Theorem II.5(ii)]{Fa93}.
\end{proof}

\begin{prop}\label{P:irr-hd}Let $\DG^{\ab}_i$ be an irreducible component of $\DG^{\ab}$, for some $i=1,\ldots,m$. Then the inverse image $(h^{\delta})^{-1}(\DG^{\ab}_i)$ is an irreducible divisor of  $\hitgd_G^{\delta}$.
\end{prop}

\begin{proof} To make the notation easier to follow, we write $\DG$ instead of $\DG^{\ab}_i$. Since the Hitchin morphism $h^{\delta}$ is faithfully flat (see Theorem \ref{T:ss-higgs}$\eqref{T:ss-higgs3}$), the inverse image $(h^{\delta})^{-1}(\DG)$ is a subscheme of $\hitgd_G^{\delta}$ of pure codimension one. It remains to prove the irreducibility. Let $\DG^*\subset \DG$ be the open subset of generic sections in the sense of Proposition \ref{T:geom-fiber}$(\ref{T:i2:geom-fiber})$. The restriction $(h^{\delta})^{-1}(\DG^{*})\to \DG^{*}$ is a projective morphism with geometrically integral fibers. In particular, $(h^{\delta})^{-1}(\DG^*)$ is an irreducible open subset of the divisor $(h^{\delta})^{-1}(\DG^{\ab})$. To conclude the proof, it is enough to show that $(h^{\delta})^{-1}(\DG^*)$ is also dense in the divisor $(h^{\delta})^{-1}(\DG^{\ab})$. Since $(h^{\delta})^{-1}(\DG)$ has pure codimension one, the statement is equivalent to showing that the complement 
	$$
	(h^{\delta})^{-1}(\DG)\setminus(h^{\delta})^{-1}(\DG^{*})=(h^{\delta})^{-1}(\DG\setminus\DG^*)
	$$
has codimension at least two in $\hitgd_G^\delta$. And this follows because $h^{\delta}$ have equi-dimensional fibers and $\DG\setminus\DG^*$ have codimension at least two in $\AG$.
\end{proof}

\begin{teo}\label{T:hit-disc}The inverse image $(h^{\delta})^{-1}(\DG^{\ab})$ is the closure of the union of projective rational curves in $\hitgd_G^\delta$. 
\end{teo}

\begin{proof}Let $R$ be the union of projective rational curves in $\hitgd_G$. We want to show that $(h^{\delta})^{-1}(\DG^{\ab})=\overline{R}.$\\
\fbox{$\supset$} It is enough to show $h^{\delta}(R)\subset \DG^{\ab}$. Let $X=\mathbb P^1$ be a projective line and $f:X\to\hitgd_G^\delta$ be a (finite) morphism. Since $X$ is integral and projective and $\AG$ is affine, the image $a:=h^\delta(f(X))\subset \AG$ must be a point. In particular, the curve $f(X)$ is contained in the Hitchin fiber $\hitgd_G^\delta(a)$. Since an abelian variety does not contain rational curves, we must have $a\in\DG^{\ab}$ (see Proposition \ref{T:geom-fiber}).\\
\fbox{$\subset$} By Proposition \ref{P:irr-hd}, it is enough to show $(h^\delta)^{-1}(\DG^{\ab,*})\subset R$, where $\DG^{\ab,*}\subset\DG^{\ab}$ is the open subset of generic sections in the sense of Proposition \ref{T:geom-fiber}$(\ref{T:i2:geom-fiber})$. By \emph{loc.cit.} $(h^\delta)^{-1}(\DG^{\ab,*})\to\DG^{\ab,*}$ is a family of irreducible uniruled varieties. In particular, $(h^\delta)^{-1}(\DG^{\ab,*})$ is a union of projective rational curves, i.e. $(h^\delta)^{-1}(\DG^{\ab,*})\subset R$.
\end{proof}

The same proof also gives the following corollary, which was already known for $G=\op{SL}_n$ \cite[Theorem 3.3]{BGMsl}, $G=\op{Sp}_{2n}$ \cite[Theorem 3.3]{BGMsym} and $G$ simply-connected of type either $E_6$ or $F_4$ \cite[Proposition 5]{Sanc}. Following \emph{loc.cit.}, we call \emph{Hitchin discriminant} the intersection $(h^{\delta})^{-1}(\DG^{\ab})\cap T^*\mt M_G^\delta$.

\begin{cor}\label{C:hit-disc}The Hitchin discriminant $(h^{\delta})^{-1}(\DG^{\ab})\cap T^*\mt M_G^\delta$ is the closure of the union of projective rational curves in $T^*\mt M_G^\delta$. 
\end{cor}

\begin{proof}The proof is a copy of the proof of Theorem \ref{T:hit-disc}. The only difference is that, instead of using Proposition \ref{T:geom-fiber}, we need to use Corollary \ref{C:geom-fiber} as reference.
\end{proof}

\section{Torelli Theorem.}\label{Sec:torelli}
In this section, we show a Torelli-type theorem for the moduli space of $G$-bundles. Our proof is a replica of the proof of Biswas, G\'omez and Mu\~noz in \cite{BGMsl}, for $G=\operatorname{SL}_n$, in \cite{BGMsym}, for $G=\operatorname{Sp}_{2n}$, and later adapted by Ant\'on Sancho  \cite{Sanc} for the cases $G=\mathbb E_6^{\operatorname{sc}}, \mathbb F_4$. We also remark that this result has been already established with different methods and in greater generality in \cite{BH12}. However, later in the paper, we will need some of the results of this section (especially Corollary \ref{C:var-l}) for the computation of the automorphisms of the moduli spaces $\overline{\mt M}_G$.

Here, we will always assume $G$ almost-simple and $C$ connected, smooth projective curve of genus $g\geq 2$, with the only exception of the cases where $g=2$ and $G$ is either $\op{SL_2}$ or $\op{PGL}_2$, which are excluded.

\begin{lem}\label{L:glob-func}Let $h^\delta:\thig_G^\delta\to \AG$ be the Hitchin morphism. Then, the pull-back of global regular functions 
	$$
	(h^\delta)^*:H^0(\AG,\mathcal O_\AG)\cong H^0(\thig_G^\delta,\oo_{\thig_G^\delta})
	$$
	is an isomorphism of $k$-algebras.
\end{lem}

\begin{proof}We remove the references to $G$ and to $\delta$ from the notation. Consider the following commutative diagram
\begin{equation}\label{E:0}
\xymatrix{
	\widetilde{\thig}\ar[r]^{\mu}\ar@{^{(}->}[d]^\iota&\thig\ar@{^{(}->}[d]\ar[dr]^h\\
\widetilde{\hitgd}\ar[r]^{\nu}&\hitgd\ar[r]^H&\AG
}
\end{equation}	
where the square is cartesian, the arrows $\mu$ and $\nu$ are the normalization morphisms, the vertical arrows are the natural open embeddings and $H$ is the Hitchin morphism from the entire moduli space $\hitgd(=\hitgd_G^\delta)$ of semistable $G$-Higgs bundles.

The Hitchin morphism $H$ is proper and over the open and dense subset $\AG^{\mathrm{ab}}\subset\AG$ is a fibration of abelian varieties. In particular, the generic fiber of the proper morphism $H\circ \nu$ is an abelian variety. By  \cite[\href{https://stacks.math.columbia.edu/tag/0AY8}{Tag 0AY8}]{stacks-project}, the homomorphism of structure sheaves $\mathcal O_{\AG}\to (H\circ\nu)_*\mathcal O_{\widetilde{\hitgd}}$ is an isomorphism. In particular, the pull-back of global regular functions
\begin{equation}\label{E:I}
(H\circ \nu)^*:H^0(\AG,\mathcal O_\AG)\cong H^0(\widetilde{\hitgd},\mathcal O_{\widetilde{\hitgd}})
\end{equation}
is an isomorphism. On the other hand, the complement $\hitgd\setminus\thig$ has codimension at least two in $\hitgd$ (see \cite[Theorem II.6]{Fa93}(ii)+(iii)). Since $\nu$ is finite, the same statement holds for $\widetilde{\hitgd}\setminus\widetilde{\thig}$ in $\widetilde{\hitgd}$. Since $\widetilde{\hitgd}$ is normal, the restriction morphism of global regular functions
\begin{equation}\label{E:II}
	\iota^*:H^0(\widetilde{\hitgd},\mathcal O_{\widetilde{\hitgd}})\cong H^0(\widetilde{\thig},\mathcal O_{\widetilde{\thig}})
\end{equation}
is an isomorphism. Observe that $\iota^*\circ(H\circ \nu)^*=\mu^*\circ h^*$ (see Diagram \eqref{E:0}) is an isomorphism (by \eqref{E:I} and \eqref{E:II}). Since $\mu$ is birational, the pull-back $\mu^*$ is injective. In particular, $h^*$ is an isomorphism and the assertion follows.
\end{proof}

\begin{prop}\label{P:dual-curve}The intersection $C^{*}:=\DG^{\ab}\cap \AG_r$ is irreducible and $\mathbb P(C^*)\subset \mathbb P(\AG_r)$ is the dual variety of the curve $C\subset\mathbb P(\AG_r^*)$ for the embedding given by the line bundle $\omega^{d_r}$.
\end{prop}

\begin{proof} Fix a section $a=(0,0,\ldots,0,a_r)\in \AG_r\subset \AG$. Note that the restriction of the discriminant $\Dg_{\omega}$ to $\Spec(k[p_1,\ldots,p_r]/(p_1,\ldots,p_{r-1}))\subset \cg$ is given by the solution of the polynomial $p_r^r$, see Lemma \ref{R:tang-cone-disc}. Then, the section $a$ is contained in $\DG^{\ab}$ if and only if $a(p)=(0,\ldots,0)$ and $\operatorname{Im}(da_p)=(0,\ldots,0)$ for some $p\in C$. The last statement is equivalent to the condition that the section $a_r\in \AG_r$ vanishes at $p$ with order at least two for some $p\in C$, i.e. $a_r\in H^0(C,\omega^{d_r}(-2p))\subset H^0(C,\omega^{d_r})=\AG_r$. By varying the point $p$, we have the equality
	$$
	C^{*}=\bigcup_{p\in C}H^0(C,\omega^{d_r}(-2p))\subset \AG_r.
	$$
In particular, the projectivization $\mathbb P(C^{*})\subset \mathbb P(\AG_r)$ is exactly the dual variety of the curve $C$ with respect to the linear system $\AG_r$.
\end{proof}

\begin{teo}[Torelli Theorem]\label{T:torelli}Let $C_1$ and $C_1$ be two smooth projective curves. If $\mt M^{\delta_1}_{G}(C_1)\cong \mt M^{\delta_2}_{G}(C_2)$ for some $\delta_1,\delta_2\in\pi_1(G)$, then $C_1\cong C_2$. 
\end{teo}

\begin{proof}By hypothesis, there exists an isomorphism $\varphi:\mt M_{G}^{\delta_1}(C_1)\to \mt M_{G}^{\delta_2}(C_2)$. Then, the codifferential $d\varphi^*:\mt T^*\mt M^{\delta_2}_{G}(C_2)\cong\mt  T^*\mt M^{\delta_1}_{G}(C_1)$ is an isomorphism of cotangent bundles. By Lemma \ref{L:glob-func}, we have a cartesian diagram
	\begin{equation*}\label{E:iso-cart}
		\xymatrix@R=0.5cm{
			\mt T^*\mt M^{\delta_2}_G(C_2)\ar[d]\ar[r]^{d\varphi^*}&	\mt T^*\mt M^{\delta_1}_G(C_1)\ar[d]\\
			\Spec(H^0(\mt T^*\mt M^{\delta_2}_G(C_2),\mathcal O))\ar[r]^\cong\ar[d]^\cong&\Spec(H^0(\mt T^*\mt M^{\delta_2}_G(C_2),\mathcal O))\ar[d]^\cong\\
			\AG\ar[r]^\cong_{\varphi_{\AG}}&\AG
		}
	\end{equation*}
where the compositions of the vertical maps are the Hitchin morphisms. Note that the linearity of $d\varphi^*$ makes the isomorphism $\varphi_{\AG}$ $\mathbb G_m$-equivariant, where the action on $\AG$ is defined as follows 
$$
\lambda.(a_1,\ldots,a_r)=(\lambda^{d_1}a_1,\ldots,\lambda^{d_r}a_r),\text{ for }(a_1,\ldots,a_r)\in\AG\text{ and }\lambda \in \mathbb G_m.
$$
In particular, it preserves the linear subspace in $\AG$ of highest weight, i.e. $\AG_r$. Thus, we get a linear isomorphism
$$
\varphi_{\AG_r}:=(\varphi_{\AG})_{|_\AG}:\AG_r\to\AG_r.
$$
By Corollary \ref{C:geom-fiber}, the divisor $\DG^{\ab}$ is preserved by $\varphi_{\AG_r}$, otherwise you will get an abelian variety birational to a uniruled variety. Hence, the morphism $\varphi_{\AG_r}$ preserves the intersection $\DG^{\ab}\cap\AG_r$. By proposition \ref{P:dual-curve}, we have an isomorphism of the dual varieties of the curves $C_2$ and $C_1$ with respect to the embedding into $\mathbb P(\AG_r)$. Since the double dual of a curve is the curve itself, we get the desired isomorphism $\sigma:C_1\to C_2$ of curves.
\end{proof}

\begin{cor}\label{C:torelli}Let $C_1$ and $C_2$ be two smooth projective curves. If $\overline{\mt M}^{\delta_1}_{G}(C_1)\cong \overline{\mt M}^{\delta_2}_{G}(C_2)$, then $C_1\cong C_2$. 
\end{cor}

\begin{proof}Given an isomorphism $\varphi:\overline{\mt M}^{\delta_1}_{G}(C_1)\cong \overline{\mt M}^{\delta_2}_{G}(C_2)$, we obtain an isomorphism $\mt M^{\delta_1}_G(C_1)\cong \varphi(\mt M^{\delta_1}_G(C_1))$. Both sides are smooth varieties, by Theorem \ref{T:bg-fac}. Furthermore, the tangent space $\mt T^*(\varphi(\mt M^{\delta_1}_G(C_1)))$ satisfies the analogous statement of Lemma \ref{L:glob-func}. Then, by repeating the same argument of Theorem \ref{T:torelli}, we get the corollary. 
\end{proof}
In the proof of Torelli Theorem (= Theorem \ref{T:torelli}), we constructed a morphism of sets
\begin{equation}\label{E:morf}
	\operatorname{Aut}(\M_G^{\delta})\to \operatorname{Aut}(\AG):\varphi\mapsto\varphi_{\AG}.
\end{equation}
A priori, we know that $\varphi_{\AG}$ is $\mathbb{G}_m$-equivariant with respect to the canonical $\mathbb{G}_m$-action on $\AG$, induced by the one of $\cg_{\omega}$. We will show later that this isomorphism comes from the composition of an isomorphism of the vector bundle $\cg_{\omega}$ together with the pull-back along an automorphism of $C$. At the moment, we are only able to show the following fact.

\begin{cor}\label{C:var-l}Let $\varphi$ be an automorphism of $\mt M^{\delta}_G(C)$ and $\varphi_{\AG}$ be the induced automorphism on the Hitchin basis, given by \eqref{E:morf}. Then there exists $\lambda\in k\setminus\{0\}$ and $\sigma\in\operatorname*{Aut}(C)$ such that the restriction $$\varphi_{\AG_r}:=(\varphi_{\AG})_{|_{\AG_r}}:\AG_r\to \AG_r$$ is equal to $(\lambda\cdot Id_{\AG_r})\circ \sigma^*$, where $\sigma^*$ is the isomorphism on $\AG_r$ given by the pull-back of the line bundle $\omega^{d_r}$ along $\sigma$.
\end{cor}

\begin{proof}In the proof of Theorem \ref{T:torelli}, we deduced that the the projective isomorphism
	$
	\overline{\varphi}_{\AG}:\mathbb P(\AG_r)\to\mathbb P(\AG_r)
	$
preserves the dual variety of $C$. Furthermore, the dual map $\overline{\varphi}^*_{\AG}$ preserves the embedding $C\hookrightarrow\mathbb P(\AG^*_r)$. Since $C$ is not contained in any hyperplane of $\mathbb P(\AG^*_r)$ and $\overline{\varphi}^*_{\AG}$ is linear, we get $\overline{\varphi}^*_{\AG}\in\operatorname{Aut}(C)$.
\end{proof}
\section{The tautological automorphisms of $\overline{\mt M}_G$.}\label{Sec:taut-aut}

The aim of this section is to define the homomorphism of Theorem \ref{T:A} and proving its injectivity. Unless otherwise stated, in this section, $G$ is a reductive group and $C$ is a smooth projective curve of genus $g\geq 2$.
\vspace{0.1cm}

Given a $G$-bundle $E$, we may construct some new (possibly isomorphic) $G$-bundles, by using the automorphisms of the curves, the automorphisms of the group and the $\mathscr Z(G)$-bundles on $C$. More precisely, given a $G$-bundle $E$:
\begin{enumerate}[(a)]
	\item if $\sigma:C\cong C$ is an isomorphism, the pull-back $\sigma^*E$ is a $G$-bundle,
	\item if $\rho:G\cong G$ is a group-isomorphism, then we define the $G$-bundle $\rho(E):=E\times^{\rho, G}G$,
	\item if $Z$ is a $\mathscr Z(G)$-bundle, the tensorization by $Z$, denoted by $E\otimes Z$, is a $G$-bundle, where $E\otimes Z$ is the fibered product $E\times_C Z$ modulo the equivalence relation $(et,z)\sim (e,zt)$, for any $e\in E$, $z\in Z$ and $t\in\scr Z(G)$.
\end{enumerate}
In particular, they define automorphisms for the moduli stack $\BG_G$ of $G$-bundles. We remark that some of these automorphisms may be trivial, as the next lemma shows.
\begin{lem}If $\rho\in \operatorname{Aut}(G)$ is an inner automorphism (i.e. $\rho\in\operatorname{Inn}(G)$), then $E\cong\rho(E)$ for any $G$-bundle.
\end{lem}

\begin{proof}By hypothesis, there exists $h\in G$ such that $hgh^{-1}=\rho(g)$ for any $g\in G$. Then, the map
	\begin{align*}
		E&\xrightarrow{f}\rho(E):=E\times^{\rho,G}G\\
		e&\mapsto [eh,1]
	\end{align*}
is a well-defined morphism of $G$-bundles and it gives the desidered isomorphism. Indeed, for any $e\in E$ and $g\in G$, we have
$
f(eg)=[egh,1]=[ehh^{-1}gh,1]=[eh\rho^{-1}(g),1]=[eh,g]\stackrel{\operatorname{def}}{=}f(e)g.
$
\end{proof}
In particular, the isomorphism $\BG_G\to \BG_G:E\mapsto\rho(E)$, induced by an automorphism $\rho\in \Aut(G)$, depends only on the class $[\rho]$ in the quotient $\Out(G):=\Aut(G)/\operatorname{Inn}(G)$. With abuse of notation, we denote by the same symbol $\rho$ an automorphism $\rho$ of $G$ and its equivalence class $[\rho]$ in $\Out(G)$. Next, we want to show that all these automorphisms of $\BG_G$ actually descend to automorphisms of the moduli space $\overline{\mt M}_G$ of semistable bundles.
\begin{lem}Let\label{L:sss} $E$ be a simple, resp. stable, resp. semistable, $G$-bundle. Then the following $G$-bundles are simple, resp. stable, resp. semistable.
	\begin{enumerate}[(i)]
		\item\label{L:sss1} $\sigma^*(E)$, where $\sigma\in \Aut(C)$,
		\item\label{L:sss2} $\rho(E)$, where $\rho \in \Out(G)$,
		\item\label{L:sss3} $E\mapsto E\otimes Z$, where $Z$ is a $\mathscr Z(G)$-bundle over $C$. 
	\end{enumerate}
\end{lem}

\begin{proof}We recall that the operations above define automorphisms of the moduli stack $\BG_G$. In particular, all the bundles in the list have automorphism groups equal to $\Aut(E)$. Hence, if $E$ simple, the same holds for the other bundles in the list. We now focus on the (semi)stable case. Fix a parabolic subgroup $P\subset G$.\\	
\fbox{$(\ref{L:sss1})$} The $P$-reductions of $E$ are in bijection with the $P$-reductions of $\sigma^*E$. Indeed, $F$ is a $P$-reduction of $E$ if and only if $\sigma^*F$ is a $P$-reduction of $\sigma^*E$. Furthermore, the pull-back along $\sigma$ commutes with the formation of the adjoint bundles and preserves the degree, i.e.
$$
\deg\left( F\times^P\mathfrak p\right)=\deg \left( \sigma^*F\times^P\mathfrak p\right).
$$
Putting all together, we get that $E$ is (semi)stable if and only if $\sigma^*E$ is (semi)stable.\\
\fbox{$(\ref{L:sss2})$} Since $\rho$ is an isomorphism, it induces a bijection between the parabolic subgroups of $G$, i.e. $P$ is parabolic if and only if $\rho(P)$ is parabolic. Moreover, similarly to the previous point, the $P$-reductions of $E$ are in bijection with the $\rho(P)$-reductions of $\rho(E)$. Indeed, $F$ is a $P$-reduction of $E$ if and only if $\rho(F):=F\times^{\rho, P}\rho(P)$ is a $\rho(P)$-reduction of $\rho(E)$. Let $d\rho:\mathfrak p\to d\rho(\mathfrak p)$ be the isomorphism of Lie algebras given by $\rho:P\to\rho(P)$. Then, the homomorphism of vector bundles 
$$F\times^P\mathfrak p\to \rho(F)\times^{\rho(P)}d\rho(\mathfrak p)=F\times^{\rho,P}d\rho(\mathfrak p):[(f,x)]\mapsto [(f,d\rho(x))]$$ is an isomorphism and, so, they have the same degree. Hence $E$ is (semi)stable if and only if $\rho(E)$ is (semi)stable.\\
\fbox{$(\ref{L:sss3})$} As in the point $(i)$, the $P$-reductions of $E$ are in bijection with the $P$-reductions of $E\otimes Z$. Indeed, $F$ is a $P$-reduction of $E$ if and only if $F\otimes Z$ is a $P$-reduction of $E\otimes Z$. Furthermore, the adjoint bundles 
$$
F\times^P\mathfrak p \cong (F/\mathscr Z(G))\times^{P/\mathscr Z(G)}\mathfrak p\cong (F\otimes Z)\times ^P\mathfrak p
$$ are isomorphic and so they have the same degree. Hence, $E$ is (semi)stable if and only if $E\otimes Z$ is (semi)stable.
\end{proof}

By Lemma \ref{L:sss}, all these automorphisms preserve the locus $\BG_G^{(s)s}$ of (semi)stable bundles. By the universal property of the good moduli spaces, the automorphisms descend to the moduli space $\overline{\mt M}_G$ of semistable $G$-bundles. Putting together all the results, we have the following proposition.

\begin{prop}\label{P:aut-m} We have a homomorphism of groups
\begin{equation}\label{E:aut-m}
	\begin{array}{cll}
			H^1(C,\mathscr Z(G))\rtimes\left(\Aut(C)\times \Out(G)\right)&\longrightarrow \Aut\left(\overline{\mt M}_G\right)\\
			\nonumber(Z,\sigma,\rho)&\longmapsto \left(E\mapsto \rho((\sigma^{-1})^*E)\otimes Z\right)
		\end{array}
\end{equation}
where:
\begin{enumerate}[(i)]
	\item $\Aut(C)$ acts on $H^1(C,\mathscr Z(G))$ by pull-back, i.e. $Z\mapsto (\sigma^{-1})^*Z$,
	\item $\Out(G)$ acts on $H^1(C,\mathscr Z(G))$ as follows $\rho.Z=Z\times^{\rho,\mathscr Z(G)}\mathscr Z(G)$.
\end{enumerate} 
The same holds if we replace $\overline{\mt M}_G$ with the locus $\mt M_G$ of simple and stable $G$-bundles.
\end{prop}

We remark that these moduli spaces are disconnected, if $G$ is not simply-connected. So, a priori, these automorphisms may move the connected components of the moduli space $\overline{\mt M}_G$. We are interested in those automorphisms, which fix these components. Before studying them, we need to fix some notations. Any element in $\rho\in \Out(G)$ gives a (possibly trivial) automorphism of the fundamental group of $G$
$$\begin{array}{ccc}
\pi_1(G)&\xrightarrow{\pi_1(\rho)}&\pi_1(G)\\
	\delta&\longmapsto&	\pi_1(\rho)(\delta).
\end{array}$$
On the other hand, the group $\mathscr Z(G)$ is a (possibly disconnected) multiplicative group. It is isomorphic to $F\times \mathbb G_m^{s}$, where $F$ is a finite group over $k$ and $s\geq 0$. In particular, the connected components of the moduli stacks of $\scr Z(G)$-bundles are in bijection with $\widehat{F}^{2g}\times\pi_1(\mathbb G_m^s)=\operatorname{Hom}(F,\mathbb G_m)^{2g}\times \operatorname{Hom}(\mathbb G_m,\scr Z(G))$. Note that the factor $\operatorname{Hom}(\mathbb G_m,\scr Z(G))$ is independent of the choice of the splitting. 

\begin{defin}\label{D:degree}Fix a maximal torus $T_G\subset G$. We say that a $\mathscr Z(G)$-bundle $Z$ is \emph{of degree $z\in \operatorname{Hom}(\mathbb G_m,\scr Z(G))$}, if its image $Z\times^{\mathscr Z(G)}T_G$ along the morphism $\BG_{\scr Z(G)}\to \BG_{T_G}$ is contained in the connected component $\BG_{T_G}^{\pi_1(\iota)(z)}$, where $\pi_1(\iota):\operatorname{Hom}(\mathbb G_m,\scr Z(G))\hookrightarrow \operatorname{Hom}(\mathbb G_m,T_G)=\pi_1(T_G)$ is the obvious inclusion.
\end{defin}

Since $\scr Z(G)$ is contained in any maximal torus of $G$, the definition of degree above is independent from the choice of the maximal torus. The next lemma explains how the automorphisms in the image of \eqref{E:aut-m} move the connected components of $\overline{\mt M}_G$.
\begin{lem}\label{L:mov-con}The following fact holds.
	\begin{enumerate}[(i)]
		\item Let $\sigma\in \Aut(C)$, then the map $E\mapsto \sigma^*E$ gives an automorphism of varieties $\overline{\mt M}^{\delta}_G\cong \overline{\mt M}^{\delta}_G$.
		\item Let $\rho\in\Out(G)$, then the map $E\mapsto \rho(E)$ gives an automorphism of varieties $\overline{\mt M}^{\delta}_G\cong \overline{\mt M}^{\pi_1(\rho)(\delta)}_G$.
		\item Let $Z$ be a $\mathscr Z(G)$-bundle of degree $z\in\operatorname{Hom}(\mathbb G_m,\scr Z(G))$, then the map $E\mapsto E\otimes Z$ gives an automorphism of varieties $\overline{\mt M}^{\delta}_G\cong \overline{\mt M}^{\delta+[\pi_1(\iota)(z)]}_G$, where $\pi_1(\iota):\operatorname{Hom}(\mathbb G_m,\scr Z(G))\hookrightarrow\pi_1(T_G)$ is the obvious inclusion.
	\end{enumerate}
\end{lem}

\begin{proof}It is enough to show the analogous statements for the moduli stack $\BG_G$ of $G$-bundles. Fix a maximal torus $T$. Point $(i)$: the natural map $
	\BG_{T}\to\BG_G: E\mapsto E\times ^{T}G
	$
induces a surjection from the point of view of connected components $\pi_1(T)\twoheadrightarrow\pi_1(G)$. Then, we may assume $G=T$ torus. After choosing a splitting of the torus, we may assume $T=\mathbb G_m$. So, we have reduced the problem to check if the pull-back along $\sigma$ preserves the degree of a line bundle. The latter statement is well-known. Point $(ii)$: it is a consequence of \cite[Theorem 5.8]{Ho10}. Point $(iii)$: arguing as in point $(i)$, it is enough to show that the isomorphism
	\begin{align}\label{E:torus-tensor}
		\BG_{T}&\xrightarrow{\sim}\BG_{T}\\
		\nonumber E&\mapsto E\otimes Z
	\end{align}
sends the connected component corresponding to $d\in\pi_1(T)$ to the one corresponding to $d+\pi_1(\iota)(z)\in \pi_1(T)$. Let $m:T\times T\to T:(t,s)\mapsto ts$ be the product homomorphism. We then have the following isomorphisms of $T$-bundles
$$
\def\arraystretch{1.5}\begin{array}{ccccc}
 E\otimes Z&\xrightarrow{\sim}& (E\times_C Z)\times^{m,T\times\mathscr Z(G)}T&\xrightarrow{\sim}&\Big(E\times_C(Z\times^{\mathscr Z(G)}T)\Big)\times^{m,T\times T}T\\
\, [e,z]&\mapsto & [(e,z),1] &&\\
 & &[(e,z),t] &\mapsto &[(e,[z,1]),t]
\end{array}
$$
Hence, the isomorphism \eqref{E:torus-tensor} is the restriction to the substack $\BG_{T}\times \{Z\times^{\mathscr Z(G)}T\}$ of the morphism of moduli stacks
$$
\BG_{T}\times\BG_{T}\cong\BG_{T\times T}\to \BG_T:(E,F)=(E\times_CF)\mapsto (E\times_C F)\times^{m,T\times T}T.
$$
associated to the product homomorphism $m:T\times T\to T$. From this, the assertion follows by \cite[Theorem 5.8]{Ho10}.
\end{proof}

\begin{rmk}The notation $E\otimes Z$ is not accidental. Indeed, if $G=\op{GL}_n$ (and so $\scr Z(G)=\mathbb G_m$), the vector bundle attached to the tensor product $E\otimes Z$ is nothing but the usual tensor product of $\oo_C$-modules between the vector bundle attached to $E$ and the line bundle attached to $Z$.
\end{rmk}

By Lemma \ref{L:mov-con}, we have a complete description of the automorphisms in the image of the homomorphism described in Proposition \ref{P:aut-m}, which fix a connected component. 
\begin{prop}\label{P:aut-m}Let $G$ be a semi-simple group and fix $\delta\in \pi_1(G)$. Then, we have an injective homomorphism of groups
	\begin{equation}\label{E:aut-delta}
		\def\arraystretch{1.5}\begin{array}{cll}
			H^1(C,\mathscr Z(G))\rtimes\left(\Aut(C)\times \Out(G,\delta)\right)&\lhook\joinrel\longrightarrow & \Aut(\overline{\mt M}_G^{\delta})\\
			\nonumber(Z,\sigma,\rho)&\longmapsto &\left(E\mapsto \rho((\sigma^{-1})^*E)\otimes Z\right)
		\end{array}
	\end{equation}
where:
\begin{enumerate}[(i)]
	\item $\Aut(C)$ acts on $H^1(C,\mathscr Z(G))$ by pull-back, i.e. $Z\mapsto(\sigma^{-1})^*Z$,
	\item $\Out(G,\delta)=\{\rho\in\Out(G)\vert \pi_1(\rho)(\delta)=\delta\}$ acts on $H^1(C,\mathscr Z(G))$ as follows: $\rho.Z=Z\times^{\rho,\mathscr Z(G)}\mathscr Z(G)$.
\end{enumerate}
The same holds if we replace $\overline{\mt M}_G^{\delta}$ with the locus $\mt M_G^{\delta}$ of simple and stable $G$-bundles.
\end{prop}

\begin{proof}The only fact, which has not been proved yet, is the injectivity. We need to show that for any stable and simple $G$-bundle $E$ there exists an isomorphism of $G$-bundles
\begin{equation}\label{E:assump}
	\psi: E\cong \rho(\sigma^*E)\otimes Z,
\end{equation}
then $(Z,\rho,\sigma)=(\oo_C,\Id_G,\Id_C)$. We divide the proof in three steps.\\
\underline{Reduction from $(Z,\rho,\sigma)$ to $(Z,\rho,\Id_C)$.} From the results of Section \ref{Sec:torelli} (see, for instance, Corollary \ref{C:var-l}), we get that if $\sigma\neq \Id_C$, then the automorphism of $\mt M_G^{\delta}$ is always non-trivial.\\
\underline{Reduction from $(Z,\rho,\Id_C)$ to $(\oo_C,\rho,\Id_C)$.} It is enough to show that the action of $H^1(C,\mathscr Z(G))$ is effective on $\mt M_G^{\delta}$. For the statement with respect to the moduli stack $\BG_G$, see \cite[Example 5.1.4]{BHLB}. As a corollary, we get the analogous statement for the (coarse) moduli space $\mt M_G$ of simple and stable $G$-bundles.\\
\underline{Case $(\oo_C,\rho,\Id_C)$}. It is enough to show that the fixed-point loci
$$
\mt M_G(\rho):=\{E\in\mt M_G\vert\, E\cong\rho(E)\}
$$
is a proper subvariety of $\mt M_G$. This statement follows by the results of \cite{OrRa}. Here, we briefly recall the necessary steps. First, we need to fix some notations. It is well-known that $\Out(G)$ is a finite group and the quotient map $\Aut(G)\to \Out(G)$ admits a (non-canonical) splitting. Then, if $\rho\in\Out(G)$ is an element of order $n$, we fix a lift $\theta\in\Aut(G)$ such that $\theta^n=\Id_G$. Let $E$ be a stable and simple $G$-bundle such that there exists an isomorphism $\psi:E\cong\rho(E)$ of $G$-bundles. Since $\rho$ has order $n$, we have an equality of $G$-bundles $\rho^n(E)=E$. Furthermore, the $n$-composition $\psi^n$ is an automorphism of $E$. Since we assumed $E$ simple, we must have $\psi^n=z\in \mathscr Z(G)$. We then define the following subset
$$
S(z):=\{s\in G\vert \,s\theta(s)\cdots\theta^{n-1}(s)=z\in \mathscr Z(G)\}.
$$
The group $G$ acts on $S(z)$ by twisted conjugation, i.e. $s.g=g^{-1}s\theta(g)$. By \cite[Proposition 3.9]{OrRa}, the bundle $E$ admits a $G^{\theta,s}$-reduction, where $G^{\theta,s}$ is the proper subgroup of $G$ of fixed points under the automorphism 
$$
G\xrightarrow{\theta}G\xrightarrow{s(-)s^{-1}}G,
$$
where $s$ is any element in a certain unique $G$-orbit of $S(z)$. We remark that if $E$ is stable, then its $G^{\theta,s}$-reduction remains stable. In particular, the fixed-point locus $\mt M_G(\rho)$ is contained in the (closure of the) image of the natural morphisms of moduli spaces $$f_{s,z}:\mt M_{G^{\theta,s}}\to\mt M_G\text{ for any }s\in S(z) \text{ and for any }z\in\mathscr Z(G).$$
By the above discussion, the intersections $\mt M_G(\rho)\cap \op{Im}(f_{s,z})$ and $\mt M_G(\rho)\cap \op{Im}(f_{s',z'})$ coincide if $z=z'$ and $s$ and $ s'$ are in the same $G$-orbit of $S(z)$. Since $\mathscr Z(G)$ is a finite group and the set of $G$-orbits of $S(z)$ is finite, the fixed-point locus $\mt M_G(\rho)$ is contained in a finite union of proper subvarieties of $\mt M_G$.
\end{proof}

\begin{rmk}Using Lemma \ref{L:mov-con}, one could generalize the homomorphism of Proposition \ref{P:aut-m} to any reductive group $G$. However, in this generality, the description of the subgroup of the left-hand side of \eqref{E:aut-m} is more involved. Namely, it consists of triples $(Z,\sigma,\rho)$ such that $\pi_1(\rho)(\delta)+[\pi_1(\iota)(\deg Z)]=\delta\in\pi_1(G)$. 
	
We remark that, in this generality, the injectivity may fail. For example when $G=\mathbb G_m$ and $C$ hyperelliptic, the outer isomorphism $\mathbb G_m\to\mathbb G_m:t\mapsto t^{-1}$ and the hyperelliptic involution $\sigma:C\cong C$ define the same automorphism $\mt J\to \mt J:L\mapsto L^{-1}$ on the jacobian $\mt J=\overline{\mt M}_{\mathbb G_m}$ of the curve.
\end{rmk}

\section{Nilpotent Cone and Flag Varieties.}\label{Sec:nilp}
Before completing the proof of Theorem \ref{T:A}, we recall some facts about the nilpotent cone of a Lie algebra and the cohomology of the tangent bundle on the flag variety. They will be used in the proof of Theorem \ref{T:T}, where we will prove the surjectivity of the homomorphism in Theorem \ref{T:A}.
\vspace{0.1cm}

The \emph{nilpotent cone $\mt N$} is the subset of nilpotent elements in $\g$, i.e.
$$
\mt N:=\{x\in\g\vert\, \text{the adjoint endomorphism }[x,-]:\g\to\g \text{ is nilpotent}\}.
$$
Note that any nilpotent element in $\g$ is contained in at least one Borel subalgebra. If $G$ is a reductive group with Lie algebra $\g$, the Borel subalgebras are parametrized by the flag variety $G/B$. The Lie algebra $\mathfrak b=\tg\oplus\mathfrak n$, of the Borel group $B$, is a Borel subalgebra and it splits as direct sum of the Cartan algebra $\tg$ and a nilpotent Lie algebra $\mathfrak n$.

\begin{lem}Let\label{L:nc} $G$ be a reductive group and $\g$ be its reductive Lie algebra. The following facts hold.
	\begin{enumerate}[(i)]
		\item\label{L:nc1} $\mt N$ is the fiber over the zero-element of the morphism $\g\xrightarrow{\chi}\cg$ introduced in Theorem \ref{T:konstant}$(\ref{T:konstant3})$.
		\item\label{L:nc2} $\nc$ is an irreducible variety of dimension $|\Phi|=\#\{\operatorname{roots}\}=\dim G-\dim T_G$.
		\item\label{L:nc3} the morphism $\widetilde\nc=\{(x, \mathfrak l)\in \nc\times G/B\vert\, x\in \mathfrak l\}\xrightarrow{\operatorname{pr}_1}\nc$ is a resolution of singularities.
		\item\label{L:nc4} there exists a commutative diagram of schemes
		\begin{equation}
			\xymatrix{
				\widetilde\nc\ar[r]^{\sim}\ar[dr]^{\operatorname{pr_2}}&G\times^B\mathfrak n\ar[d]^q&\mt T^*(G/B)\ar[l]_{\sim}\ar[ld]^{p}\\
				&G/B&
			}
		\end{equation}
		where the horizontal rows are $G$-equivariant isomorphisms and $p$ is the obvious morphism from the co-tangent bundle of $G/B$ to $G/B$.
		\item\label{L:nc5} A point $x\in \nc$ is smooth if and only if it is regular (i.e. the adjoint homomorphism $[x,-]:\g\to\g$ has kernel of dimension $r$).
		\item\label{L:nc6} There exists a $G$-equivariant fibration $\nc^{\sm}\to G/B$ with fibers isomorphic to 
		$
		\mathbb G_m^{l}\times \mathbb A^{\frac{|\Phi|}{2}-l}
		$, where $l$ is the number of simple roots (i.e. the rank of the semi-simple Lie algebra $[\g,\g]$).
	\end{enumerate}
\end{lem}

\begin{proof}Point $(\ref{L:nc1})$: see \cite[Proposition 3.2.5]{CG}. Point $(\ref{L:nc2})$: see \cite[Corollary 3.2.8]{CG}. Point $(\ref{L:nc3})$: it is a consequence of \cite[Proposition 3.2.14]{CG} and the discussion after its proof. Point $(\ref{L:nc4})$: see \cite[Lemma 3.2.2]{CG}. Point $(\ref{L:nc5})$: observe that $\nc=\chi^{-1}(0)=\Spec k[\g]/(p_1,\ldots,p_r)$ (see point $(\ref{L:nc1})$) and $\dim \nc=\dim \g-r$ (see \eqref{L:nc2}). Hence, $\nc$ is smooth at $x$ if and only the differentials $dp_1,\ldots,dp_r$ are linearly independent at $x$. The latter condition is equivalent to $x$ being regular (see \cite[Theorem 9]{Ko63}). Point $(\ref{L:nc6})$: our candidate is the composition of $G$-equivariant morphisms
	\begin{equation}\label{E:fibration}
		\nc^{\sm}\cong pr_1^{-1}(\nc^{\sm})\cong G\times^B\mathfrak n^{\reg}\subset G\times^B\mathfrak n\xrightarrow{q}G/B.
	\end{equation}
	It remains to give the description of the fibers. By $G$-equivariance, it is enough to study the fiber over $[1]=[\mathfrak b]\in G/B$, which is isomorphic to $\mathfrak n^{\reg}$. For any root $\alpha$, fix a non-zero vector $v_{\alpha}$ in the eigenspace $\g_{\alpha}$. By \cite[Lemma 3.2.12]{CG}, an element $x\in \g$ is contained in $\mathfrak n^{\reg}$ if and only if
	$
	x=\sum_{\alpha>0}c_{\alpha} v_{\alpha},
	$
	where $c_{\alpha}\in k$ and $c_{\alpha}\neq 0$ if $\alpha$ simple root. In other words, $\mathfrak n^{\reg}\cong \mathbb G_m^{l}\times \mathbb A^{\frac{|\Phi|}{2}-l}$.
\end{proof}
\begin{rmk}A more direct proof of the above lemma for the case $G=\op{Sp}_{2n}$ can be found in \cite[Lemma 5.1]{BGMsym}.
\end{rmk}
We also will need the description of the cohomology of the tangent sheaf of the flag varieties.
\begin{lem}\label{L:aut-flag}Let $G$ be an almost-simple group and $\mt T(G/B)$ be the tangent bundle of the flag variety $G/B$. Then, we have
	\begin{enumerate}[(i)]
		\item an isomorphism $\g\cong H^0(G/B,\mathcal T(G/B))$ of Lie algebras, where the Lie bracket on the right hand side is the bracket of vector fields;
		\item $H^i(G/B,\mathcal T(G/B))=0$, for $i>0$.
	\end{enumerate}
\end{lem}

\begin{proof}Without loss of generality, we may assume $G$ adjoint. The statement, with the exception of the injectivity of the homomorphism at the first point, is the content of \cite[Proposition 2, p. 182]{D77}. Note that the statement in \emph{loc.cit.} holds for non exceptional (in the sense of \cite[\S 2]{D77}) parabolic subgroups of $G$. However, as explained at the end of \cite[\S 2]{D77}, a Borel subgroup is always non exceptional. The injectivity follows by observing that the surjective homomorphism of Lie algebras $L(\varphi):\g\to H^0(G/B,\mathcal T(G/B))$ in \cite[Proposition 2, p. 182]{D77} is the derivation of the isomorphism $\varphi^o:G=Aut(G)^o\cong Aut(G/B)^o$, where $\varphi^o$ is the isomorphism $\varphi$ in \cite[Theorem 1, p. 182]{D77} restricted to the connected components of the identity element. 
\end{proof}

\section{The automorphism group of $\overline{\mt M}_G$.}\label{Sec:last}
The aim of this section is to conclude the proof of Theorem \ref{T:A}, by showing the surjectivity of the homomorphism in \emph{loc.cit.}. As in Section \ref{Sec:torelli}, our strategy is similar to the one given in \cite{BGMsl}, for $G=\operatorname{SL}_n$, in \cite{BGMsym}, for $G=\operatorname{Sp}_{2n}$, and in \cite{Sanc} for the cases $G=\mathbb E_6^{\operatorname{sc}}, \mathbb F_4$.  Unless otherwise stated, we will always assume $G$ almost-simple and $C$ smooth projective curve of genus $g\geq 4$.\\

Any automorphism of $\overline{\mt M}^\delta_G$ preserves its smooth locus. Under the assumption $g\geq 3$, the smooth locus coincides with the locus of simple and stable bundles (see \cite[Corollary 3.4]{BH12}). Hence, we have a well-defined restriction group homomorphism $\Aut(\overline{\mt M}_G^\delta)\to \Aut(\mt M_G^\delta)$. We have the following

\begin{prop}\label{P:aut-finite}Let $\delta\in\pi_1(G)$. The following facts hold true.
	\begin{enumerate}[(i)]
		\item\label{P:aut-finite1} The restriction homomorphism $\Aut(\overline{\mt M}_G^\delta)\to \Aut(\mt M_G^\delta)$ is an isomorphism.
		\item\label{P:aut-finite2} The moduli space $\overline{\mt M}_G^\delta$ does not have non-trivial global vector fields, i.e. $H^0(\overline{\mt M}_G^\delta,\mt T\overline{\mt M}_G^\delta)=0$.
		\item\label{P:aut-finite3} The automorphism group $\Aut(\overline{\mt M}_G^\delta)$ is a finite group. The same holds for the open locus $\mt M_G^\delta$ of stable and simple $G$-bundles. 
	\end{enumerate}
\end{prop}

\begin{proof} \fbox{\eqref{P:aut-finite1}} The injectivity follows because $\mt M^\delta$ is a dense open subset of $\overline{\mt M}^\delta$. We now focus on the surjectivity. We recall that $\overline{\mt M}^\delta$ is a Gorenstein projective variety (see \cite[Proposition 13.6]{BLS}). Moreover, its Picard group is equal to $\mathbb Z$ (see \cite[Theorem (b)]{BLS}, \cite{Sor}) and its dualising line bundle $\omega_{\overline{\mt M}^\delta}$ is anti-ample (see \cite[Proposition 4.3]{BH12}). In particular, there exists a negative integer $n$ such that $\omega^n_{\overline{\mt M}^\delta}$ is a very ample line bundle. Since complement of $\mt M^\delta$ in the normal variety $\overline{\mt M}^\delta$ is precisely the singular locus, it has codimension at least two. Hence, the restriction homomorphism
	\begin{equation}\label{E:l-restr}
	H^0(\overline{\mt M}^\delta, \omega^n_{\overline{\mt M}^\delta})\to H^0(\mt M^\delta,\omega^n_{\mt M^\delta})
	\end{equation}
is an isomorphism. Any automorphism $\varphi:\mt M^\delta\to\mt M^\delta$ fixes the dualising line bundle $\omega_{\mt M^\delta}$. By the isomorphism \eqref{E:l-restr}, $\varphi$ acts on $H^0(\overline{\mt M}^\delta, \omega^n_{\overline{\mt M}^\delta})$. Hence, it is the restriction of an automorphism of $\mathbb P(H^0(\overline{\mt M}^\delta, \omega^n_{\overline{\mt M}^\delta}))$, which must fix the image of $\overline{\mt M}^\delta$. In other words, the automorphism $\varphi$ extends to an automorphism on the whole variety $\overline{\mt M}^\delta$.

\noindent\fbox{\eqref{P:aut-finite2}} By definition $\mt T\overline{\mt M}^\delta=Hom(\mt T^*\overline{\mt M}^\delta,\mt O_{\overline{\mt M}^\delta})$, where here $\mt T^*\overline{\mt M}_G^\delta$ denotes the cotangent sheaf. In particular, a global vector field is equivalent to a morphism $s:\mt T^*\overline{\mt M}^\delta\to \mt O_{\overline{\mt M}^\delta}$ of  $\mt O_{\overline{\mt M}^\delta}$-modules. Since the moduli space $\overline{\mt M}_G^\delta$ is integral, we could check the triviality of $s$ on a (dense) open subset. By \cite[Corollary III.3]{Fa93}, the open locus of stable bundles does not have global vector fields.

\noindent\fbox{\eqref{P:aut-finite3}} Since $\overline{\mt M}^\delta$ is projective, the automorphism group $\Aut(\overline{\mt M}^\delta)$ is the group $G(k)$ of $k$-valued points of some group scheme $G$ over $k$ of dimension $\dim\mathrm{Lie}(G)=\dim H^0(\overline{\mt M}_G^\delta,\mt T\overline{\mt M}_G^\delta)$. By point \eqref{P:aut-finite2}, the group $G$ is zero-dimensional, i.e. $G=\pi_0(G)$, where $\pi_0(G)$ denotes the group scheme of the connected components of $G$. On the other hand, from the proof of point \eqref{P:aut-finite1}, we also have that $G$ fixes the class $[\omega^n_{\overline{\mt M}^\delta}]$ in the (torsion-free) N\'eron-Severi group $\mathrm{NS}(\overline{\mt M}^\delta)=\Pic(\overline{\mt M}^\delta)=\bbZ$. By \cite[Theorem 2.10]{BrionNotes}, the group scheme $G=\pi_0(G)$ (and so the group $G(k)$) is finite. The statement about the locus of stable and simple bundles follows from point \eqref{P:aut-finite1}.
\end{proof}

The fact that the automorphism group of our moduli spaces is finite implies that the $\mathbb G_m$-equivariant automorphism of the Hitchin basis induced by an automorphism of $\mt M_G^\delta$ is linear. More precisely,

\begin{lem}\label{L:vp=linear}Let $\varphi$ be an automorphism of $\Md_G$ and  $\varphi_{\AG}$ be the induced automorphism on the Hitchin basis, given by \eqref{E:morf}. Then there exists a $\mathbb G_m$-equivariant linear isomorphism $L_\varphi\in \op{GL}(\AG)$ and $\sigma \in\operatorname*{Aut}(C)$ such that 
	$
	\varphi_{\AG}=L_\varphi\circ\sigma^*
	$
and $L_\varphi$ restricted to $\AG_r$ is equal to $\lambda\cdot \mathrm{Id}_{\AG_r}$ for some $\lambda\in k\setminus\{0\}$.
\end{lem}
\begin{proof}
As in in the proof of Torelli Theorem \ref{T:torelli}, we have a cartesian diagram
\begin{equation}\label{E:iso-cartbis}
	\xymatrix@R=0.6cm{
		\thigd\ar[d]^-{h^{\delta}}\ar[r]^{d\varphi^*}&\thigd\ar[d]^-{h^{\delta}}\\
		\AG\ar[r]^{\varphi_{\AG}}&\AG
	}
\end{equation}
By Corollary \ref{C:var-l}, the restriction of $\varphi_{\AG}$ to $\AG_r$ is equal to $(\lambda\cdot \Id_{\AG_r})\circ \sigma^*$ for some $\lambda\in k\setminus\{0\}$ and $\sigma\in\operatorname{Aut}(C)$. After composing $\varphi$ with $(\sigma^{-1})^*$, we may assume $(\varphi_{\AG})_{|\AG_r}=\lambda\cdot\Id_{\AG_r}$. It remains to show that $\varphi_\AG$ is linear. By Proposition \ref{P:aut-finite}\eqref{P:aut-finite3}, we know that $\varphi^n=\mathrm{Id}_{\mt M^\delta}$ for some integer $n$. In particular, the same holds for the co-differential $d\varphi^*$ and, so, $\varphi_\AG^n=\mathrm{Id}_{\AG}$. Hence, the morphism $\varphi_\AG$ is an $n$-torsion element of the group $\Aut_{\mathbb{G}_m}(\AG)$ of $\mathbb{G}_m$-equivariant automorphisms of the affine space $\AG$. This group coincides with the automorphism group of the weighted projective stack $[(\AG\setminus\{0\})/\mathbb G_m]$ (see \cite[Proposition 7.2]{Noohi}). This group may be written as a semidirect product
$$
\Aut_{\mathbb{G}_m}(\AG)=V\rtimes G,
$$
where $G$ is the subgroup of $\mathbb G_m$-equivariant linear isomorphisms of $\AG$ and $V$ is a normal subgroup admitting a subnormal series ${1}\subset V_1\subset \cdots V_t=V$ such that $V_{i+1}/V_i\cong \mathbb G_a^{n_i}$ for some integer $n_i$ (see \cite[Theorem 7.7]{Noohi}). In particular, $V$ does not have torsion elements and, so, $\varphi_\AG\in G$, i.e. $\varphi_\AG$ is linear.
\end{proof}

We want to show that the linear isomorphism $L_\varphi$ defined in Lemma \ref{L:vp=linear} is induced by an automorphism of the vector bundle $\mathfrak c_{\omega}$. Before proving this, we need some preparation.
\begin{lem}\label{L:gen-st}For a generic $G$-bundle $E$, we have  $H^0(C,\ads(E)(p))=0$, for any point $p\in C$.
\end{lem}

\begin{proof}We will prove the lemma with the help of the results contained in \cite[Sections $3$ and $4$]{BH10}. Without loss of generality, we may assume $G$ of adjoint type, i.e. $\scr Z(G)=\{1\}$. Consider the entire moduli stack $\mathbf{Bun}_G$ of $G$-bundles. Since the vanishing in the assertion is an open condition on $\mathbf{Bun}_G$, it is enough to find one $G$-bundle with that vanishing property, for any connected component $\mathbf{Bun}^{\delta}_G$, with $\delta\in\pi_1(G)$. Fix a $\delta\in\pi_1(G)$ and let $B$ be a Borel subgroup and $U$ its unipotent radical. The inclusion $B\subset G$ gives a morphism of moduli stacks $\mathbf{Bun}^d_B\to\mathbf{Bun}^{\delta}_G:E\mapsto E\times^BG$, where $d\in\pi_1(B)=\pi_1(T)$ is a lift of $\delta$. Then, it is enough to show that there exists a lift $d\in\pi_1(B)$ of $\delta$ and a $B$-bundle $E$ in $\mathbf{Bun}^d_B$ such that $H^0(C,\ads(E)(p))=0$ for any $p\in C$, where $\ads(E):=E\times^B\g$. The adjoint action gives an inclusion of vector bundles
	$$\begin{array}{rcl}
	\ads(E)=E\times^B\g&\hookrightarrow&\displaystyle \frac{End(\ads(E))}{k.\Id}=E\times^{B}\left(\frac{\operatorname{End}(\g)}{k.\Id}\right)\\
	x&\mapsto& \ads(x)=[x,-].
\end{array}$$
Then, it is enough to show that the vanishing of the global sections of the sheaf $$(End(\ads(E))/k.\Id)(p),$$ for any point $p\in C$. By \cite[Chapitre
	VIII, \S7, Proposition 8]{Bu}, we know that if $\delta\neq 0$, it admits a unique lift to a minuscule element $d\in\pi_1(B)$, i.e. a co-character $d$ such that $\alpha(d)\in \{0,1\}$ for any positive root $\alpha$. Let $d\in\pi_1(B)$ such that either $-d$ is minuscule or $d=0$.  We will show that there exists a $B$-bundle $E$ in $\mathbf{Bun}_{B}^{d}$ such that $H^0(C,(End(\ads(E))/k.\Id)(p))=0$ for any $p\in C$. By the above discussion, this fact will give the lemma.
	
Let $E$ be an arbitrary $B$-bundle. By \cite[Corollary 4.5]{BH10} the vector bundle $End(\ads(E))/k.\Id$ is a successive extension of vector bundles of the following type:
\begin{itemize}
	\item line bundles $E(\alpha,1):=E\times^{B} \mathbb A^1$, where the action on $\mathbb A^1$ is given by a root $B\xrightarrow{\alpha}\mathbb G_m$,
	\item line bundles $E(\alpha-\beta,1):=E\times^{B} \mathbb A^1$, where the action on $\mathbb A^1$ is given by the difference of two roots $B\xrightarrow{\alpha-\beta}\mathbb G_m$,
	\item vector bundles $E(\alpha,2)=E\times^{B} \mathbb A^2$, where the action on $\mathbb A^2$ is given by the (unique) homomorphism of groups $B\xrightarrow{\alpha,u_{\alpha}} \mathbb G_m\ltimes\mathbb G_a\hookrightarrow GL_2$ such that $\alpha$ is a simple root and the composition $\g_{\alpha}\hookrightarrow\g\xrightarrow{du_{\alpha}} \operatorname{Lie}(\mathbb G_a)$ is an isomorphism.
\end{itemize}
Since the vanishing of $H^0$ is an open condition, it is enough to show that the loci in $\mathbf{Bun}_B^{d}$ 
\begin{align*}
\mt E(\alpha,1):=&\{E\in \mathbf{Bun}_B^{d}\vert H^0(C,E(\alpha,1)(p))=0\text{ for any } p\in C\},\\
\mt E(\alpha-\beta,1):=&\{E\in \mathbf{Bun}_B^{d}\vert H^0(C,E(\alpha-\beta,1)(p))=0\text{ for any } p\in C\},\\
\mt E(\alpha,2):=&\{E\in \mathbf{Bun}_B^{d}\vert H^0(C,E(\alpha,2)(p))=0\text{ for any } p\in C\}
\end{align*}
are non-empty.

\underline{Case $\mt E(\alpha,1)$, $\mt E(\alpha-\beta,1)$}. Any character $\chi:B\to\mathbb G_m$ gives a faithfully-flat morphism of stacks $\mathbf{Bun}^d_B\to \mathbf{Bun}^{\pi_1(\chi)(d)}_{\mathbb G_m}$. So, without loss of generality, we may assume $B=\mathbb G_m$ and $d\in \{-2,-1,0,1,2\}$ (here we are using $-d$ minuscule). If $d=-2$ any line bundle $L(p)$ has negative degree, hence there are no sections. Let us assume $-1\leq d\leq 2$. Consider the Abel-Jacobi map 
\begin{align*}
C^{d+1}&\to J^{d+1}\\
D'&\mapsto \oo(D').
\end{align*}The image is the locus of line bundles of degree $d+1$ with non-zero global sections. Equivalently, the image may be described as the locus of line bundles of the form $L(p)$, where $L$ is a line bundle of degree $d$ and $p$ point in $C$ such that $H^0(C,L(p))\neq 0$. The codimension of the image in $J^{d+1}$ is $\geq g- d-1\geq g-3\geq 1$. Hence, the sets $\mt E(\alpha,1)$, $\mt E(\alpha-\beta,1)$ are non-empty.

\underline{Case $\mt E(\alpha,2)$}. Any character $\chi:B\to\mathbb G_m\ltimes\mathbb G_a$ gives a faithfully flat morphism of stacks $\mathbf{Bun}^d_B\to \mathbf{Bun}^{\pi_1(\chi)(d)}_{\mathbb G_m\ltimes\mathbb G_a}$. So, without loss of generality, we may assume $B=\mathbb G_m\ltimes\mathbb G_a$, $\alpha:B\to\mathbb G_m:(t,u)\mapsto t$ and $d=\alpha(d)\in\{0,-1\}$ (here we are using $-d$ minuscule). Furthermore, a $B$-bundle $E$ is equivalent to the data of an exact sequence
\begin{equation}\label{E:ex-seq}
	0\to L\to F\to \oo_C\to 0
\end{equation}
of vector bundles, where $L=E(\alpha,1)$ and $F=E(\alpha,2)$. Twisting by $p$ and taking the global sections, we get an exact sequence of vector spaces
$$
0\to H^0(C,L(p))\to H^0(C,F(p))\to H^0(C,\oo_C(p))=k\xrightarrow{\partial} H^1(C,L(p)).
$$
By previous point, we may assume $H^0(C,L(p))=0$ for any $p\in C$. Note that $H^0(C,F(p))= 0$ if and only if the co-boundary map $\partial$ is injective. The latter condition holds if the element $v\in H^1(C,L)$, corresponding to the original extension \eqref{E:ex-seq}, is not in the kernel of the morphism $H^1(C,L)\to H^1(C,L(p))$. The kernel coincides with the line $L(p)_{|_p}$. Hence, the injectivity of $\partial$ holds if $v\notin\bigcup_{p\in C}L(p)_{|_p}=: \mathcal L$. The dimension of $\mathcal L$ is at most $2$. On the other hand, the vector space $H^1(C,L)$ has dimension $-\deg L+g-1\geq g-1\geq 3$. Hence, there exists $v\in H^1(C,L)\setminus \mathcal L$. Concluding the proof.
\end{proof}

\begin{lem}\label{L:boh}Let $E$ be a $G$-bundle in $\Md_G$ which is generic in the sense of Lemma \ref{L:gen-st}. Consider the composition
	$$
	h^{\delta}_{E,r}:H^0(C,\ad{E})=\mt T^*_E\mt M_G\subset \thig_G\xrightarrow{h} \AG\to \AG_r.
$$
	Let $\theta$ be a section in $H^0(C,\ad{E})$ and $p$ be a point in $C$. Then the following conditions are equivalent:
	\begin{enumerate}[(i)]
		\item the section $\theta$ vanishes at $p$,
		\item the section $h_{E,r}^{\delta}(\theta +\zeta)$ in $\omega^{d_r}$ vanishes at $p$, for any other global section $\zeta$ in $\ad{E}$ such that the image $h_{E,r}^{\delta}(\zeta)$ in $\omega^{d_r}$ vanishes at $p$.
	\end{enumerate}
	In symbols:
	$$
	H^0(C,\ad{E}(-p))=\left\{\theta\in H^0(C,\ad{E})\middle\vert\begin{array}{l} h_{E,r}^{\delta}(\theta +\zeta)\in H^0(C,\omega^{d_r}(-p))\\\forall\,\zeta\in  (h_{E,r}^{\delta})^{-1}\left(H^0(C,\omega^{d_r}(-p))\right)\end{array}\right\}.
	$$
\end{lem}

\begin{proof}The sequence of vector spaces
	$$
	0\to H^0(C,\ad{E}(-p))\to H^0(C,\ad{E})\xrightarrow{\operatorname{ev}_p} \ads(E_p)\to 0
	$$
is exact, by Lemma \ref{L:gen-st}. In particular, $h_{E,r}^{\delta}(\theta)|_p=p_r(\operatorname{ev}_p(\theta))=p_r(\theta(p))$. Since the restriction $\operatorname{ev}_p$ is surjective, the statement is a consequence of the following fact: if $x\in\g$ satisfies $p_r(x+y)=0$ for any $y\in \g$ such that $p_r(y)=0$, then $x=0$. Indeed, if $p_r(x+y)=0$ for any $y\in \g$ with $p_r(y)=0$, then the zero-set $\{p_r=0\}$ is a cone in $\g$ with vertex $x$. Since $p_r$ is homogeneous, we must have $x=0$.
\end{proof}

The next result describes the Hitchin morphism when restricted to the locus of $G$-Higgs bundles with sections vanishing at a fixed point of the curve.

\begin{lem}\label{L:bohi}Fix $\delta\in\pi_1(G)$ and $p\in C$. Let $U_p^\delta$ be the locally closed subscheme in $\mathcal T^*\mathcal M^\delta_G$ of those pairs $(E,\theta)$ such that $H^0(C,\mathrm{ad}(E)(p))=0$ and the section $\theta$ vanishes at $p$. Then, the following facts hold true.
\begin{enumerate}[(i)]
	\item $U_p^\delta$ is a connected smooth scheme of finite type over $k$.
	\item the image of $U_p^\delta$ along 
	the Hitchin morphism $h^\delta:\mathcal T^*\mathcal M^\delta_G\to \AG$ is dense in $\oplus_{i=1}^rH^0(C,\omega^{d_i}(-d_ip))$.
\end{enumerate}	
In particular, $\oplus_{i=1}^rH^0(C,\omega^{d_i}(-d_ip))$ is equal to the linear subspace generated by the image $h^\delta((U^\delta_p)')$, where $(U^\delta_p)'$ is any open subscheme of $U_p^\delta$.
\end{lem}

\begin{proof} Consider the morphism $F:U_p^\delta\to\mathcal M^\delta_G:(E,\theta)\mapsto E$, which forgets the section. The vanishing $H^0(C,\mathrm{ad}(E)(p))=0$ implies that the morphism $F$ makes the source a vector bundle over its image, which is an open and dense (by Lemma \ref{L:gen-st}) subscheme of the connected smooth variety $\mathcal M_G^\delta$. In particular, we have the first point. We now focus on the second point. Set $\mathcal L:=\omega(-p)$ and $\mathfrak c_\mathcal L:=\mathfrak c\times^{\mathbb G_m}\mathcal L^*$.  Let $\hit_{G,\mathcal L}$ be the moduli stack of $\mathcal L$-twisted $G$-Higgs bundles introduced in Subsection \ref{SS:twisted}. Consider the $\mathcal L$-twisted Hitchin fibration
$$
h_\mathcal L:\hit_{G,\mathcal L}\to \AG_{\mathcal L}=\oplus_{i=1}^rH^0(C,\omega^{d_i}(-d_ip)).
$$
We denote by $V^\delta_p$ the open substack of those $\mathcal L$-twisted $G$-Higgs bundles $(E,\theta)$ such that $E$ is a $G$-bundle in $\BG_G^\delta$ and $H^0(C,\mathrm{ad}(E)(p))\cong H^1(C,\mathrm{ad}(E)\otimes\mathcal L)=0$. Arguing as for $U_p^\delta$, one can show that $V_p^\delta$ is a vector bundle over an open and dense subset of the smooth and connected stack $\BG_G^\delta$. In particular, $V_p^\delta$ is a connected smooth algebraic stack. For proving the point, it is enough to show that the image $h^\delta_\mathcal L(V_p^\delta)$ is dense in $\AG_\mathcal L$. Without loss of generality, we may assume $G$ adjoint. We make the following claim
\begin{equation}\label{E:claim-V}
h_\mathcal L(V^\delta_p)\cap \AG^\heartsuit_{\mathcal L}\neq \emptyset,
\end{equation}
where $\AG^\heartsuit_\mathcal L:=\{a\in\AG_\mathcal L|\,a(C)\nsubseteq \Dg_{\mathcal L}\}$. The claim \eqref{E:claim-V} implies the lemma. Indeed, by \eqref{E:claim-V}, the intersection  $V_p^{\delta}\cap \hit_{G,\mathcal L}^{\heartsuit}$ and $\AG_\mathcal L^\heartsuit$ are non-empty (where $\hit_{G,\mathcal L}^{\heartsuit}:=h^{-1}_\mathcal L(\AG_\mathcal L^\heartsuit)$). By Corollary \ref{C:twisted}$(\ref{i:bohi2})$, the intersection $V_p^\delta\cap\hit_{G,\mathcal L}^{\mathrm{reg},\heartsuit}$ is dense in $V_p^{\delta}\cap \hit_{G,\mathcal L}^{\heartsuit}$. By Corollary \ref{C:twisted}$(\ref{i:bohi1})$, the composition $V_p^\delta\cap\hit_{G,\mathcal L}^{\mathrm{reg},\heartsuit}\hookrightarrow\hit_{G,\mathcal L}^{\mathrm{reg},\heartsuit}\to \AG^{\heartsuit}_{\mathcal L}$ is smooth (hence dominant) and, so, we have the lemma. It remains to show \eqref{E:claim-V}. Let $B$ a Borel subgroup of $G$ containing the maximal torus $T$. We claim that there exists a $B$-bundle $E$ in $\BG_B^d$ and $q\in C$ such that \begin{enumerate}[(a)]
	\item\label{i:ev1} $[d]=\delta\in\pi_1(G)$ and $H^0(C,(E\times^B\g)(p))=0$;
	\item\label{i:ev3} the morphism $$\nu_q:H^0(C,(E\times^B\mathfrak{b})\otimes\mathcal L)\xrightarrow{\mathrm{ev}_{q}} ((E\times^B\mathfrak{b})\otimes\mathcal L)_q=\mathfrak{b}\twoheadrightarrow((E\times^B(\mathfrak{b}/\mathfrak n))\otimes\mathcal L)_q=\mathfrak b/\mathfrak n=\mathfrak t$$ is surjective.
\end{enumerate}
Assume that this $B$-bundle $E$ exists. By Point \eqref{i:ev3}, there exists a section $\theta\in H^0(C,(E\times^B\mathfrak{b})\otimes\mathcal L)\subset H^0(C,(E\times^B\mathfrak{g})\otimes\mathcal L)$ such that $\nu_q(\theta)\in\tg^{\mathrm{reg}}:=\{t\in\tg|\,\chi(t)\in \mathfrak c\setminus\mathfrak{D}\}$, where $\chi:\g\to\mathfrak c=\g/G$ is the morphism introduced in Theorem \ref{T:konstant}.

By Point \eqref{i:ev1}, $(E\times^BG,\theta)\in V_p^\delta$. By Theorem \ref{T:konstant}$(\ref{T:konstant6})$, we have that $\mathfrak b/B\cong \mathfrak t$. In particular, for any $x\in\tg$ and $y\in\mathfrak n$, the elements $x+y$ and $x$ are in the closure of the same $G$-orbit. Hence, $\chi(x+y)=\chi(y)$. If we set $x+y=\mathrm{ev}_q(\theta)\in\mathfrak b$ and $x= \nu_q(\theta)\in\mathfrak t$, we have $\chi(\nu_q(\theta))=\chi(\mathrm{ev}_q(\theta))$. By construction, the section $a:=h^\delta(E\times^BG,\theta)$ is in $\AG_\mathcal L^{\heartsuit}$ (because $a(q)=\chi(\mathrm{ev}_q(\theta))=\chi(\nu_q(\theta))\notin \mathfrak D$) and, so, we have \eqref{E:claim-V}. 

It remains to show the existence of a $B$-bundle with the required properties. From the proof of Lemma \ref{L:gen-st}, we know that there exists a lift $d$ of $\delta$ and an open substack $\mathcal S$ in $\BG_B^d$ such that any $B$-bundle in $\mt S$ satisfies \eqref{i:ev1}. We now show that there exists a point $q\in C$ and a generic $B$-bundle in $\mathcal S$ satisfying \eqref{i:ev3}. Fix a point $q\in C$ and a $B$-bundle $E$ in $\mathcal S$. The nilpotent subalgebra $\mathfrak n\subset\mathfrak b$ is preserved by the action of $B$ and the induced $B$-action on the quotient $\mathfrak b/\mathfrak n=\mathfrak t$ is trivial. In particular, we obtain an exact sequence of vector bundles
\begin{equation}\label{E:nt}
0\to (E\times^B\mathfrak n)\otimes\mathcal L\to (E\times^B\mathfrak b)\otimes \mathcal L\to (E\times^B(\mathfrak b/\mathfrak n))\otimes \mathcal L=\mathfrak t\otimes_k\mathcal L\to 0.
\end{equation}
Following the notation in the proof of Lemma \ref{L:gen-st}, the first vector bundle in \eqref{E:nt} is an iterated extension of line bundles of type $E(\alpha,1)\otimes\mathcal L$, where $\alpha$ is a positive root. By the proof of loc.cit and Serre duality, we may assume that $H^1(E(\alpha,1)\otimes\mathcal L)=H^0(E(-\alpha,1)(p))=0$ and, so, $H^1((E\times^B\mathfrak n)\otimes\mathcal L)=0$. The latter vanishing implies that the morphism of global sections $H^0(C,E\times^B\mathfrak b)\otimes \mathcal L)\to H^0(C,\mathfrak t\otimes_k\mathcal L)$, induced by the sequence \eqref{E:nt}, is surjective. Hence, it is enough to show the surjectivity of the evaluation map $\mathrm{ev}_q^{\mathfrak t}:H^0(C,\mathfrak t\otimes_k\mathcal L)\to \tg\otimes_k\mathcal L_q$. By definition of the bundle $\mathcal L=\omega(-p)$, the surjectivity of $\mathrm{ev}_q^{\mathfrak t}$ follows from the surjectivity of the evaluation morphism $H^0(C,\omega)\to \omega_p\oplus\omega_q$ of the canonical line bundle at the points $q$ and $p$. If $C$ is non-hyperelliptic the morphism is surjective for any choice of two points in $C$. If $C$ is hyperelliptic and $p$ is a fixed point, the surjectivity still holds for a generic choice of $q$.
\end{proof}

We are now finally ready to show that the linear isomorphism $L_\varphi:\AG\to\AG$ introduced in Lemma \ref{L:vp=linear} is induced by an isomorphism of the vector bundle $\mathfrak c_\omega$. More precisely,

\begin{prop}\label{P:vp=scalar}Let $\varphi$ be an automorphism of $\mt M^{\delta}_G$ and  $\varphi_{\AG}$ be the induced automorphism on the Hitchin basis, given by \eqref{E:morf}. Then there exists an isomorphism $\psi$ of $\cg_{\omega}$ (as vector bundle $\oplus_{i=1}^r \omega^{d_i}$) and $\sigma \in\operatorname*{Aut}(C)$ such that 
		$$
		\varphi_{\AG}=H^0(\psi)\circ\sigma^*,
		$$
where $H^0(\psi)$ is the isomorphism on the global sections of $\cg_{\omega}$ induced by $\psi$.
\end{prop}

\begin{proof}Arguing as in Lemma \ref{L:vp=linear}, we may assume $\varphi_{\AG_r}=\lambda\circ \Id_{\AG_r}$ for some $\lambda\in k\setminus\{0\}$ and $\varphi_\AG=L_\varphi$ linear isomorphism. Let $E$ be a generic simple and stable $G$-bundle and set $E':=\varphi^{-1}(E)$. By Lemma \ref{L:gen-st}, we may assume that $H^0(C,\mathrm{ad}(E)(p))=H^0(C,\mathrm{ad}(E')(p))=0$, for any $p\in C$. By Lemma \ref{L:boh}, the co-differential
	$$
	d\varphi^*_E:H^0(C,\ad{E})\to H^0(C,\ad{E'})
	$$
identifies the linear subspaces $H^0(C,\ad{E}(-p))$ and $H^0(C,\ad{E'}(-p))$ for any point $p\in C$. In particular, for any point $p\in C$, $d\varphi^*$ preserves an open and dense subset of the locally closed subscheme $U_p^\delta\subset T^*\mathcal M_G^\delta$ of those pairs $(E,\theta)$ such that $H^0(C,\mathrm{ad}(E)(p))=0$ and the section $\theta$ vanishes at $p$. By Lemma \ref{L:bohi}, the linear isomorphism $\varphi_{\AG}:\AG\to\AG$ preserves the linear subspace $
\oplus_{i=1}^rH^0(C,\omega^{d_i}(-d_ip))\subset \AG
$, for any $p\in C$.
By moving the point $p$, we get a direct sum of osculating bundles
$$
\bigoplus_{i=1}^r\op{Osc}(d_i)=\bigoplus_{i=1}^r\left(\bigcup_{p\in C}H^0(C,\omega^{d_i}(-d_ip))\right).
$$
It is a vector bundle of rank $\sum_i\dim\AG_i-\sum_i d_i=\dim\AG- (r+|\Phi|/2)$ (here we used that $g\geq 3$). Furthermore, we have a commutative diagram of vector bundles
\begin{equation*}
	\xymatrix{
		0\ar[r]& \bigoplus_{i=1}^r\op{Osc}(d_i)\ar[d]\ar[r]&\AG\otimes\oo_C\ar[d]^{\varphi_{\AG}\otimes \Id_{\oo_C}}\ar[rr]^-{(\phi^{d_1},\ldots,\phi^{d_r})}&&\bigoplus_{i=1}^r\mt P^{d_i-1}(\omega^{d_i})\ar[r]\ar[d]^{\widetilde\psi}&0\\
				0\ar[r]& \bigoplus_{i=1}^r\op{Osc}(d_i)\ar[r]&\AG\otimes\oo_C\ar[rr]^-{(\phi^{d_1},\ldots,\phi^{d_r})}&&\bigoplus_{i=1}^r\mt P^{d_i-1}(\omega^{d_i})\ar[r]&0
	}
\end{equation*}
where the rows are exact and the columns are isomorphisms. The sheaf $\mt P^{d_i-1}(\omega^{d_i})$ is the bundle of principal parts of degree $d_i-1$ of the line bundle $\omega^{d_i}$. It is a vector bundle of rank $d_i$. The morphism $\phi^{d_i}$ is defined locally, on a formal disc $\Delta_p:=\Spec(\widehat\oo_{C,p})\cong\Spec(\disk{t})$ around a point $p$ of $C$, as follows:
\begin{equation}\label{E:bun-parts}
\def\arraystretch{1.5}\begin{array}{ccl}
	\disk{t}^{\oplus\dim\AG_i}&\longrightarrow& \disk{t}^{\oplus d_i}\\
(b_1,\ldots,b_{\dim\AG_i})&\longmapsto&\displaystyle \left(\sum_{j=1}^{\dim\AG_i}b_jf_j,\sum_{j=1}^{\dim\AG_i}\frac{d(b_jf_j)}{dt},\ldots,\sum_{j=1}^{\dim\AG_i}\frac{d^{d_i-1}(b_jf_j)}{dt^{d_i-1}}\right),
\end{array}
\end{equation}
where the map $(b_1,\ldots,b_{\dim\AG_i})\mapsto \sum_{i}b_if_i$ is the local expression of the evaluation morphism $\op{ev}:\AG_i\otimes\oo_C\twoheadrightarrow\omega^{d_i}$. The bundle of principal parts admits a canonical surjective morphism
$
\pi_i:\mt P^{d_i-1}(\omega^{d_i})\to \omega^{d_i}
$ of vector bundles.
With the local interpretation in \eqref{E:bun-parts}, the morphism $\pi_i$ is nothing but the projection on the first factor, i.e. $\disk{t}^{\oplus d_i}\to\disk{t}:(t_1,\ldots,t_{d_i})\mapsto t_1$. We denote by $\pi$ the morphism $(\pi_1,\ldots,\pi_{r}):\oplus_i\mt P^{d_i-1}(\omega^{d_i})\to \cg_{\omega}$. On the other hand, each projection $\pi_i$ admits a canonical section $
\delta_i:\omega^{d_i}\to \mt P^{d_i-1}(\omega^{d_i})
$, which is a $k$-linear morphism of abelian sheaves (not necessarily of $\oo_C$-modules). With the local interpretation in \eqref{E:bun-parts}, the section sends a polynomial in $\disk{t}$ in his first $d_i-1$ derivatives, i.e.
$$
\disk{t}\to\disk{t}^{\oplus d_i}:f\mapsto\left(f,\frac{df}{dt},\ldots, \frac{d^{d_i-1}f}{dt^{d_i-1}}\right).
$$ 
We denote by $\delta$ the inclusion $(\delta_1,\ldots,\delta_r):\cg_{\omega}\hookrightarrow \oplus_{i}\mt P^{d_i-1}(\omega^{d_i})$. We claim that the diagram
\begin{equation}\label{E:claim-princ}
	\xymatrix{
		0\ar[r]& \ker(\op{ev})\ar[d]\ar[r]&\AG\otimes\oo_C\ar[d]^{\varphi_{\AG}\otimes \Id_{\oo_C}}\ar[r]^{\operatorname{ev}}&\cg_{\omega}\ar[r]\ar[d]^{\psi:=\pi\circ \widetilde\psi\circ \delta}&0\\
		0\ar[r]&\ker(\op{ev})\ar[r]&\AG\otimes\oo_C\ar[r]^{\operatorname{ev}}&\cg_{\omega}\ar[r]&0
	}
\end{equation}
commutes, the rows are exact and the columns are isomorphisms of vector bundles. Note that if such a diagram exists, we must have $H^0(\psi)=\varphi_{\AG}$, and so the assertion follows.

It remains to prove the claim. The exactness of the rows and the commutativity of the diagram are obvious. We now check that the composition $\psi$ is well-defined and bijective. A priori, $\psi$ is a $k$-linear morphism of abelian sheaves. Using the local descriptions of $\pi_i$ and $\delta_i$, one may check that it is also $\oo_C$-linear. Since the column in the middle is an isomorphism, the morphism $\psi$ is surjective. Hence, it is also injective, because $\mathfrak c_\omega$ is a vector bundle. By snake lemma, we get that also the left-hand side column is an isomorphism and, so, we have the claim.
\end{proof}

\begin{teo}\label{T:T}Fix $\delta\in\pi_1(G)$. The (injective) homomorphism of groups of Proposition \ref{P:aut-m}
		\begin{equation}
		\def\arraystretch{1.5}\begin{array}{cll}
			H^1(C,\mathscr Z(G))\rtimes\left(\Aut(C)\times \Out(G,\delta)\right)&\lhook\joinrel\longrightarrow& \Aut(\overline{\mt M}_G^{\delta}(C))\\
			\nonumber(Z,\sigma,\rho)&\longmapsto &\left(E\mapsto \rho((\sigma^{-1})^*E)\otimes Z\right)
		\end{array}
	\end{equation}
is surjective. The same holds if we replace $\overline{\mt M}_G^{\delta}$ with the locus $\mt M_G^{\delta}$ of simple and stable $G$-bundles.
\end{teo}

\begin{rmk}Before presenting the proof, we would like to comment about the hypotheses of the theorem. In this section (and, so, in Theorem \ref{T:T}), we assumed that the genus of the curve must be at least $4$. 
The reason of the bound is merely technical: it has been used for showing the vanishing of $H^0(C,\ads(E)(p))$ in Lemma \ref{L:gen-st}. The proofs of the other results in this section, apart from citing that lemma, are still correct, if we assume $g\geq 3$. In other words, if one can show the Lemma \ref{L:gen-st} under these weaker hypotheses, the Theorem \ref{T:T} will extend immediately.
\end{rmk}

\begin{proof} Let $\varphi:\mt M^{\delta}\to\mt M^{\delta}$ be an automorphism. Fix a generic $G$-bundle $E$ and set $E'=\varphi^{-1}(E)$. By Proposition \ref{P:vp=scalar}, we have a cartesian diagram of affine spaces
\begin{equation}
	\xymatrix{
		H^0(C,\ad{E})\ar[rr]^{d\varphi_E^*}\ar[d]^{h_E}&&H^0(C,\ad{E'})\ar[d]^{h_E}\\
		\AG\ar[rr]^{\varphi_{\AG}=H^0(\psi)\circ\sigma^*}&&\AG	
	}
\end{equation}
where $\psi:\cg_{\omega}\cong\cg_{\omega}$ is an isomorphism of vector bundles. Without loss of generality, we may assume $\sigma=\Id_C$.  In particular, the codifferential $d\varphi^*_E$ identifies the subsets
\begin{eqnarray}
	\label{E:sub1}(h_{E}^{\delta})^{-1}\left(H^0(C,\cg_{\omega}(-p))\right)\subset H^0(C,\ad{E}),\\
	\label{E:sub2}(h_{E'}^{\delta})^{-1}\left(H^0(C,\cg_{\omega}(-p))\right)\subset H^0(C,\ad{E'}).
\end{eqnarray}
Observe that the subset \eqref{E:sub1} is the inverse image of the nilpotent cone 
$$
\mt N_{E,p}=\{x\in\ad{E_p}_p=\g\vert\text{ the adj. end. } [x,-]: \g\to\g\text{ is nilpotent}\}\equiv\mt N,
$$
under the evaluation $H^0(C,\ad{E})\xrightarrow{\operatorname{ev_p}} \ad{E_p}_p$ at the fiber over $p$ (the same holds for $E'$). On the other hand, we have a commutative diagram of vector bundles
	\begin{equation}\label{E:mtE-st}
		\xymatrix{
			0\ar[r] &\mt E\ar[r]\ar[d]^{\cong} &H^0(C,\ad{E})\otimes \oo_C\ar[d]^{\cong}\ar[r] &\ad{E}\ar[r]\ar[r]\ar@{-->}[d]^{\Psi}&0\\
			0\ar[r] &\mt E\ar[r] &H^0(C,\ad{E'})\otimes \oo_C\ar[r] &\ad{E'}\ar[r]&0.
		}
	\end{equation}
	where:
	\begin{itemize}
		\item the rows are exact (see Lemma \ref{L:gen-st}); 
		\item the column in the middle is given by the co-differential $d\varphi_E^*$ at $E$;
		\item the left-hand side column exists because $d\varphi_E^*$ sends the vector space $\mt E_p=H^0(C,\ad{E}(-p))$ to $\mt E'_p=H^0(C,\ad{E'}(-p))$ (see Lemma \ref{L:boh});
		\item the existence of the right-hand side column $\Psi$ is guaranteed by the previous morphisms.
	\end{itemize} 
	In particular, the linear morphism $d\varphi^*_E$ is induced by the isomorphism of vector bundles
	\begin{equation}\label{E:coddif-vecE}
		\Psi:\ads(E)\to \ads(E').
	\end{equation}
By the previous discussion, the isomorphism $\Psi$ preserves the bundles of nilpotent cones of $\ad{E}$ and $\ad{E'}$. In symbols, we have a commutative diagram
\begin{equation*}\label{E:nc}
	\xymatrix{
	\mt N_E\ar[dr]\ar[rr]_{\Psi}^{\sim}&&\mt N_{E'}\ar[dl]\\
	&C&	
	}
\end{equation*}
By Lemma \ref{L:nc}, we get an isomorphism of flag bundles
\begin{equation*}\label{E:flg}
	\xymatrix{
		E/B=E\times^GG/B\ar[dr]_{f}\ar[rr]_{\widehat{\Psi}}^{\sim}&&E'\times^GG/B=E'/B\ar[dl]^{f'}\\
		&C&
	}
\end{equation*}
Taking the relative tangent bundles, we get an isomorphism of vector bundles
\begin{equation}\label{E:tg}
	\xymatrix{
		\ads(E)=\mt T_f\ar[dr]\ar[rr]_{d\widehat{\Psi}}^{\sim}&&\mt T_{f'}=\ads(E')\ar[dl]\\
		&C&
	}
\end{equation}
The equalities between the adjoint bundles and the tangent bundles follows from Lemma \ref{L:aut-flag}. Note that, the isomorphism in \eqref{E:tg} is an isomorphism of Lie-algebra bundles, i.e. it commutes with the Lie-brackets. By Lemma \ref{L:ad-tau} below, we must have $E'\cong \rho(E)\otimes Z$ for some $Z\in H^1(C,\mathscr Z(G))$ and $\rho\in \Out(G,\delta)$. Since $E$ is generic, the equality holds for any other simple and stable $G$-bundle. The result for the entire moduli space $\overline{\mt M}^{\delta}$ follows because the restriction homomorphism 
$
\Aut(\overline{\mt M}^{\delta})\to \Aut(\mt M^{\delta})
$ is bijective (see Proposition \ref{P:aut-finite}).
\end{proof}

\begin{lem}\label{L:ad-tau} Let $E$ and $E'$ be two $G$-bundles in the same connected component $\mt M^{\delta}_G$, for some $\delta\in\pi_1(G)$, such that $	\ads(E)$ and $\ads(E')$ are isomorphic as Lie-algebra bundles over $C$. Then there exists $Z\in H^1(C,\mathscr Z(G))$ and $\rho\in\Out(G,\delta)$ such that $E'\cong \rho(E)\otimes Z$. 
\end{lem}

\begin{proof}We first study the set of $G$-reductions of the Lie algebra bundle $\ads(E)$. Note that having a Lie-algebra bundle (with respect to $\g$) is the same as having an $\operatorname{Aut}(\g)$-bundle, where $\Aut(\g)$ is the group of $k$-linear isomorphisms of $\g$ commuting with the Lie bracket. We divide the study in three steps, based on the following homomorphism of groups
	\begin{equation*}
		G\twoheadrightarrow G^{\operatorname{ad}}=\operatorname{Inn}(\g)=\operatorname{Inn}(G)\hookrightarrow\operatorname{Aut}(\g),
	\end{equation*}
	where $\operatorname{Inn}(\g)$ is the connected component of the identity of $\operatorname{Aut}(\g)$ and it coincides with the subgroup of inner automorphisms of $\g$.\\
	\underline{First Step:} $G^{\operatorname{ad}}$-reductions of the $\operatorname{Aut}(\g)$-bundle $\ads(E)$. We have an exact sequence of groups
	\begin{equation*}
		1\to G^{\ads}=\operatorname{Inn}(\g)\to \operatorname{Aut}(\g)\to \operatorname{Out}(\g)\to 1,
	\end{equation*}
	where the right-hand side is the group of outer automorphisms of $\g$. The $G^{\operatorname{ad}}$-reductions of $\ads(E)$ are in bijection with the sections of the $\operatorname{Out}(\g)$-bundle $E\times^G\Out(\g)\equiv \ads(E)/\operatorname{Inn}(\g)\to C$. Note that the $\operatorname{Out}(\g)$-bundle admits sections (because $\ads(E)$ has a $G$-reduction: namely $E$), then it must be trivial. In particular, two $G^{\operatorname{ad}}$-reductions of $F$ differs by an element of $\operatorname{Out}(\g)$.\\
	\underline{Second Step:} $G$-reductions of an arbitrary $G^{\operatorname{ad}}$-bundle. From the exact sequence of groups
	\begin{equation*}
		1\to\mathscr Z(G)\to G\to G^{\operatorname{ad}}\to 1,
	\end{equation*}
	we obtain an exact sequence of pointed sets
	\begin{align*}
		1&\to H^1(C,\mathscr Z(G))\to H^1(C,G)\to H^1(C,G^{\operatorname{ad}})
	\end{align*}
	In particular, if the set of $G$-reductions of a $G^{\operatorname{ad}}$-bundle is non-empty, it is a torsor under the group $H^1(C,\mathscr Z(G))$. More precisely, if $E$ is a $G$-reduction for an arbitrary $G^{\operatorname{ad}}$-bundle $E^{\operatorname{ad}}$, the other $G$-reductions of $E^{\operatorname{ad}}$ have the form $
	E\otimes Z
	$, where $Z$ is $\mathscr Z(G)$-torsor over the curve $C$.\\
	\underline{Third Step:} The action of $\Out(\g)$ on $\mt M_G$. We recall that $\Out(G)\subset \Out(\g)=\Out(G^{\operatorname{sc}})$, where $G^{\op{sc}}\to G$ is the universal cover of $G$. More precisely, if $\rho\in \Out(G^{\operatorname{sc}})\setminus \Out(G)$, then $\rho:G^{\operatorname{sc}}\to G^{\operatorname{sc}}$ does not descend to an automorphism of $G$. In particular, it does not induce an automorphism of the moduli space $\M_G$. Moreover, the group $\Out(G)$ acts (possibily non-trivially) on the sets of connected components of $\mt M_G$ (see Lemma \ref{L:mov-con}). The subgroup fixing the connected component $\mt M^{\delta}$ is precisely the subgroup $\Out(G,\delta)$. \\
	Putting all together, the assertion follows immediately.
\end{proof}

\section{The automorphism group of $\overline{\mt M}_G$: an explicit description.}\label{Sec:comp}
In this section, using Theorem \ref{T:A}, we provide the explicit presentations of the automorphism group of $\smt_G^{\delta}$ contained in Corollary \ref{T:conti}. The key point is the explicit computation of the following objects:
\begin{enumerate}[(a)]
	\item the character group $\op{Hom}(\mathscr Z(G),\mathbb G_m)$ of the center,
	\item the fundamental group $\pi_1(G)$,
	\item the outer automorphism group $\Out(G)$,
	\item the action of $\Out(G)$ on $\op{Hom}(\mathscr Z(G),\mathbb G_m)$ and $\pi_1(G)$.
\end{enumerate}
We will complete this task in the next subsections. Our references are \cite{Frin} for (a), (b) and \cite{BuLie} for (c), (d). We now briefly explain the strategy.

The group $\Out(G)$, when $G$ is simply-connected and almost-simple, is equal to the automorphism group of its Dynkin diagram, which is well-known. Roughly speaking, each element in $\Out(G)$ admits a representative $\rho\in\Aut(G)$, which fixes a Borel subgroup $B$, its maximal torus and an additional structure given by a certain kind of basis of the nilpotent subalgebra $\mathfrak n$ in the Lie algebra $\mathfrak b$ of $B$ (the reductive groups together with these data are called \emph{pinned reductive groups}). 

Because of these properties, the isomorphism $\rho$ induces an isomorphism on the (co)character lattices of the maximal torus and it preserves the set of roots, coroots and simple roots. Hence, it induces an automorphism of the (based) root data and also an automorphism of the Dynkin diagram. Then, we use the description in \cite{BuLie} of what these automorphisms do on the root data, for describing the action of $\rho$ on the $\op{Hom}(\mathscr Z(G),\mathbb G_m)$ and $\pi_1(G)$.

More in general, if $G$ is an almost-simple group, it admits an isogeny $G^{\operatorname{sc}}\to G$, with kernel $\mu$, from the (unique) simply-connected, almost-simple group of the same type. Furthermore, the isogeny gives the following identification of groups
\begin{equation}\label{E:act-out}
	\Out(G)=\{\rho\in\Out(G^{\operatorname{sc}})\vert \rho(\mu)\subset\mu\}\subset\Out(G^{\operatorname{sc}}).
\end{equation}
In particular, $\Out(G^{\ads})=\Out(G^{\operatorname{sc}})$. Hence, knowing the action of $\Out(G^{\operatorname{sc}})$ on the center of $G^{\operatorname{sc}}$ will give the description of $\Out(G)$. On the other hand, knowing the action of $\Out(G)$ on $\pi_1(G)$ will give the description of $\Out(G,\delta)$ for any $\delta\in\pi_1(G)$.

In this section, $G$ is an almost-simple group and $C$ is a smooth projective curve of genus $g\geq 4$. For any positive integer $l$, we denote by $\Pic(C)[l]$ the group of $l$-torsion line bundles over $C$. Note that, there exists an isomorphism of groups $\Pic(C)[l]\cong (\mathbb Z/l\mathbb Z)^{2g}$.

\subsection{Types $A_{n-1}$}
The unique simply-connected and almost-simple group of type $A_{n-1}$, $n\geq 2$, is $\op{SL}_n$. It is well-known that $\Out(\op{SL}_n)$ is trivial if $n=2$ and cyclic of order two (i.e. $\mathbb Z/2\mathbb Z$) if $n\geq 3$. In the latter case, a representative of the generator is the inverse-transpose isomorphism: \begin{equation}\label{E:trin}
	\op{SL}_n\to\op{SL}_n:M\mapsto (M^{-1})^t.
\end{equation}
All the almost-simple groups of type $A_{n-1}$ are isomorphic to the quotient $\op{SL}_n/\mu_r$, where $r$ divides $n$. Hence, the non-trivial outer automorphism \eqref{E:trin} descends to an element in $\Out(\op{SL}_n/\mu_r)$ for any $r$ dividing $n$. From \cite[Section 7.1]{Frin}, we deduce the following table.
$$
{\small
\setlength\arraycolsep{1.5pt}\begin{array}{|c|c|c|c|c|}
	\hline
	\g&G&\Out(G)&\operatorname{Hom}(\mathscr Z(G),\mathbb G_m)&\pi_1(G)\\
	\hline
	A_1& \operatorname{SL}_{2}=\op{SL}_2/\mu_1& \{1\}& \mathbb Z/2\mathbb Z& \{0\}\\
	& \operatorname{PSL}_{2}=\op{SL_2}/\mu_2& \{1\}& \{0\}& \mathbb Z/2\mathbb Z\\
	\hline
	A_{n-1},\;\scriptstyle{n\geq 3}& \operatorname{SL}_{n}/\mu_r,\;\scriptstyle{\text{s.t. }r|n}& \mathbb Z/2\mathbb Z& \mathbb Z/\frac{n}{r}\mathbb Z& \mathbb Z/r\mathbb Z\\
	\hline
\end{array}}
$$
Furthermore, the group $\Out(G)=\bbZ/2\bbZ$ acts on $\operatorname{Hom}(\mathscr Z(G),\mathbb G_m)$ and $\pi_1(G)$ by taking the inverse, i.e. $[1].x=-x$. Moreover, there exists an $\Out(G)$-equivariant isomorphism
$
H^1(C,\mathscr Z(\op{SL_n}/\mu_r))=\Pic(C)[n/r],
$
where $\Out(G)$ acts on the right-hand side by taking the dual, i.e. $[1].L=L^{-1}$. Putting all together, we have the following theorem.

\begin{teo}\label{T:An}Let $G$ be an almost-simple group of type $A_{n-1}$ ($n\geq 2$), i.e. $G=\op{SL}_n/\mu_r$, for some $r$ dividing $n$. 
		\begin{enumerate}[(i)]
		\item Type $A_1$. The isomorphism of Theorem \ref{T:A} gives the following equalities
			\begin{equation*}
			\arraycolsep=0.3pt\Aut(\overline{\mt M}_{G}(C)^{\delta})=\left\{\begin{array}{ll}\Pic(C)[2]\rtimes\Aut(C),& \text{if }G=\op{SL}_2\text{ and }\delta\in\{0\},\\
				\Aut(C),& \text{if }G=\op{PGL}_2\text{ and }\delta\in\bbZ/2\bbZ,
			\end{array}\right.
		\end{equation*}
		where $\Aut(C)$ acts on $\Pic(C)[2]$ by pull-back. 
		\item Type $A_{n-1}$, $n\geq 3$. The isomorphism of Theorem \ref{T:A} gives the following equalities
			\begin{equation*}
			\arraycolsep=0pt\Aut(\overline{\mt M}_{\frac{\operatorname{SL}_n}{\mu_r}}(C)^{\delta})=\left\{\begin{array}{ll}
				\Pic(C)[n/r]\rtimes \left(\mathbb Z/2\times\Aut(C)\right),&\text{ if } 2\delta=0\in\mathbb Z/r\mathbb Z,\\
				\Pic(C)[n/r]\rtimes\Aut(C),&\text{ if }2\delta\neq0\in \mathbb Z/r\mathbb Z,
			\end{array}\right.
		\end{equation*}
		where $\Aut(C)$ acts on $\Pic(C)[n/r]$ by pull-back and $\mathbb Z/2\mathbb Z$ acts on $\Pic(C)[n/r]$ by taking the dual, i.e. $[1].L=L^{-1}=L^{\frac{n}{r}-1}$. 
	\end{enumerate}
The same holds for the locus $\mt M_G$ of simple and stable bundles.
\end{teo}
\subsection{Types $B_{n}$ and $C_n$}
An almost-simple group of type $B_n$, $n\geq 2$, is isomorphic to either the (simply-connected) spin group $\op{Spin_{2n+1}}$ or the (adjoint) special orthogonal group $\op{SO}_{2n+1}$. On the other hand, an almost-simple group of type $C_n$, $n\geq 3$, is isomorphic to either the (simply-connected) symplectic group $\op{Sp_{2n+1}}$ or the (adjoint) projective symplectic group subgroup $\op{PSp}_{2n+1}$. All the automorphisms of these groups are inner, see \cite[Chap.VI, \S 4.5, 4.6]{BuLie}. A description of their centers and their fundamental groups can be found in \cite[Section 7.2]{Frin}. Putting all together, we get the following table.
$$
{\small
	\setlength\arraycolsep{1.5pt}
\begin{array}{|c|c|c|c|c|}
	\hline
	\g&G&\Out(G)&\operatorname{Hom}(\mathscr Z(G),\mathbb G_m)&\pi_1(G)\\
	\hline
	B_n,\;\scriptstyle{n\geq 2}& \operatorname{Spin}_{2n+1} & \{1\} & \mathbb Z/2\mathbb Z & \{0\}\\
	& \operatorname{SO}_{2n+1}	&	\{1\} & \{0\} & \mathbb Z/2\mathbb Z\\
	\hline
	C_n,\;\scriptstyle{n\geq 3}& \operatorname{Sp}_{2n} & \{1\} & \mathbb Z/2\mathbb Z & \{0\}\\
	& \operatorname{PSp}_{2n}	&	\{1\} & \{0\} & \mathbb Z/2\mathbb Z\\
	\hline
\end{array}}
$$
Arguing as in the previous case, from the table we deduce the following result.
\begin{teo}\label{T:BnCn} Let $G$ be an almost-simple group of type $B_n$ ($n\geq 2$) or $C_n$ ($n\geq 3$). The isomorphism of Theorem \ref{T:A} gives the following equalities
	$$
	\arraycolsep=0pt\Aut(\smt_G^{\delta}(C))=\left\{\begin{array}{lll}
		\Pic(C)[2]\rtimes\Aut(C),&\text{ if }G=\op{Spin}_{2n+1},\op{Sp}_{2n} &\text{ and }\delta\in\{0\},\\
		\Aut(C),&\text{ if }G=\op{SO}_{2n+1},\op{PSp}_{2n}&\text{ and }\delta\in\mathbb Z/2\mathbb Z,
	\end{array}\right.
	$$
	where $\Aut(C)$ acts on $\Pic(C)[2]$ by pull-back. The same holds for the locus $\mt M_G$ of simple and stable bundles.
\end{teo}

\subsection{Types $D_n$}
This case is the most delicate. Let us first recall some properties of the root system $D_{n}$, following \cite[\S 7.3]{Frin}. Consider the vector space $\mathbb R^n$ endowed with the standard scalar product $(-,-)$ and with the canonical bases $\{\epsilon_1,\ldots, \epsilon_n\}$.  
We will freely identify $\mathbb R^n$ with its dual vector space by mean of the (restriction of the) standard scalar product $(-,-)$. 
The root (resp. coroot) lattices and  the weight (resp. coweight) lattices of $D_{n}$ are given by 
\begin{equation*}
	Q(D_{n})=Q(D_{n}^\vee)=\{\xi\in \bbZ^n \: : (\xi,\xi)\: \text{ is even}\}  \subset P(D_{n})=P(D_{n}^\vee)= \bbZ^n +\left\langle \frac{\sum_i \epsilon_i}{2}\right\rangle.
\end{equation*}
We set 
$$
w_n:=\frac{\epsilon_1+\ldots +\epsilon_n}{2} \quad \text{ and } \quad w_{n-1}:=\frac{\epsilon_1+\ldots+\epsilon_{n-1}-\epsilon_n}{2}.
$$
The group $P(D_n)/Q(D_n)$ is equal to 
\begin{equation*} 
	\left\{\begin{array}{llll}
		\bbZ/4\bbZ&=&\{0,\epsilon_1,w_n,w_{n-1}\vert\, 2\epsilon_1=0,\,2w_n=\epsilon_1, \,3w_n=w_{n-1}\},& \text{if } n \: \text{ is odd}, \\
		 \bbZ/2\bbZ\times\bbZ/2\bbZ&=&\{0,\epsilon_1,w_n,w_{n-1}\vert\, 2\epsilon_1=0,\,w_n+w_{n-1}=\epsilon_1\},& \text{if } n \: \text{ is even.} 
	\end{array}\right.
\end{equation*}
An almost-simple group $G$ of type $D_{n}$ is isomorphic to either the (simply-connected) spin group $\Spin_{2n}$, or the orthogonal group $\SO_{2n}$,   or the (adjoint) projective orthogonal group $\PSO_{2n}$ 
or, if $n$ is even, to one of the two almost-simple groups  
\begin{equation}\label{E:semispin}\Spin_{2n}/\langle w_{n-1}\rangle,\quad \Spin_{2n}/\langle w_{n}\rangle.
\end{equation}
Note that the character group $\op{Hom}(\mathscr Z(\Spin_{2n}),\mathbb G_m)$ is canonically isomorphic to $P(D_n)/Q(D_n)$. 

We first assume $n\neq 4$. By \cite[Chap.VI, \S 4.8]{BuLie}, the outer automorphism group $\Out(\Spin_{2n})$ is cyclic of order two. A generator $\gamma$ is given by an automorphism which acts on $P(D_n^{(\vee)})/Q(D_n^{(\vee)})$ by permuting $w_{n-1}$ and $w_n$. We recall that $\gamma$ descends to an outer automorphism of a quotient $\Spin_{2n}/\mu$, if it preserves the subgroup $\mu$. In particular, when $n$ is even, the isomorphism $\gamma$ does not descend to an isomorphism on both groups \eqref{E:semispin}. Furthermore, this automorphism identifies these groups, we denote their isomorphism class by $\op{SemiSpin}_{2n}$. Using the computations in \cite[\S 7.3]{Frin}, we deduce the table.
$$
{\small
	\setlength\arraycolsep{1.5pt}
\begin{array}{|c|c|c|c|c|}
	\hline
	\g&G&\Out(G)&\operatorname{Hom}(\mathscr Z(G),\mathbb G_m)&\pi_1(G)\\
	\hline
	D_{2l},& \operatorname{Spin}_{4l} &\mathbb Z/2\mathbb Z=\langle\gamma\rangle&(\mathbb Z/2\mathbb Z)^2=\langle w_{2l-1}\rangle\times \langle w_{2l}\rangle& \{0\}\\
	\scriptstyle{\text{s.t. }l\geq 3}& \operatorname{SemiSpin}_{4l} & \{1\} & \mathbb Z/2\mathbb Z& \mathbb Z/2\mathbb Z\\
	& \operatorname{SO}_{4l} &\mathbb Z/2\mathbb Z =\langle\gamma\rangle&\mathbb Z/2\mathbb Z &\mathbb Z/2\mathbb Z\\
		& \operatorname{PSO}_{4l} &\mathbb Z/2\mathbb Z=\langle\gamma\rangle & \{0\} & (\mathbb Z/2\mathbb Z)^2=\langle w_{2l-1}\rangle\times \langle w_{2l}\rangle\\
		\hline
	D_{2l+1},& \operatorname{Spin}_{4l+2} &\mathbb Z/2\bbZ=\langle\gamma\rangle& \mathbb Z/4\bbZ=\langle w_{2l}\rangle=\langle w_{2l+1}\rangle& \{0\}\\
	\scriptstyle{\text{s.t. }l\geq 2}& \operatorname{SO}_{4l+2} &\mathbb Z/2\bbZ =\langle\gamma\rangle&\mathbb Z/2\mathbb Z&\mathbb Z/2\mathbb Z\\
	& \operatorname{PSO}_{4l+2} &\mathbb Z/2\mathbb Z=\langle\gamma\rangle & \{0\} & \mathbb Z/4\mathbb Z=\langle w_{2l+1}\rangle=\langle w_{2l}\rangle\\
	\hline 
\end{array}}
$$
We remark that $\Out(G)$ acts trivially on those groups equal to $\mathbb Z/2\bbZ$, by permuting the coordinates to those ones equal to $(\mathbb Z/2\bbZ)^2$ and by taking the inverse in those ones equal to $\bbZ/4\bbZ$.

We now focus on the case $D_4$. By \cite[Chap.VI, \S 4.8]{BuLie}, the outer automorphism of $\Out(\Spin_{8})$ is isomorphic to $S_3$ and it acts on the (co)weight lattice $P(D_4^{(\vee)})$ by permuting the elements $\{\epsilon_1,w_3,w_4\}$. In particular, the almost-simple groups
$$
\op{SemiSpin}_{8}^+=\Spin_{8}/\langle w_{3}\rangle,\quad \op{SO}_8=\Spin_{8}/\langle \epsilon_1\rangle,\quad \op{SemiSpin}_{8}^-=\Spin_{8}/\langle w_{4}\rangle,
$$
are isomorphic via an outer automorphism. So, we may identify the semispin groups with $\SO_8$. Note that the automorphism $\gamma$ (defined for the cases $n\neq 4$) is contained in $S_3$: it corresponds to the cycle permuting the set $\{w_3,w_4\}$ and fixing $\epsilon_1$. Using the computations in \cite[\S 7.3]{Frin}, we get the following table.
$$
{\small
	\setlength\arraycolsep{1.5pt}\begin{array}{|c|c|c|c|c|}
	\hline
	\g&G&\Out(G)&\operatorname{Hom}(\mathscr Z(G),\mathbb G_m)&\pi_1(G)\\
	\hline
	D_4 & \operatorname{Spin}_8 & S_3=\operatorname{GL}_2(\mathbb Z/2\mathbb Z) & (\mathbb Z/2\mathbb Z)^2=\langle w_3\rangle\times\langle w_4\rangle & \{0\}\\
	& \operatorname{SO}_{8} &\mathbb Z/2\mathbb Z=\langle\gamma\rangle&\mathbb Z/2\mathbb Z &\mathbb Z/2\mathbb Z\\
	& \operatorname{PSO}_{8} &S_3=\operatorname{GL}_2(\mathbb Z/2\mathbb Z) & \{0\} & (\mathbb Z/2\mathbb Z)^2=\langle w_3\rangle\times\langle w_4\rangle\\
\hline
\end{array}}
$$
We remark that $\Out(G)$ acts trivially on those groups equal to $\bbZ/2\bbZ$ and it acts as $S_3=\op{GL}_2(\bbZ/2\bbZ)$ on those groups equal to $(\bbZ/2\bbZ)^2$. Using these tables, the next theorem follows immediately.
\begin{teo}Let\label{T:Dn} $G$ be an almost-simple group of type $D_n$ ($n\geq 4$).
		\begin{enumerate}[(i)]
			\item Type $D_4$. The isomorphism of Theorem \ref{T:A} gives an identification between the auomotphism group $\Aut(\smt^{\delta}_G(C))$ and the group
			$$
			\arraycolsep=0pt
			\left\{\begin{array}{lll}
				\left(\Pic(C)[2]\right)^2\rtimes(S_3\times\Aut(C)), &\text{ if }G=\op{Spin}_{8}&\text{ and }\delta\in\{0\},\\
				\Pic(C)[2]\rtimes (\bbZ/2\bbZ\times\Aut(C)), &\text{ if }G=\op{SO}_{8}&\text{ and }\delta\in\mathbb Z/2\mathbb Z,\\
				S_3\times\Aut(C), &\text{ if }G=\op{PSO}_8&\text{ and }\delta=(0,0)\in(\mathbb Z/2\mathbb Z)^2,\\
				\bbZ/2\bbZ\times\Aut(C), &\text{ if }G=\op{PSO}_8&\text{ and }\delta\neq (0,0)\in(\mathbb Z/2\mathbb Z)^2,
			\end{array}\right.
			$$
			where $\Aut(C)$ acts on $(\Pic(C)[2])^2$ and $\Pic(C)[2]$ by pull-back, $\bbZ/2\bbZ$ acts trivially on $\Pic(C)[2]$, and $S_3=\op{GL}_2(\mathbb Z/2\mathbb Z)$ acts on $(\Pic(C)[2])^2$ as follows
			$$
			\left(\begin{array}{cc}
				a&b\\
				c&d
			\end{array}\right).(L,M)=(L^a\otimes M^b,L^c\otimes M^d),
			$$
			with $a,b,c,d\in\bbZ/2\bbZ$ such that $ad-bc=1\in\bbZ/2\bbZ$.
		\item Type $D_{2l}$, $l\geq 3$.The isomorphism of Theorem \ref{T:A} gives an identification between the auomotphism group $\Aut(\smt^{\delta}_G(C))$ and the group
		$$
			\arraycolsep=0pt
		\left\{\begin{array}{lll}
			(\Pic(C)[2])^2\rtimes (\mathbb Z/2\mathbb Z\times \Aut(C)), &\text{ if }G=\op{Spin}_{4l}&\text{ and }\delta\in\{0\},\\
			\Pic(C)[2]\rtimes \Aut(C), &\text{ if }G=\op{SemiSpin}_{4l}&\text{ and }\delta\in\mathbb Z/2\mathbb Z,\\
			\Pic(C)[2]\rtimes (\bbZ/2\bbZ\times\Aut(C)), &\text{ if }G=\op{SO}_{4l}&\text{ and }\delta\in\mathbb Z/2\mathbb Z,\\
			\mathbb Z/2\mathbb Z\times \Aut(C), &\text{ if }G=\op{PSO}_{4l}&\text{ and }\delta=(0,0),(1,1)\in(\mathbb Z/2\mathbb Z)^2,\\
			\Aut(C), &\text{ if }G=\op{PSO}_{4l}&\text{ and }\delta= (1,0),(0,1)\in(\mathbb Z/2\mathbb Z)^2,
		\end{array}\right.
		$$
		where $\Aut(C)$ acts on $(\Pic(C)[2])^2$ and $\Pic(C)[2]$ by pull-back and $S_2=\mathbb Z/2\mathbb Z$ acts trivally on $\Pic(C)[2]$ and by permutation on $(\Pic(C)[2])^2$, i.e. $
		(12).(L,M)=(M,L).
		$
		\item Type $D_{2l+1}$, $l\geq 2$. The isomorphism of Theorem \ref{T:A} gives an identification between the auomotphism group $\Aut(\smt^{\delta}_G(C))$ and the group
		$$
	\arraycolsep=0pt
		\left\{\begin{array}{lll}
			\Pic(C)[4]\rtimes (\mathbb Z/2\mathbb Z\times \Aut(C)), &\text{ if }G=\op{Spin}_{4l+2}&\text{ and }\delta\in\{0\},\\
		\Pic(C)[2]\rtimes (\bbZ/2\bbZ\times\Aut(C)), &\text{ if }G=\op{SO}_{4l+2}&\text{ and }\delta\in\mathbb Z/2\mathbb Z\\
			\mathbb Z/2\mathbb Z\times \Aut(C), &\text{ if }G=\op{PSO}_{4l+2}&\text{ and }\delta=0,2\in\mathbb Z/4\mathbb Z\\
			\Aut(C), &\text{ if }G=\op{PSO}_{4l+2}&\text{ and }\delta=1,3\in\mathbb Z/4\mathbb Z
		\end{array}\right.
		$$
		where $\Aut(C)$ acts on $\Pic(C)[2]$ and $\Pic(C)[4]$ by pull-back and $\mathbb Z/2\mathbb Z$ acts trivally on $\Pic(C)[2]$ and by taking the dual on $(\Pic(C)[4])^2$, i.e. $
		[1].(L,M)=(L^{-1},M^{-1})=(L^3,M^3).
		$
		\end{enumerate}
	The same holds for the locus $\mt M_G$ of simple and stable bundles.
\end{teo}
\subsection{Types $E_6$}
Let us first recall some properties of the root system $E_6$. Consider the vector space $\mathbb R^8$ endowed with the standard scalar product $(-,-)$ and with the canonical bases $\{\epsilon_1,\ldots, \epsilon_8\}$.  
Inside $\mathbb R^8$, we will consider the following subvector space \begin{equation*}
	\begin{aligned}
		V(E_6):=\{\xi=(\xi_1,\ldots, \xi_8)\in \mathbb R^8\: : \xi_8=-\xi_6, \xi_7=\xi_6\}.
	\end{aligned}
\end{equation*}
We will freely identify $V(E_6)$ with its dual vector spaces by mean of the (restriction of the) standard scalar product $(-,-)$. 
The root (resp. coroot) lattices and  the weight (resp. coweight) lattices of $E_6$ are given by 
\begin{equation}\label{E:QP-E}
	\begin{aligned}
		Q(E_6)=Q(E_6^{\vee})=\left(\{\xi\in \bbZ^l \: : (\xi,\xi)\: \text{ is even}\} +\left\langle \frac{\sum_i \epsilon_i}{2}\right\rangle\right)\cap V(E_6)\subset\\
		\subset P(E_6)=P(E_6^\vee)=Q(E_6)+\left\langle w_1:=\frac{2}{3}(\epsilon_8-\epsilon_7-\epsilon_6)\right\rangle.
	\end{aligned}
\end{equation}
In particular, the group $P(E_6)/Q(E_6)$ is equal to $\mathbb Z/3\bbZ\cong\langle w_1\rangle$. An explicit basis of the lattices $Q(E_6)$ is given by 
\begin{equation*}
	\begin{aligned}
		& \alpha_1=\frac{\epsilon_1-\epsilon_2-\epsilon_3-\epsilon_4-\epsilon_5-\epsilon_6-\epsilon_7+\epsilon_8}{2}, \alpha_2=\epsilon_1+\epsilon_2, \\&\alpha_3=\epsilon_2-\epsilon_1, \alpha_4=\epsilon_3-\epsilon_2,
		\alpha_5=\epsilon_4-\epsilon_3, \alpha_6=\epsilon_5-\epsilon_4.
	\end{aligned}
\end{equation*}
An almost-simple group of type $E_6$ is either the simply-connected group $\mathbb E_6^{\operatorname{sc}}$ or the adjoint group $\mathbb E_6^{\ads}$. The character lattice $\op{Hom}(\scr Z(\mathbb E_6^{\op{sc}}),\mathbb G_m)$ is canonically isomorphic to $P(E_6)/Q(E_6)$. By \cite[Chap.VI, \S 4.12]{BuLie}, the group $\Out(\mathbb E_6^{\op{sc}})$ is cyclic of order two. A generator is given by an automorphism of $\mathbb E_6^{\op{sc}}$, which acts on the lattice $Q(E_6^{\vee})$ as $-\gamma$, where $\gamma$ is the (unique) element in the Weyl group $W_{\mathbb E^{\op{sc}}_6}$ which transforms $\alpha_1,\alpha_3,\alpha_4,\alpha_5,\alpha_6,\alpha_2$ to $-\alpha_6,-\alpha_5,-\alpha_4,-\alpha_3,-\alpha_1,-\alpha_2$, respectively. A direct computation shows that $-\gamma$ sends $w_1\op{mod}Q(E_6)$ to $-w_1\op{mod}Q(E_6)$. From \cite[Section 7.4]{Frin}, we deduce the following table.
$$
{\small
	\setlength\tabcolsep{1.5pt}
\begin{array}{|c|c|c|c|c|}
	\hline
	\g&G&\Out(G)&\operatorname{Hom}(\mathscr Z(G),\mathbb G_m)&\pi_1(G)\\
	\hline
	E_6 & \mathbb E_6^{\operatorname{sc}} & \mathbb Z/2\mathbb Z &\mathbb Z/3\mathbb Z=\langle w_1\rangle &\{0\}\\
	& \mathbb{E}_6^{\ads} & \mathbb Z/2\mathbb Z & \{0\} &\mathbb Z/3\mathbb Z=\langle w_1\rangle\\
	\hline
\end{array}}
$$
We remark that $\Out(G)$ acts on the non-trivial groups by taking the inverse. Putting all together, we have proved the following result.

\begin{teo}\label{T:E6} Let $G$ be an almost-simple group of type $E_6$. The isomorphism of Theorem \ref{T:A} gives the following equalities
	$$
\arraycolsep=0pt	\Aut(\smt_G^{\delta}(C))=\left\{\begin{array}{lll}
\Pic(C)[3]\rtimes\left(\mathbb Z/2\mathbb Z\times \Aut(C)\right),&\text{ if }G=\mathbb{E}_6^{\operatorname{sc}}&\text{ and }\delta\in\{0\},\\
\mathbb Z/2\mathbb Z\times \Aut(C),&\text{ if }G=\mathbb{E}_6^{\operatorname{ad}}&\text{ and }\delta=0\in\mathbb Z/3\mathbb Z,\\
\Aut(C),&\text{ if }G=\mathbb{E}_6^{\operatorname{ad}}&\text{ and }\delta\neq 0\in\mathbb Z/3\mathbb Z,
	\end{array}\right.
	$$
where $\Aut(C)$ acts on $\Pic(C)[3]$ by pull-back and $\mathbb Z/2\mathbb Z$ acts on $\Pic(C)[3]$ by taking the dual, i.e. $[1].L=L^{-1}=L^2$. The same holds for the locus $\mt M_G$ of simple and stable bundles.
\end{teo}
\subsection{Types $E_7$, $E_8$, $F_4$ and $G_2$.} We follow \cite[\S7.4, \S7.5]{Frin} closely. An almost-simple group of type $E_7$ is either the simply-connected group $\mathbb E_7^{\operatorname{sc}}$ or the adjoint group $\mathbb E_7^{\ads}$. For the other cases, there exists a unique group for each type. We denote by $\mathbb{E}_8$, resp. $\mathbb{F}_4$, resp. $\mathbb{G}_2$, the unique (simply-connected) almost-simple group of type $E_8$, resp. $F_4$, resp. $G_2$. From \cite[Chap.VI, \S 4.9, \S4.10, \S4.11, \S4.13]{BuLie} and \cite[\S7.4, \S7.5]{Frin}, we deduce the following table.
$$
{\small
	\setlength\tabcolsep{1.5pt}
\begin{array}{|c|c|c|c|c|}
	\hline
	\g&G&\Out(G)&\operatorname{Hom}(\mathscr Z(G),\mathbb G_m)&\pi_1(G)\\
	\hline
	E_7 &\mathbb E_7^{\operatorname{sc}} & \{1\} & \mathbb Z/2\bbZ & \{0\}\\
	&\mathbb E_7^{\ads} & \{1\} & \{0\} & \mathbb Z/2\mathbb Z\\
	\hline 
	E_8 &\mathbb E_8 & \{1\} &\{0\}& \{0\}\\
	
	\hline 
	F_4 &\mathbb F_4 & \{1\} &\{0\}& \{0\}\\
	
	\hline 
	G_2 &\mathbb G_2 & \{1\} &\{0\}& \{0\}\\
	\hline
\end{array}
}
$$
From the table, we deduce the following theorem.
\begin{teo}\label{T:resto} Let $G$ be an almost-simple group of type $E_7$, $E_8$, $F_4$ and $G_2$. The isomorphism of Theorem \ref{T:A} gives the following equalities
	$$
\arraycolsep=0pt
	\Aut(\smt_G^{\delta}(C))=\left\{\begin{array}{lll}
		\Pic(C)[2]\rtimes\Aut(C),&\text{ if }G=\mathbb{E}_7^{\operatorname{sc}}&\text{ and }\delta\in\{0\},\\
		\Aut(C),&\text{ if }G=\mathbb{E}_7^{\operatorname{ad}}&\text{ and }\delta\in\mathbb Z/2\mathbb Z,\\
		\Aut(C), &\text{ if }G=\mathbb E_8,\mathbb F_4,\mathbb G_2&\text{ and }\delta\in\{0\},
	\end{array}\right.
	$$
	where $\Aut(C)$ acts on $\Pic(C)[2]$ by pull-back. The same holds for the locus $\mt M_G$ of simple and stable bundles.
\end{teo}

\vspace{1cm}
\noindent {\bf Acknowledgments.} We are grateful to the anonymous referee for pointing out a crucial mistake in a earlier version of the paper. We thank Eduardo Esteves, Andres Fernandez Herrero, Inder Kaur, Martina Lanini, Antonio Rapagnetta
e Filippo Viviani for helpful conversations. The author acknowledges PRIN2017 CUP E84I19000500006,
and the MIUR Excellence Department Project awarded to the Department of Mathematics, University of Rome Tor Vergata, CUP E83C18000100006. 

\bibliographystyle{abbrv}
\bibliography{BiblioTorelli}
\end{document}